\documentclass[12pt]{amsart}
\usepackage{amsmath,amssymb,url,color,tikz,colortbl}
\usetikzlibrary{patterns}
\usepackage[all]{xy}
\usepackage{young}
\usepackage{enumerate}
\usepackage[margin=1in]{geometry}
\usepackage{hyperref}
\usepackage{amsfonts}
\usepackage{amsmath}
\usepackage{amssymb}
\usepackage{amsthm}
\usepackage{xcolor}
\usepackage{graphicx}
\usepackage{listings}
\usepackage{tram}
\usepackage{manfnt}
\usepackage{url,color,tikz,euscript}

\usetikzlibrary{arrows,automata}
\usepackage[colorinlistoftodos,color=red!70]{todonotes}

\theoremstyle{theorem}
\newtheorem{thm}{Theorem}

\newtheorem{lemma}[thm]{Lemma}
\newtheorem{prop}[thm]{Proposition}

\newtheorem{cor}[thm]{Corollary}

\theoremstyle{definition}
\newtheorem{defn}[thm]{Definition}
\newtheorem{definition}[thm]{Definition}
\newtheorem{notn}[thm]{Notation}

\newtheorem{ex}[thm]{Example}
\newtheorem{example}[thm]{Example}

\newtheorem{remark}[thm]{Remark}

\numberwithin{thm}{section}
\catcode`@=11
\def\m@th{\mathsurround\z@}
\def\cases#1{\left\{\,\vcenter{\normalbaselines\m@th
    \ialign{$##\hfil$&\quad##\hfil\crcr#1\crcr}}\right.}
\def\hang{\hangindent 24pt}
\def\d@nger{\medbreak\begingroup\clubpenalty=10000
  \def\par{\endgraf\endgroup\medbreak} %
  \noindent\hang\hangafter=-2
  \hbox to0pt{\hskip-\hangindent\dbend\hfill}}
\outer\def\danger{\d@nger}
\newcommand{\rr}{\mathbb{R}}
\newcommand{\zz}{\mathbb{Z}}

\newcommand{\nn}{\mathbb{N}}

\newcommand{\kk}{\mathbb{K}}
\renewcommand{\ss}{\mathfrak{S}}
\renewcommand{\aa}{\mathfrak{A}}

\newcommand{\comp}{\operatorname{comp}}
\newcommand{\tog}{\operatorname{Tog}}

\newcommand{\row}{\operatorname{Row}}
\newcommand{\gyr}{\operatorname{Gyr}}

\newcommand{\mc}{\operatorname{MC}}

\newcommand{\rk}{\operatorname{rk}}

\newcommand{\ra}{\rightarrow}
\newcommand{\sm}{\setminus}
\definecolor{green}{HTML}{006600}
\definecolor{orange}{HTML}{FF6200}
\definecolor{purple}{HTML}{990099}
\definecolor{coral}{HTML}{FF7F50}
\definecolor{mahogany}{HTML}{C04000}
\definecolor{gold}{HTML}{DAA541}
\definecolor{chocolate}{HTML}{D5691E}

\catcode`@=12 % at signs are no longer letters
\renewcommand{\labelitemii}{\scriptsize\raise2pt\hbox{$\!{\blacktriangleright}$}} % second bullet style now triangles

\newcommand{\eso}{\EuScript{O}}

\newcommand{\cala}{\mathcal{A}}

\newcommand{\calc}{\mathcal{C}}

\newcommand{\calf}{\mathcal{F}}

\newcommand{\calj}{\mathcal{J}}

\newcommand{\call}{\mathcal{L}}

\newcommand{\cals}{\mathcal{S}}

\newcommand{\calz}{\mathcal{Z}}

\newcommand{\bfF}{\mathbf{F}}
\newcommand{\bfI}{\mathbf{I}}

\newcommand{\orb}{\mathbf{OR}}
\newcommand{\opb}{\mathbf{OP}}

\begin{document}

\title{Antichain toggling and rowmotion}
\author[Joseph]{Michael Joseph}
\address{Department of Technology and Mathematics, Dalton State College, 650 College Dr., Dalton, GA 30720, USA}
\email{mjjoseph@daltonstate.edu}
%\date{\today}

\subjclass[2010]{05E18}
\keywords{alternating group, antichain, chain polytope, graded poset, isomorphism, linear extension, order ideal, piecewise-linear toggle, poset, rowmotion, symmetric group, toggle group}

\begin{abstract}
In this paper, we analyze the toggle group on the set of antichains of a poset.  Toggle groups, generated by simple involutions, were first introduced by Cameron and Fon-Der-Flaass for order ideals of posets.  Recently Striker has motivated the study of toggle groups on general families of subsets, including antichains.  This paper expands on this work by examining the relationship between the toggle groups of antichains and order ideals, constructing an explicit isomorphism between the two groups (for a finite poset).
We also focus on the \textit{rowmotion} action on antichains of a poset that has been well-studied in dynamical algebraic combinatorics, describing it as the composition of antichain toggles.
We also describe a piecewise-linear analogue of toggling to Stanley's chain polytope.  We examine the connections with the piecewise-linear toggling Einstein and Propp introduced for order polytopes and prove that almost all of our results for antichain toggles extend to the piecewise-linear setting.
%\vspace{3 ex}

%\noindent\textbf{Keywords: }TBA
\end{abstract}

\maketitle
\vspace{-0.333 in}
\tableofcontents

\section{Introduction}\label{sec:intro}
In~\cite{cameronfonder}, Cameron and Fon-Der-Flaass defined a group (now called the \textbf{toggle group}) consisting of permutations on the set $\calj(P)$ of order ideals of a poset $P$.  This group is generated by $\#P$ simple maps called \textbf{toggles}
each of which correspond to an element of the poset.  The toggle corresponding to $e\in P$ adds or removes $e$ from the order ideal if the resulting set is still an order ideal, and otherwise does nothing.
While each individual toggle has order 2, the composition of toggles can mix up $\calj(P)$ in a way that is difficult to describe in general.
In fact, Cameron and Fon-Der-Flaass proved that on any finite connected poset, the toggle group is either the symmetric or alternating group on $\calj(P)$.

More recently, Striker has noted that there is nothing significant about order ideals of a poset in the definition of the toggle group.  For any sets $E$ and $\call\subseteq 2^E$, we can define a toggle group corresponding to $\call$.  Striker has studied the behavior of the toggle group on various sets of combinatorial interest, including many subsets of posets: chains, antichains, and interval-closed sets~\cite{strikergentog}.

In Section~\ref{sec:combinatorial}, we analyze the toggle group for the set $\cala(P)$ of antichains of a finite poset $P$; this set is in bijection with the set of order ideals of $P$.  Striker proved that like the classical toggle group on order ideals, the antichain toggle group of a finite connected poset is always either the symmetric or alternating group on $\cala(P)$.  We take this work further and describe the relation between antichain toggles and order ideal toggles (Theorems~\ref{thm:t-star} and~\ref{thm:tau-star}).  In particular, we obtain an explicit isomorphism between the toggle groups of antichains and order ideals of $P$.

Throughout the paper, we also focus on a map first studied by Brouwer and Schrijver~\cite{brouwer1974period} as a map on antichains.  It is named \textbf{rowmotion} in~\cite{strikerwilliams}, though it has various names in the literature.  Rowmotion can be defined as a map on order ideals, order filters, or antichains, as it is the composition of three maps between these sets.

For specific posets, rowmotion has been shown to exhibit nice behavior, which is why it has been of significant interest.  In general, the order of rowmotion is unpredictable, but for many posets it is known to be small.
Also, rowmotion has been shown to exhibit
various phenomena recently introduced under the heading dynamical algebraic combinatorics.
One of these is the \textbf{homomesy} phenomenon, introduced by Propp and Roby in~\cite{propproby}, in which a statistic on a set (e.g.\ cardinality) has the same average across every orbit.  In fact, one of the earliest examples of homomesy is the conjecture of Panyushev~\cite{panyushev} proven by Armstrong, Stump, and Thomas~\cite{ast} that cardinality is homomesic under antichain rowmotion on positive root posets of Weyl groups.  Striker proved a ``toggleability'' statistic to be homomesic under rowmotion on any finite poset~\cite{strikerRS}.
Other homomesic statistics have been discovered on many posets, including on products of chains, minuscule posets, and zigzag posets~\cite{propproby,vorland-3-chains,robydac,shahrzad,rush-wang-homomesy-minuscule,indepsetspaper}.
Other phenomena discovered for rowmotion on various posets include Reiner, Stanton, and White's \textbf{cyclic sieving phenomenon}~\cite{csp,cspsagan,csp-brief} and Dilks, Pechenik, and Striker's \textbf{resonance} phenomenon~\cite{dpsresonance}.

%Another phenomenon discovered for rowmotion on certain posets is the cyclic sieving phenomenon, introduced by Reiner, Stanton, and White~\cite{csp,cspsagan}, in which the number of elements of a set $X$ fixed under an action is counted by plugging roots of unity into a generating function for $X$. Armstrong, Stump, and Thomas proved an instance of this for rowmotion on positive root posets by constructing an equivariant bijection between rowmotion on the poset and a well-studied action called Kreweras complementation on noncrossing partitions~\cite{ast}.
%Recently, Dilks, Pechenik, and Striker have introduced another phenomenon called resonance, which can be used to explain a tendency for orbit sizes to be a multiple of the order of a related map.  They proved an instance of resonance for rowmotion on products of three chains\cite{dps-resonance}.

Cameron and Fon-Der-Flaass showed that rowmotion on $\calj(P)$ can also be expressed as the composition of every toggle, each used exactly once, in an order specified by a linear extension~\cite{cameronfonder}.
Having multiple ways to express rowmotion has proven to be fruitful in studying the action on various posets; for this reason rowmotion has received far more attention as a map on order ideals as opposed to antichains.
In Subsection~\ref{subsec:anti-tog}, we show that antichain rowmotion can also be expressed as the composition of every toggle, each used exactly once, in a specified order (Proposition~\ref{prop:row-toggles-anti}).
This gives another tool to studying rowmotion.
In~\cite{indepsetspaper}, Roby and the author proved results for rowmotion on zigzag posets by first analyzing toggles for independent sets of path graphs (which are the antichains of zigzag posets in disguise) and then translating them back to the language of order ideals.

In Subsection~\ref{subsec:gp}, we discuss antichain toggles on graded posets.  As has already been studied for order ideals~\cite{strikerwilliams}, we can apply antichain toggles for an entire rank at once in a graded poset.  We detail the relation between rank toggles for order ideals and antichains.
Furthermore, we delve into a natural analogue of gyration to the toggle group of antichains.  Gyration is an action defined by Striker~\cite{strikerRS} within the toggle group of a graded poset,
named for its connection to Wieland's gyration
on alternating sign matrices~\cite{wieland}.

In Section~\ref{sec:cpl} we explore a generalization to the piecewise-linear setting.  There we define toggles as continuous maps on the chain polytope of a poset, defined by Stanley~\cite{Sta86}. These correspond to antichain toggles when restricted to the vertices.  This follows work of Einstein and Propp~\cite{einpropp} who generalized the notion of toggles from order ideals to the order polytope
of a poset, also defined by Stanley~\cite{Sta86}.
Surprisingly, many properties of rowmotion on order ideals also extend to the order polytope, and we show here that the same is true between antichain toggles and chain polytope toggles.
The main results of this section are Theorems~\ref{thm:iso-cpl} and~\ref{thm:row-C}.

As one would likely expect, some properties of antichain toggles extend to
the chain polytope while others do not.  In
Subsection~\ref{subsec:zigzag}, we give concrete examples as we consider chain polytope toggles on zigzag posets.
We demonstrate that while the main homomesy result of the author and Roby on toggling antichains of zigzag posets~\cite{indepsetspaper} does not
extend to the chain polytope, a different homomesy result does extend.
Despite numerous homomesy results in the literature for finite orbits,
Theorem~\ref{thm:indep-sets-C} is one of the few known results of an asymptotic generalization to orbits that are probably not always finite.

Our new results are in Subsections~\ref{subsec:anti-tog},~\ref{subsec:gp},~\ref{subsec:tog-poly}, and~\ref{subsec:zigzag}.  Some directions for future research are discussed in Section~\ref{sec:future tense}.
The other sections detail the necessary background material and framework as well as the notation we use, much of which varies between sources.

\section{Toggle groups for order ideals and antichains}\label{sec:combinatorial}
\subsection{Poset terminology and notation}
We assume the reader is familiar with elementary poset theory.  Though we very minimally introduce and define the terms and notation used in the paper, any reader
unfamiliar with posets should visit Stanley's text for a thorough introduction
~\cite[Ch.~3]{ec1ed2}.

\begin{defn}
A \textbf{partially ordered set} (or \textbf{poset} for short) is a set $P$ together with a binary relation `$\leq$' on $P$ that is reflexive, antisymmetric, and transitive.
\end{defn}

We use the notation $x\geq y$ to mean $y\leq x$, $x<y$ to mean ``$x\leq y$ and $x\not=y$,'' and $x>y$ to mean ``$x\geq y$ and $x\not=y$.''

Throughout this paper, let $P$ denote a finite poset.

\begin{definition}
For $x,y\in P$, we say that $x$ is \textbf{covered} by $y$ (or equivalently $y$ \textbf{covers} $x$), denoted $x\lessdot y$,
if $x<y$ and there does not exist $z$ in $P$ with $x<z<y$. The notation $x\gtrdot y$ means that $y$ is covered by $x$.  If either $x\leq y$ or $y\leq x$, we say $x$ and $y$ are \textbf{comparable}.
Otherwise, $x$ and $y$ are \textbf{incomparable}, denoted $x\parallel y$.
\end{definition}

For a finite poset, all relations can be formed by the cover relations and transitivity.
We depict such posets by their \textbf{Hasse diagrams}, where each cover relation $x\lessdot y$ is represented by placing $y$ above $x$ and connecting $x$ and $y$ with an edge.

\subsection{Order ideals, antichains, and rowmotion}
In this subsection, we discuss an action that was first studied by Brouwer and Schrijver~\cite{brouwer1974period} and more recently by many others, particularly in~\cite{cameronfonder, panyushev, strikerwilliams, propproby,robydac}.  This action has several names in the literature; we use the name ``rowmotion'' due to Striker and Williams~\cite{strikerwilliams}.

\begin{defn} \hspace{-8 in}.
\begin{itemize}
\item An \textbf{order ideal} (resp.\ \textbf{order filter}) of $P$ is a subset $I\subseteq P$ such that if $x\in I$ and $y<x$ (resp.\ $y>x$) in $P$, then $y\in I$.  We denote the sets of order ideals and order filters of $P$ as $\calj(P)$ and $\calf(P)$ respectively.
\item An \textbf{antichain} (resp.\ \textbf{chain}) of $P$ is a subset $S\subseteq P$ in which any two elements are incomparable (resp.\ comparable).  The set of antichains of $P$ is denoted $\cala(P)$.
\item For a subset $S\subseteq P$, an element $x\in S$ is a \textbf{maximal} (resp.\ \textbf{minimal}) element of $S$ if $S$ does not contain any $y>x$ (resp.\ $y<x$).
\end{itemize}
\end{defn}

Complementation is a natural bijection between $\calj(P)$ and $\calf(P)$.  Let $\comp(S)$ denote the complement of a subset $S\subseteq P$.  Also, any order ideal (resp.\ filter) is uniquely determined by its set of maximal (resp.\ minimal) elements, which is an antichain.  Any antichain $S$ of $P$ generates an order ideal
$\bfI(S) := \{x\in P\;|\; x\leq y, y\in S\}$ whose set of maximal elements is $S$ and an order filter $\bfF(S) := \{x\in P\;|\; x\geq y, y\in S\}$ whose set of minimal elements is $S$.  This gives natural bijections $\bfI:\cala(P)\ra \calj(P)$ and $\bfF:\cala(P)\ra \calf(P)$.

For an antichain $S\in\cala(P)$, we call $\bfI(A)$ the \textbf{order ideal generated by \textsl{A}}, and $\bfF(A)$ the \textbf{order filter generated by \textsl{A}}.

We compose the bijections from above to obtain maps from one of $\calj(P)$, $\cala(P)$, or $\calf(P)$ into itself.

\begin{defn}
For an antichain $A\in \cala(P)$, define $\row_\cala(A)$ to be the set of minimal elements of the complement of the order ideal generated by $A$.
For an order ideal $I\in \calj(P)$, define $\row_\calj(I)$ to be the order ideal generated by the minimal elements of the complement of $I$.
For an order filter $F\in \calf(P)$, define $\row_\calf(F)$ to be the order filter generated by the maximal elements of the complement of $F$.

These maps can each be expressed as the composition of three maps as follows.
\begin{align*}
\row_\cala &: \cala(P) \stackrel{\bfI}{\longrightarrow} \calj(P) \stackrel{\comp}{\longrightarrow} \calf(P) \stackrel{\bfF^{-1}}{\longrightarrow} \cala(P)\\
\row_\calj &: \calj(P) \stackrel{\comp}{\longrightarrow} \calf(P) \stackrel{\bfF^{-1}}{\longrightarrow} \cala(P) \stackrel{\bfI}{\longrightarrow} \calj(P)\\
\row_\calf &: \calf(P) \stackrel{\comp}{\longrightarrow} \calj(P) \stackrel{\bfI^{-1}}{\longrightarrow} \cala(P) \stackrel{\bfF}{\longrightarrow} \calf(P)
\end{align*}
\end{defn}

These bijections are all called \textbf{rowmotion}.  We will focus primarily on $\row_\cala$ and $\row_\calj$ (since $\row_\calf:\calf(P)\ra \calf(P)$ is equivalent to $\row_\calj$ for the dual poset that swaps the `$\leq$' and `$\geq$' relations).  There is a correspondence between the orbits under these two maps; each $\row_\cala$-orbit $\eso$ has a corresponding $\row_\calj$-orbit consisting of the order ideals generated by the antichains in $\eso$, and vice versa.  The following commutative diagram depicts this relation.

\begin{center}
\begin{tikzpicture}
\node at (0,1.8) {$\cala(P)$};
\node at (0,0) {$\calj(P)$};
\node at (3.25,1.8) {$\cala(P)$};
\node at (3.25,0) {$\calj(P)$};
\draw[semithick, ->] (0,1.3) -- (0,0.5);
\node[left] at (0,0.9) {$\bfI$};
\draw[semithick, ->] (0.7,0) -- (2.5,0);
\node[below] at (1.5,0) {$\row_\calj$};
\draw[semithick, ->] (0.7,1.8) -- (2.5,1.8);
\node[above] at (1.5,1.8) {$\row_\cala$};
\draw[semithick, ->] (3.25,1.3) -- (3.25,0.5);
\node[right] at (3.25,0.9) {$\bfI$};
\end{tikzpicture}
\end{center}

\begin{ex}\label{ex:a3-row}
Consider the following poset $P$ (which is the positive root poset $\Phi^+(A_3)$).

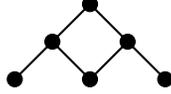
\begin{center}
\begin{tikzpicture}[scale=0.5]
\draw[thick] (-0.1, 1.9) -- (-0.9, 1.1);
\draw[thick] (0.1, 1.9) -- (0.9, 1.1);
\draw[thick] (-1.1, 0.9) -- (-1.9, 0.1);
\draw[thick] (-0.9, 0.9) -- (-0.1, 0.1);
\draw[thick] (0.9, 0.9) -- (0.1, 0.1);
\draw[thick] (1.1, 0.9) -- (1.9, 0.1);
\draw[fill] (0,2) circle [radius=0.2];
\draw[fill] (-1,1) circle [radius=0.2];
\draw[fill] (1,1) circle [radius=0.2];
\draw[fill] (-2,0) circle [radius=0.2];
\draw[fill] (0,0) circle [radius=0.2];
\draw[fill] (2,0) circle [radius=0.2];
\end{tikzpicture}
\end{center}

Below we show an example of each of $\row_\cala$ acting on an antichain and $\row_\calj$ acting on an order ideal as their respective three-step processes.  In each, hollow circles represent elements of $P$ not in the antichain, order ideal, or order filter.  
Notice that the order ideal we start with is generated by the antichain we begin with.  After applying rowmotion to both, we get the order ideal generated by the antichain we obtain.

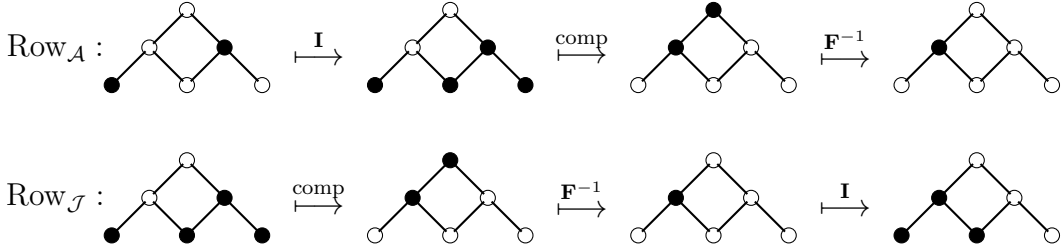
\begin{center}
\begin{tikzpicture}[scale=0.5]
\node at (-3.5,1) {$\row_\cala:$};
\node at (-3.5,-3) {$\row_\calj:$};
\begin{scope}
\draw[thick] (-0.1, 1.9) -- (-0.9, 1.1);
\draw[thick] (0.1, 1.9) -- (0.9, 1.1);
\draw[thick] (-1.1, 0.9) -- (-1.9, 0.1);
\draw[thick] (-0.9, 0.9) -- (-0.1, 0.1);
\draw[thick] (0.9, 0.9) -- (0.1, 0.1);
\draw[thick] (1.1, 0.9) -- (1.9, 0.1);
\draw (0,2) circle [radius=0.2];
\draw (-1,1) circle [radius=0.2];
\draw[fill] (1,1) circle [radius=0.2];
\draw[fill] (-2,0) circle [radius=0.2];
\draw (0,0) circle [radius=0.2];
\draw (2,0) circle [radius=0.2];
\end{scope}
\node at (3.5,1) {$\stackrel{\bfI}{\longmapsto}$};
\begin{scope}[shift={(7,0)}]
\draw[thick] (-0.1, 1.9) -- (-0.9, 1.1);
\draw[thick] (0.1, 1.9) -- (0.9, 1.1);
\draw[thick] (-1.1, 0.9) -- (-1.9, 0.1);
\draw[thick] (-0.9, 0.9) -- (-0.1, 0.1);
\draw[thick] (0.9, 0.9) -- (0.1, 0.1);
\draw[thick] (1.1, 0.9) -- (1.9, 0.1);
\draw (0,2) circle [radius=0.2];
\draw (-1,1) circle [radius=0.2];
\draw[fill] (1,1) circle [radius=0.2];
\draw[fill] (-2,0) circle [radius=0.2];
\draw[fill] (0,0) circle [radius=0.2];
\draw[fill] (2,0) circle [radius=0.2];
\end{scope}
\node at (10.5,1) {$\stackrel{\comp}{\longmapsto}$};
\begin{scope}[shift={(14,0)}]
\draw[thick] (-0.1, 1.9) -- (-0.9, 1.1);
\draw[thick] (0.1, 1.9) -- (0.9, 1.1);
\draw[thick] (-1.1, 0.9) -- (-1.9, 0.1);
\draw[thick] (-0.9, 0.9) -- (-0.1, 0.1);
\draw[thick] (0.9, 0.9) -- (0.1, 0.1);
\draw[thick] (1.1, 0.9) -- (1.9, 0.1);
\draw[fill] (0,2) circle [radius=0.2];
\draw[fill] (-1,1) circle [radius=0.2];
\draw (1,1) circle [radius=0.2];
\draw (-2,0) circle [radius=0.2];
\draw (0,0) circle [radius=0.2];
\draw (2,0) circle [radius=0.2];
\end{scope}
\node at (17.5,1) {$\stackrel{\bfF^{-1}}{\longmapsto}$};
\begin{scope}[shift={(21,0)}]
\draw[thick] (-0.1, 1.9) -- (-0.9, 1.1);
\draw[thick] (0.1, 1.9) -- (0.9, 1.1);
\draw[thick] (-1.1, 0.9) -- (-1.9, 0.1);
\draw[thick] (-0.9, 0.9) -- (-0.1, 0.1);
\draw[thick] (0.9, 0.9) -- (0.1, 0.1);
\draw[thick] (1.1, 0.9) -- (1.9, 0.1);
\draw (0,2) circle [radius=0.2];
\draw[fill] (-1,1) circle [radius=0.2];
\draw (1,1) circle [radius=0.2];
\draw (-2,0) circle [radius=0.2];
\draw (0,0) circle [radius=0.2];
\draw (2,0) circle [radius=0.2];
\end{scope}

\begin{scope}[shift={(0,-4)}]
\draw[thick] (-0.1, 1.9) -- (-0.9, 1.1);
\draw[thick] (0.1, 1.9) -- (0.9, 1.1);
\draw[thick] (-1.1, 0.9) -- (-1.9, 0.1);
\draw[thick] (-0.9, 0.9) -- (-0.1, 0.1);
\draw[thick] (0.9, 0.9) -- (0.1, 0.1);
\draw[thick] (1.1, 0.9) -- (1.9, 0.1);
\draw (0,2) circle [radius=0.2];
\draw (-1,1) circle [radius=0.2];
\draw[fill] (1,1) circle [radius=0.2];
\draw[fill] (-2,0) circle [radius=0.2];
\draw[fill] (0,0) circle [radius=0.2];
\draw[fill] (2,0) circle [radius=0.2];
\end{scope}
\node at (3.5,-3) {$\stackrel{\comp}{\longmapsto}$};
\begin{scope}[shift={(7,-4)}]
\draw[thick] (-0.1, 1.9) -- (-0.9, 1.1);
\draw[thick] (0.1, 1.9) -- (0.9, 1.1);
\draw[thick] (-1.1, 0.9) -- (-1.9, 0.1);
\draw[thick] (-0.9, 0.9) -- (-0.1, 0.1);
\draw[thick] (0.9, 0.9) -- (0.1, 0.1);
\draw[thick] (1.1, 0.9) -- (1.9, 0.1);
\draw[fill] (0,2) circle [radius=0.2];
\draw[fill] (-1,1) circle [radius=0.2];
\draw (1,1) circle [radius=0.2];
\draw (-2,0) circle [radius=0.2];
\draw (0,0) circle [radius=0.2];
\draw (2,0) circle [radius=0.2];
\end{scope}
\node at (10.5,-3) {$\stackrel{\bfF^{-1}}{\longmapsto}$};
\begin{scope}[shift={(14,-4)}]
\draw[thick] (-0.1, 1.9) -- (-0.9, 1.1);
\draw[thick] (0.1, 1.9) -- (0.9, 1.1);
\draw[thick] (-1.1, 0.9) -- (-1.9, 0.1);
\draw[thick] (-0.9, 0.9) -- (-0.1, 0.1);
\draw[thick] (0.9, 0.9) -- (0.1, 0.1);
\draw[thick] (1.1, 0.9) -- (1.9, 0.1);
\draw (0,2) circle [radius=0.2];
\draw[fill] (-1,1) circle [radius=0.2];
\draw (1,1) circle [radius=0.2];
\draw (-2,0) circle [radius=0.2];
\draw (0,0) circle [radius=0.2];
\draw (2,0) circle [radius=0.2];
\end{scope}
\node at (17.5,-3) {$\stackrel{\bfI}{\longmapsto}$};
\begin{scope}[shift={(21,-4)}]
\draw[thick] (-0.1, 1.9) -- (-0.9, 1.1);
\draw[thick] (0.1, 1.9) -- (0.9, 1.1);
\draw[thick] (-1.1, 0.9) -- (-1.9, 0.1);
\draw[thick] (-0.9, 0.9) -- (-0.1, 0.1);
\draw[thick] (0.9, 0.9) -- (0.1, 0.1);
\draw[thick] (1.1, 0.9) -- (1.9, 0.1);
\draw (0,2) circle [radius=0.2];
\draw[fill] (-1,1) circle [radius=0.2];
\draw (1,1) circle [radius=0.2];
\draw[fill] (-2,0) circle [radius=0.2];
\draw[fill] (0,0) circle [radius=0.2];
\draw (2,0) circle [radius=0.2];
\end{scope}
\end{tikzpicture}
\end{center}
\end{ex}

\subsection{Toggle group of $\calj(P)$}

Cameron and Fon-Der-Flaass showed
that rowmotion on $\calj(P)$ can be expressed in terms of basic involutions called \textit{toggles}.
Before discussing our new results regarding antichain toggles in the later subsections, we cover some important well-known results about toggling order ideals.

\begin{defn}[\cite{cameronfonder}]\label{def:comb-t}
Let $e\in P$.  Then the \textbf{order ideal toggle} corresponding to $e$ is the map $t_e: \calj(P)\ra \calj(P)$
defined by
$$t_e(I)=\left\{\begin{array}{ll}
I\cup\{e\} &\text{if $e\not\in I$ and $I\cup\{e\}\in \calj(P)$,}\\
I\sm\{e\} &\text{if $e\in I$ and $I\sm\{e\}\in \calj(P)$,}\\
I &\text{otherwise.}
\end{array}\right.$$
We use the convention that a composition $f_1 f_2 \cdots f_k$ of maps (such as toggles) is performed right to left.  Let $\tog_\calj(P)$ denote the \textbf{toggle group} of $\calj(P)$, which is the group generated by the toggles $\{t_e\;|\; e\in P\}$.
\end{defn}

Informally, $t_e$ adds or removes $e$ from the given order ideal $I$ provided the result is also an order ideal, and otherwise does nothing.  The following is clearly an equivalent description of the toggle $t_e$ so we include it without proof.

\begin{prop}\label{prop:wild nile ride}
Let $I\in \calj(P)$ and $e\in P$.  Then
$$t_e(I)=\left\{\begin{array}{ll}
I\cup\{e\} &\text{if $e$ is a minimal element of $P\sm I$,}\\
I\sm\{e\} &\text{if $e$ is a maximal element of $I$,}\\
I &\text{otherwise.}
\end{array}\right.$$
\end{prop}

\begin{prop}[\cite{cameronfonder}]\label{prop:J-toggle-inv-commute}
Each toggle $t_x$ is an involution (i.e., $t_x^2$ is the identity).  Two order ideal toggles $t_x,t_y$ commute if and only if neither $x$ nor $y$ covers the other.
\end{prop}

\begin{proof}
Let $x,y\in P$.
It is clear from the definition of $t_x$ that for any $I\in \calj(P)$, applying $t_x$ twice gives $I$.  Thus, $t_x^2$ is the identity.  To show when $t_x,t_y$ commute, we consider four cases.

\textbf{Case 1: $x=y$.} Then $t_xt_y=t_xt_x=t_yt_x$.

\textbf{Case 2: $x\parallel y$.}  Then whether or not one of $x$ or $y$ can be in an order ideal has no effect on whether the other can so $t_xt_y=t_yt_x$.

\textbf{Case 3: $x<y$ or $y<x$ but neither one covers the other.}  Without loss of generality, assume $x<y$.  Since $y$ does not cover $x$, there exists $z\in P$ such that $x<z<y$.  Then $z$ must be in any order ideal containing $y$, but cannot be in any order ideal that does not contain $x$.  Thus, we cannot change whether or not $x$ is in an order ideal and then do the same for $y$, or vice versa, without changing the status of $z$.  So $t_xt_y=t_yt_x$.
%Thus, given any order ideal $I$, either $t_x(t_y(I))=t_y(t_x(I))=I$, or $t_x(t_y(I))=t_y(t_x(I))=I \Delta \{x\}$, or 
%$t_x(t_y(I))=t_y(t_x(I))=I \Delta \{y\}$ where $\Delta$ %denotes the symmetric difference operation.

\textbf{Case 4: either $x\lessdot y$ or $y\lessdot x$.}  Without loss of generality, assume $x\lessdot y$.  Let $I=\{z\in P\;|\; z<y\}$ which is an order ideal that has $x$ as a maximal element.  Then $t_x t_y (I)=I\cup\{y\}$ and $t_y t_x (I)=I\sm\{x\}$ so $t_xt_y\not=t_yt_x$.
\end{proof}

\begin{defn}
A sequence $(x_1,x_2,\dots,x_n)$ containing all of the elements of a finite poset $P$ exactly once is called a \textbf{linear extension} of $P$ if it is order-preserving, that is, whenever $x_i<x_j$ in $P$ then $i<j$.
\end{defn}

\begin{prop}[\cite{cameronfonder}]\label{prop:row-toggles}
Let $(x_1,x_2,\dots,x_n)$ be any linear extension of $P$.  Then $\row_\calj=t_{x_1} t_{x_2} \cdots t_{x_n}$.
\end{prop}

%\begin{proof}
%Let $I\in \calj(P)$ and $T=t_{x_1} t_{x_2} \cdots t_{x_n}$. 
%For any $x<y$, we apply $t_y$ before $t_x$ when applying $T$.  If $y$ is in
%$P\sm I$ but is not a minimal element of $P\sm I$, then there exists some
%$x\in P\sm I$ such that $x<y$.  Note that $x$ is not in the order ideal at the time $t_y$ is applied because $x\not\in I$ and the toggle $t_y$ comes
%before $t_x$.  So $t_y$ does not insert $y$ into the order ideal.
%
%Now suppose $y$ is a minimal element of $P\sm I$.  Then every element of $P$ less than $y$ is in $I$.  The toggles for these elements below $y$ have not
%been reached yet at the time we apply $t_y$, so they are still in the order ideal.  Thus, from Proposition~\ref{prop:wild nile ride}, $t_y$ inserts $y$ into the order ideal.
%
%Let $S$ be the antichain of minimal elements of $P\sm I$.  Consider $x\in I$.  If $x\in \bfI(S)$, the order ideal generated by $S$, then some $y\in S$ with $y>x$ is in the order ideal when $t_x$ is applied, so $x$ must remain in the order ideal.  On the other hand, if $x\not\in \bfI(S)$, then $x$ need not stay in the order ideal when $t_x$ is applied (as by the same logic, elements of $I$ that are greater than $x$ have been toggled out before $t_x$ is applied), so $t_x$ removes $x$ from the order ideal.
%
%Thus, $T(I)$ is simply $\bfI(S)$, the order ideal generated by the minimal elements of $P\sm I$.
%\end{proof}

This proposition describes that (for finite posets) $\row_\calj$ is the product of every toggle exactly once in an order determined by a linear extension.  This has been particularly useful in examining rowmotion on certain posets due to the simple nature in which individual toggles act.  Additionally, for the large class of ``rowed-and-columned'' posets, Striker and Williams prove that $\row_\calj$ is conjugate in $\tog_\calj(P)$ to an action called ``promotion'' named for its connection with Sch\"{u}tzenberger's promotion on linear extensions of posets~\cite{sch72,strikerwilliams}.  In fact, they show $\row_\calj$ is conjugate to a large family of generalized rowmotion and promotion maps defined in terms of rows and columns.  As a result, the orbit structure and the homomesic property of certain types of statistics are preserved between promotion and rowmotion, so one can often use either rowmotion or promotion to study the other.  This tactic has been utilized by, e.g., Propp and Roby~\cite{propproby} and Vorland~\cite{vorland-3-chains} in studying products of chain posets.

\begin{example}\label{ex:row-toggles}
For the poset of Example~\ref{ex:a3-row}, as labeled below, $(a,b,c,d,e,f)$ gives a linear extension.
We show the effect of applying $t_at_bt_ct_dt_et_f$ to the order ideal considered in Example~\ref{ex:a3-row}.
In each step, we indicate the element whose toggle we apply next in {\color{red}red}.
Notice that the outcome is the same order ideal we obtained by the three step process, demonstrating Proposition~\ref{prop:row-toggles}.

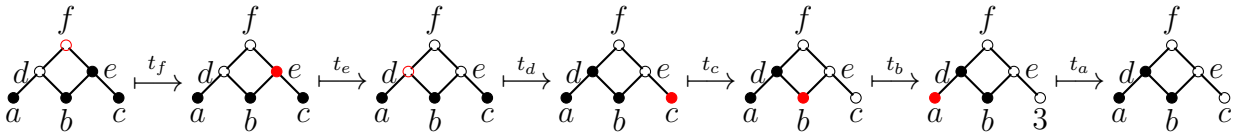
\begin{center}
\begin{tikzpicture}[scale=0.35]
\begin{scope}[shift={(0,-4)}]
\draw[thick] (-0.1, 1.9) -- (-0.9, 1.1);
\draw[thick] (0.1, 1.9) -- (0.9, 1.1);
\draw[thick] (-1.1, 0.9) -- (-1.9, 0.1);
\draw[thick] (-0.9, 0.9) -- (-0.1, 0.1);
\draw[thick] (0.9, 0.9) -- (0.1, 0.1);
\draw[thick] (1.1, 0.9) -- (1.9, 0.1);
\draw[red] (0,2) circle [radius=0.2];
\draw (-1,1) circle [radius=0.2];
\draw[fill] (1,1) circle [radius=0.2];
\draw[fill] (-2,0) circle [radius=0.2];
\draw[fill] (0,0) circle [radius=0.2];
\draw[fill] (2,0) circle [radius=0.2];
\node[above] at (0,2) {$f$};
\node[left] at (-1,1) {$d$};
\node[right] at (1,1) {$e$};
\node[below] at (-2,0) {$a$};
\node[below] at (0,0) {$b$};
\node[below] at (2,0) {$c$};
\end{scope}
\node at (3.5,-3) {$\stackrel{t_f}{\longmapsto}$};
\begin{scope}[shift={(7,-4)}]
\draw[thick] (-0.1, 1.9) -- (-0.9, 1.1);
\draw[thick] (0.1, 1.9) -- (0.9, 1.1);
\draw[thick] (-1.1, 0.9) -- (-1.9, 0.1);
\draw[thick] (-0.9, 0.9) -- (-0.1, 0.1);
\draw[thick] (0.9, 0.9) -- (0.1, 0.1);
\draw[thick] (1.1, 0.9) -- (1.9, 0.1);
\draw (0,2) circle [radius=0.2];
\draw (-1,1) circle [radius=0.2];
\draw[red,fill] (1,1) circle [radius=0.2];
\draw[fill] (-2,0) circle [radius=0.2];
\draw[fill] (0,0) circle [radius=0.2];
\draw[fill] (2,0) circle [radius=0.2];
\node[above] at (0,2) {$f$};
\node[left] at (-1,1) {$d$};
\node[right] at (1,1) {$e$};
\node[below] at (-2,0) {$a$};
\node[below] at (0,0) {$b$};
\node[below] at (2,0) {$c$};
\end{scope}
\node at (10.5,-3) {$\stackrel{t_e}{\longmapsto}$};
\begin{scope}[shift={(14,-4)}]
\draw[thick] (-0.1, 1.9) -- (-0.9, 1.1);
\draw[thick] (0.1, 1.9) -- (0.9, 1.1);
\draw[thick] (-1.1, 0.9) -- (-1.9, 0.1);
\draw[thick] (-0.9, 0.9) -- (-0.1, 0.1);
\draw[thick] (0.9, 0.9) -- (0.1, 0.1);
\draw[thick] (1.1, 0.9) -- (1.9, 0.1);
\draw (0,2) circle [radius=0.2];
\draw[red] (-1,1) circle [radius=0.2];
\draw (1,1) circle [radius=0.2];
\draw[fill] (-2,0) circle [radius=0.2];
\draw[fill] (0,0) circle [radius=0.2];
\draw[fill] (2,0) circle [radius=0.2];
\node[above] at (0,2) {$f$};
\node[left] at (-1,1) {$d$};
\node[right] at (1,1) {$e$};
\node[below] at (-2,0) {$a$};
\node[below] at (0,0) {$b$};
\node[below] at (2,0) {$c$};
\end{scope}
\node at (17.5,-3) {$\stackrel{t_d}{\longmapsto}$};
\begin{scope}[shift={(21,-4)}]
\draw[thick] (-0.1, 1.9) -- (-0.9, 1.1);
\draw[thick] (0.1, 1.9) -- (0.9, 1.1);
\draw[thick] (-1.1, 0.9) -- (-1.9, 0.1);
\draw[thick] (-0.9, 0.9) -- (-0.1, 0.1);
\draw[thick] (0.9, 0.9) -- (0.1, 0.1);
\draw[thick] (1.1, 0.9) -- (1.9, 0.1);
\draw (0,2) circle [radius=0.2];
\draw[fill] (-1,1) circle [radius=0.2];
\draw (1,1) circle [radius=0.2];
\draw[fill] (-2,0) circle [radius=0.2];
\draw[fill] (0,0) circle [radius=0.2];
\draw[red,fill] (2,0) circle [radius=0.2];
\node[above] at (0,2) {$f$};
\node[left] at (-1,1) {$d$};
\node[right] at (1,1) {$e$};
\node[below] at (-2,0) {$a$};
\node[below] at (0,0) {$b$};
\node[below] at (2,0) {$c$};
\end{scope}
\node at (24.5,-3) {$\stackrel{t_c}{\longmapsto}$};
\begin{scope}[shift={(28,-4)}]
\draw[thick] (-0.1, 1.9) -- (-0.9, 1.1);
\draw[thick] (0.1, 1.9) -- (0.9, 1.1);
\draw[thick] (-1.1, 0.9) -- (-1.9, 0.1);
\draw[thick] (-0.9, 0.9) -- (-0.1, 0.1);
\draw[thick] (0.9, 0.9) -- (0.1, 0.1);
\draw[thick] (1.1, 0.9) -- (1.9, 0.1);
\draw (0,2) circle [radius=0.2];
\draw[fill] (-1,1) circle [radius=0.2];
\draw (1,1) circle [radius=0.2];
\draw[fill] (-2,0) circle [radius=0.2];
\draw[red,fill] (0,0) circle [radius=0.2];
\draw (2,0) circle [radius=0.2];
\node[above] at (0,2) {$f$};
\node[left] at (-1,1) {$d$};
\node[right] at (1,1) {$e$};
\node[below] at (-2,0) {$a$};
\node[below] at (0,0) {$b$};
\node[below] at (2,0) {$c$};
\end{scope}
\node at (31.5,-3) {$\stackrel{t_b}{\longmapsto}$};
\begin{scope}[shift={(35,-4)}]
\draw[thick] (-0.1, 1.9) -- (-0.9, 1.1);
\draw[thick] (0.1, 1.9) -- (0.9, 1.1);
\draw[thick] (-1.1, 0.9) -- (-1.9, 0.1);
\draw[thick] (-0.9, 0.9) -- (-0.1, 0.1);
\draw[thick] (0.9, 0.9) -- (0.1, 0.1);
\draw[thick] (1.1, 0.9) -- (1.9, 0.1);
\draw (0,2) circle [radius=0.2];
\draw[fill] (-1,1) circle [radius=0.2];
\draw (1,1) circle [radius=0.2];
\draw[red,fill] (-2,0) circle [radius=0.2];
\draw[fill] (0,0) circle [radius=0.2];
\draw (2,0) circle [radius=0.2];
\node[above] at (0,2) {$f$};
\node[left] at (-1,1) {$d$};
\node[right] at (1,1) {$e$};
\node[below] at (-2,0) {$a$};
\node[below] at (0,0) {$b$};
\node[below] at (2,0) {3};
\end{scope}
\node at (38.5,-3) {$\stackrel{t_a}{\longmapsto}$};
\begin{scope}[shift={(42,-4)}]
\draw[thick] (-0.1, 1.9) -- (-0.9, 1.1);
\draw[thick] (0.1, 1.9) -- (0.9, 1.1);
\draw[thick] (-1.1, 0.9) -- (-1.9, 0.1);
\draw[thick] (-0.9, 0.9) -- (-0.1, 0.1);
\draw[thick] (0.9, 0.9) -- (0.1, 0.1);
\draw[thick] (1.1, 0.9) -- (1.9, 0.1);
\draw (0,2) circle [radius=0.2];
\draw[fill] (-1,1) circle [radius=0.2];
\draw (1,1) circle [radius=0.2];
\draw[fill] (-2,0) circle [radius=0.2];
\draw[fill] (0,0) circle [radius=0.2];
\draw (2,0) circle [radius=0.2];
\node[above] at (0,2) {$f$};
\node[left] at (-1,1) {$d$};
\node[right] at (1,1) {$e$};
\node[below] at (-2,0) {$a$};
\node[below] at (0,0) {$b$};
\node[below] at (2,0) {$c$};
\end{scope}
\end{tikzpicture}
\end{center}
\end{example}

\subsection{Toggle group of $\cala(P)$}\label{subsec:anti-tog}

While toggling order ideals has received by far the most attention
over the years since Cameron and Fon-Der-Flaass introduced the concept in 1995,
toggles can be defined for any family of subsets of a given set.  In~\cite{strikergentog}, Striker defines toggle groups for general families of subsets.  For a set $E$ and set of ``allowed subsets'' $\call\subseteq 2^E$, each $e\in E$ has a corresponding toggle map which adds or removes $e$ from any set in $\call$ provided the result is still in $\call$ and otherwise does nothing.  In $\tog_\calj(P)$, the set $E$ is the poset $P$, while the set $\call$ of allowed subsets is $\calj(P)$.

Homomesy and other nice behavior have been discovered for actions in generalized toggle groups
for noncrossing partitions~\cite{efgjmpr}
as well as for subsets of an $n$-element set whose cardinality ranges between $r$ and $n-r$~\cite{whirling}.
Also, Roby and the author prove results about rowmotion on zigzag posets by analyzing toggles on independent sets of path graphs~\cite{indepsetspaper}, which are the same as antichains of zigzag posets; see Remark~\ref{rem:indep-sets}.

In this section, we examine the antichain toggle group $\tog_\cala(P)$ where the set of allowed subsets is $\cala(P)$.  Cameron and Fon-Der-Flaass proved that for a finite connected poset $P$ (i.e., $P$ has a connected Hasse diagram), $\tog_\calj(P)$ is either the symmetric group $\ss_{\calj(P)}$ or alternating group $\aa_{\calj(P)}$ on $\calj(P)$~\cite[Theorem~4]{cameronfonder}.  Striker has analyzed antichain toggle groups in~\cite[\S3.3]{strikergentog}, where it is likewise proven that for a finite connected poset $P$, $\tog_\cala(P)$ is either the symmetric group $\ss_{\cala(P)}$ or alternating group $\aa_{\cala(P)}$ on $\cala(P)$.  We expand on this work and will construct an explicit isomorphism between $\tog_\calj(P)$ and $\tog_\cala(P)$, ruling out the possibility that for a given poset, one of these groups is a symmetric group with the other being an alternating group.

The other key result of this section is Proposition~\ref{prop:row-toggles-anti} that, for a finite poset $P$, $\row_\cala$ is the product of every antichain toggle, each used exactly once in an order given by a linear extension (but the opposite order from that of $\row_\calj$).  This provides another tool for analyzing rowmotion.
Although Brouwer and Schrijver originally considered rowmotion as a map on antichains,
rowmotion on order ideals has received far more attention due to its known description as a product of toggles.

\begin{defn}[\cite{strikergentog}]\label{def:comb-tau}
Let $e\in P$.  Then the \textbf{antichain toggle} corresponding to $e$ is the map $\tau_e: \cala(P)\ra \cala(P)$
defined by
$$\tau_e(A)=\left\{\begin{array}{ll}
A\cup\{e\} &\text{if $e\not\in A$ and $A\cup\{e\}\in\cala(P)$,}\\
A\sm\{e\} &\text{if $e\in A$,}\\
A &\text{otherwise.}
\end{array}\right.$$
Let $\tog_\cala(P)$ denote the \textbf{toggle group} of $\cala(P)$ generated by the toggles $\{\tau_e\;|\; e\in P\}$.
\end{defn}

We use $\tau_e$ for antichain toggles to distinguish them from the order ideal toggles $t_e$.
Unlike for order ideals, removing an element from an antichain always results in an antichain.  This is why we have simplified the definition above so the second case is slightly different from that of $t_e$.  For any $e$, the toggle $\tau_e$ is clearly an involution (as is any toggle defined using Striker's definition), using the same reasoning as for order ideal toggles.

\begin{prop}[{\cite[Lemma 3.12]{strikergentog}}]\label{prop:tau-commute}
Two antichain toggles $\tau_x,\tau_y$ commute if and only if $x=y$ or $x\parallel y$.
\end{prop}

Note from Propositions~\ref{prop:J-toggle-inv-commute}
and~\ref{prop:tau-commute} that antichain toggles commute less often than order ideal toggles.

\begin{proof}
Let $x,y\in P$.

\textbf{Case 1: $x=y$.} Then $\tau_x\tau_y=\tau_x\tau_x=\tau_y\tau_x$.

\textbf{Case 2: $x\parallel y$.}  Then whether $x$ is in an antichain has no effect on whether $y$ can be in that antichain and vice versa.
So $\tau_x \tau_y = \tau_y \tau_x$.

\textbf{Case 3: $x<y$ or $y<x$.}  Then $\O$, $\{x\}$, $\{y\}$ are all antichains of $P$, but not $\{x,y\}$.  In this scenario $t_xt_y(\O)=\{y\}$ but $t_yt_x(\O)=\{x\}$.
\end{proof}

\begin{defn}
For $e\in P$, let $e_1,\dots,e_k$ be the elements covered by $e$.  Define $t_e^*\in\tog_\cala(P)$ as $t_e^* := \tau_{e_1} \tau_{e_2}\cdots \tau_{e_k} \tau_e \tau_{e_1} \tau_{e_2}\cdots \tau_{e_k}$.  (If $e$ is a minimal element of $P$, then $k=0$ and so $t_e^*=\tau_e$.)
\end{defn}

Due to incomparability, all of the toggles $\tau_{e_1},\tau_{e_2},\dots, \tau_{e_k}$ commute with each other (but not with $\tau_e$).  Therefore, the definition of $t_e^*$ is well-defined and does not depend on the order of $e_1,e_2,\dots,e_k$.  For this reason, the toggles $\tau_{e_1},\tau_{e_2},\dots, \tau_{e_k}$ can be applied ``simultaneously,'' so $t_e^*$ is the conjugate of $\tau_e$ by the product of all antichain toggles for the elements covered by $e$.  As stated formally in the following theorem, applying $t_e^*$ to an antichain $A$ describes the effect that $t_e$ has on the order ideal $\bfI(A)$ generated by $A$.

\begin{thm}\label{thm:t-star}
Let $I\in \calj(P)$, $e\in P$, and $A=\bfI^{-1}(I)$ be the antichain of maximal elements of $I$.  Then the antichain $\bfI^{-1}(t_e(I))$ of maximal elements of $t_e(I)$ is $t_e^*(A)$.  That is, the following diagram commutes.

\begin{center}
\begin{tikzpicture}
\node at (0,1.8) {$\cala(P)$};
\node at (0,0) {$\calj(P)$};
\node at (3.25,1.8) {$\cala(P)$};
\node at (3.25,0) {$\calj(P)$};
\draw[semithick, ->] (0,1.3) -- (0,0.5);
\node[left] at (0,0.9) {$\bfI$};
\draw[semithick, ->] (0.7,0) -- (2.5,0);
\node[below] at (1.5,0) {$t_e$};
\draw[semithick, ->] (0.7,1.8) -- (2.5,1.8);
\node[above] at (1.5,1.8) {$t_e^*$};
\draw[semithick, ->] (3.25,1.3) -- (3.25,0.5);
\node[right] at (3.25,0.9) {$\bfI$};
\end{tikzpicture}
\end{center}
\end{thm}

We include a proof of Theorem~\ref{thm:t-star} now, but we will reprove  it later as a restriction of Theorem~\ref{thm:iso-cpl}.

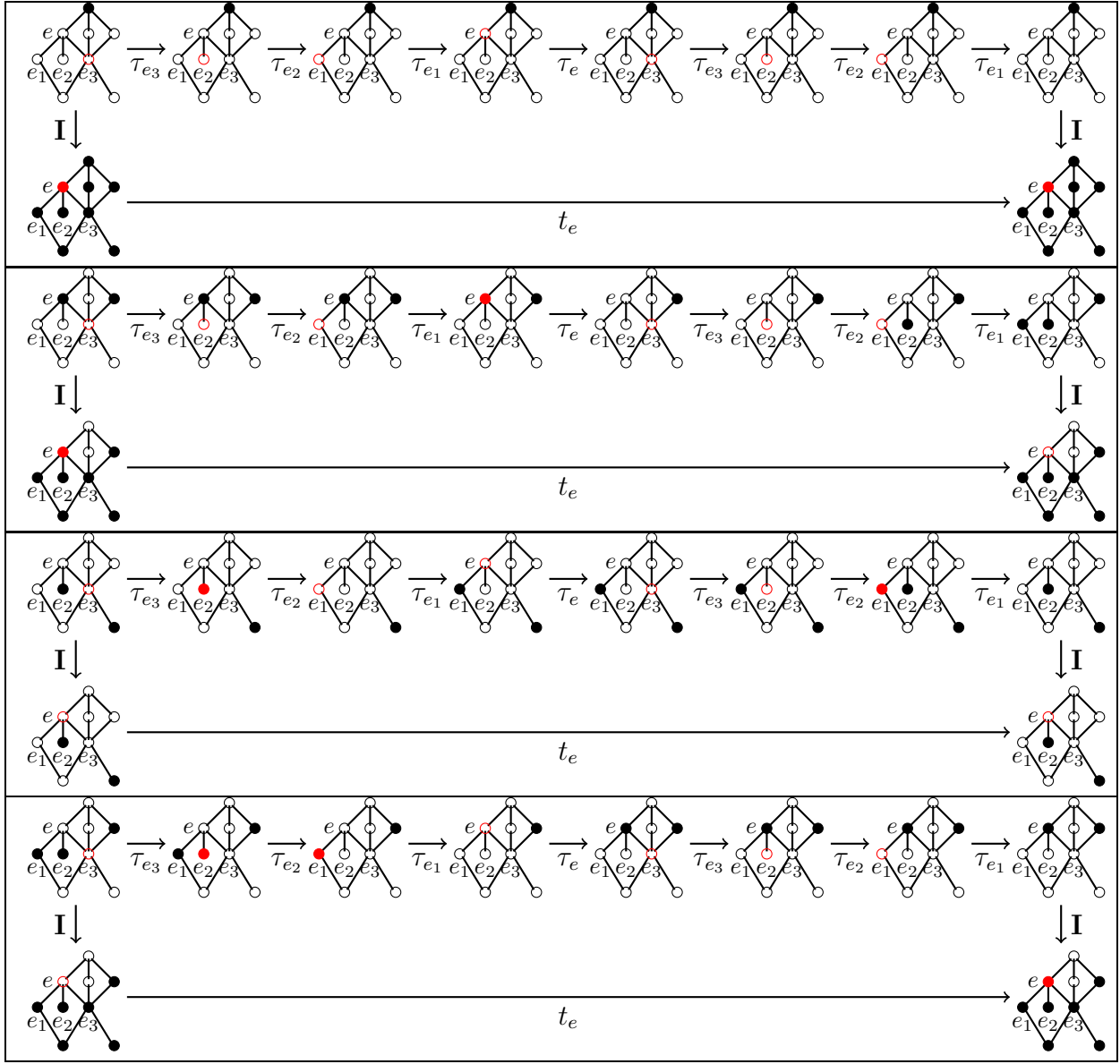
\begin{figure}
\begin{tabular}{|c|}
\hline
\begin{tikzpicture}[scale=0.38]
\draw[thick] (-0.1, 1.9) -- (-0.9, 1.1);
\draw[thick] (0, 0.1) -- (0,0.9);
\draw[thick] (0, 1.1) -- (0,1.9);
\draw[thick] (-1, 0.1) -- (-1,0.9);
\draw[thick] (0.1, 1.9) -- (0.9, 1.1);
\draw[thick] (-1.1, 0.9) -- (-1.9, 0.1);
\draw[thick] (-0.9, 0.9) -- (-0.1, 0.1);
\draw[thick] (0.9, 0.9) -- (0.1, 0.1);
\draw[thick] (-1.1, -1.4) -- (-1.9, -0.1);
\draw[thick] (0.9, -1.4) -- (0.1, -0.1);
\draw[thick] (-0.9, -1.4) -- (-0.1, -0.1);
\draw[fill] (0,2) circle [radius=0.2];
\draw (-1,1) circle [radius=0.2];
\draw (1,1) circle [radius=0.2];
\draw (0,1) circle [radius=0.2];
\draw (-2,0) circle [radius=0.2];
\draw (-1,0) circle [radius=0.2];
\draw[red] (0,0) circle [radius=0.2];
\draw (-1,-1.5) circle [radius=0.2];
\draw (1,-1.5) circle [radius=0.2];
\node[left] at (-1,1) {\footnotesize{$e$}};
\node[below] at (-2,0) {\footnotesize{$e_1$}};
\node[below] at (-1,0) {\footnotesize{$e_2$}};
\node[below] at (0,0) {\footnotesize{$e_3$}};
\draw[thick, ->] (1.5,0.4) -- (3,0.4);
\node[below] at (2.25,0.4) {$\tau_{e_3}$};
\begin{scope}[shift={(5.5,0)}]
\draw[thick] (-0.1, 1.9) -- (-0.9, 1.1);
\draw[thick] (0, 0.1) -- (0,0.9);
\draw[thick] (0, 1.1) -- (0,1.9);
\draw[thick] (-1, 0.1) -- (-1,0.9);
\draw[thick] (0.1, 1.9) -- (0.9, 1.1);
\draw[thick] (-1.1, 0.9) -- (-1.9, 0.1);
\draw[thick] (-0.9, 0.9) -- (-0.1, 0.1);
\draw[thick] (0.9, 0.9) -- (0.1, 0.1);
\draw[thick] (-1.1, -1.4) -- (-1.9, -0.1);
\draw[thick] (0.9, -1.4) -- (0.1, -0.1);
\draw[thick] (-0.9, -1.4) -- (-0.1, -0.1);
\draw[fill] (0,2) circle [radius=0.2];
\draw (-1,1) circle [radius=0.2];
\draw (1,1) circle [radius=0.2];
\draw (0,1) circle [radius=0.2];
\draw (-2,0) circle [radius=0.2];
\draw[red] (-1,0) circle [radius=0.2];
\draw (0,0) circle [radius=0.2];
\draw (-1,-1.5) circle [radius=0.2];
\draw (1,-1.5) circle [radius=0.2];
\node[left] at (-1,1) {\footnotesize{$e$}};
\node[below] at (-2,0) {\footnotesize{$e_1$}};
\node[below] at (-1,0) {\footnotesize{$e_2$}};
\node[below] at (0,0) {\footnotesize{$e_3$}};
\draw[thick, ->] (1.5,0.4) -- (3,0.4);
\node[below] at (2.25,0.4) {$\tau_{e_2}$};
\end{scope}
\begin{scope}[shift={(11,0)}]
\draw[thick] (-0.1, 1.9) -- (-0.9, 1.1);
\draw[thick] (0, 0.1) -- (0,0.9);
\draw[thick] (0, 1.1) -- (0,1.9);
\draw[thick] (-1, 0.1) -- (-1,0.9);
\draw[thick] (0.1, 1.9) -- (0.9, 1.1);
\draw[thick] (-1.1, 0.9) -- (-1.9, 0.1);
\draw[thick] (-0.9, 0.9) -- (-0.1, 0.1);
\draw[thick] (0.9, 0.9) -- (0.1, 0.1);
\draw[thick] (-1.1, -1.4) -- (-1.9, -0.1);
\draw[thick] (0.9, -1.4) -- (0.1, -0.1);
\draw[thick] (-0.9, -1.4) -- (-0.1, -0.1);
\draw[fill] (0,2) circle [radius=0.2];
\draw (-1,1) circle [radius=0.2];
\draw (1,1) circle [radius=0.2];
\draw (0,1) circle [radius=0.2];
\draw[red] (-2,0) circle [radius=0.2];
\draw (-1,0) circle [radius=0.2];
\draw (0,0) circle [radius=0.2];
\draw (-1,-1.5) circle [radius=0.2];
\draw (1,-1.5) circle [radius=0.2];
\node[left] at (-1,1) {\footnotesize{$e$}};
\node[below] at (-2,0) {\footnotesize{$e_1$}};
\node[below] at (-1,0) {\footnotesize{$e_2$}};
\node[below] at (0,0) {\footnotesize{$e_3$}};
\draw[thick, ->] (1.5,0.4) -- (3,0.4);
\node[below] at (2.25,0.4) {$\tau_{e_1}$};
\end{scope}
\begin{scope}[shift={(16.5,0)}]
\draw[thick] (-0.1, 1.9) -- (-0.9, 1.1);
\draw[thick] (0, 0.1) -- (0,0.9);
\draw[thick] (0, 1.1) -- (0,1.9);
\draw[thick] (-1, 0.1) -- (-1,0.9);
\draw[thick] (0.1, 1.9) -- (0.9, 1.1);
\draw[thick] (-1.1, 0.9) -- (-1.9, 0.1);
\draw[thick] (-0.9, 0.9) -- (-0.1, 0.1);
\draw[thick] (0.9, 0.9) -- (0.1, 0.1);
\draw[thick] (-1.1, -1.4) -- (-1.9, -0.1);
\draw[thick] (0.9, -1.4) -- (0.1, -0.1);
\draw[thick] (-0.9, -1.4) -- (-0.1, -0.1);
\draw[fill] (0,2) circle [radius=0.2];
\draw[red] (-1,1) circle [radius=0.2];
\draw (1,1) circle [radius=0.2];
\draw (0,1) circle [radius=0.2];
\draw (-2,0) circle [radius=0.2];
\draw (-1,0) circle [radius=0.2];
\draw (0,0) circle [radius=0.2];
\draw (-1,-1.5) circle [radius=0.2];
\draw (1,-1.5) circle [radius=0.2];
\node[left] at (-1,1) {\footnotesize{$e$}};
\node[below] at (-2,0) {\footnotesize{$e_1$}};
\node[below] at (-1,0) {\footnotesize{$e_2$}};
\node[below] at (0,0) {\footnotesize{$e_3$}};
\draw[thick, ->] (1.5,0.4) -- (3,0.4);
\node[below] at (2.25,0.4) {$\tau_{e}$};
\end{scope}
\begin{scope}[shift={(22,0)}]
\draw[thick] (-0.1, 1.9) -- (-0.9, 1.1);
\draw[thick] (0, 0.1) -- (0,0.9);
\draw[thick] (0, 1.1) -- (0,1.9);
\draw[thick] (-1, 0.1) -- (-1,0.9);
\draw[thick] (0.1, 1.9) -- (0.9, 1.1);
\draw[thick] (-1.1, 0.9) -- (-1.9, 0.1);
\draw[thick] (-0.9, 0.9) -- (-0.1, 0.1);
\draw[thick] (0.9, 0.9) -- (0.1, 0.1);
\draw[thick] (-1.1, -1.4) -- (-1.9, -0.1);
\draw[thick] (0.9, -1.4) -- (0.1, -0.1);
\draw[thick] (-0.9, -1.4) -- (-0.1, -0.1);
\draw[fill] (0,2) circle [radius=0.2];
\draw (-1,1) circle [radius=0.2];
\draw (1,1) circle [radius=0.2];
\draw (0,1) circle [radius=0.2];
\draw (-2,0) circle [radius=0.2];
\draw (-1,0) circle [radius=0.2];
\draw[red] (0,0) circle [radius=0.2];
\draw (-1,-1.5) circle [radius=0.2];
\draw (1,-1.5) circle [radius=0.2];
\node[left] at (-1,1) {\footnotesize{$e$}};
\node[below] at (-2,0) {\footnotesize{$e_1$}};
\node[below] at (-1,0) {\footnotesize{$e_2$}};
\node[below] at (0,0) {\footnotesize{$e_3$}};
\draw[thick, ->] (1.5,0.4) -- (3,0.4);
\node[below] at (2.25,0.4) {$\tau_{e_3}$};
\end{scope}
\begin{scope}[shift={(27.5,0)}]
\draw[thick] (-0.1, 1.9) -- (-0.9, 1.1);
\draw[thick] (0, 0.1) -- (0,0.9);
\draw[thick] (0, 1.1) -- (0,1.9);
\draw[thick] (-1, 0.1) -- (-1,0.9);
\draw[thick] (0.1, 1.9) -- (0.9, 1.1);
\draw[thick] (-1.1, 0.9) -- (-1.9, 0.1);
\draw[thick] (-0.9, 0.9) -- (-0.1, 0.1);
\draw[thick] (0.9, 0.9) -- (0.1, 0.1);
\draw[thick] (-1.1, -1.4) -- (-1.9, -0.1);
\draw[thick] (0.9, -1.4) -- (0.1, -0.1);
\draw[thick] (-0.9, -1.4) -- (-0.1, -0.1);
\draw[fill] (0,2) circle [radius=0.2];
\draw (-1,1) circle [radius=0.2];
\draw (1,1) circle [radius=0.2];
\draw (0,1) circle [radius=0.2];
\draw (-2,0) circle [radius=0.2];
\draw[red] (-1,0) circle [radius=0.2];
\draw (0,0) circle [radius=0.2];
\draw (-1,-1.5) circle [radius=0.2];
\draw (1,-1.5) circle [radius=0.2];
\node[left] at (-1,1) {\footnotesize{$e$}};
\node[below] at (-2,0) {\footnotesize{$e_1$}};
\node[below] at (-1,0) {\footnotesize{$e_2$}};
\node[below] at (0,0) {\footnotesize{$e_3$}};
\draw[thick, ->] (1.5,0.4) -- (3,0.4);
\node[below] at (2.25,0.4) {$\tau_{e_2}$};
\end{scope}
\begin{scope}[shift={(33,0)}]
\draw[thick] (-0.1, 1.9) -- (-0.9, 1.1);
\draw[thick] (0, 0.1) -- (0,0.9);
\draw[thick] (0, 1.1) -- (0,1.9);
\draw[thick] (-1, 0.1) -- (-1,0.9);
\draw[thick] (0.1, 1.9) -- (0.9, 1.1);
\draw[thick] (-1.1, 0.9) -- (-1.9, 0.1);
\draw[thick] (-0.9, 0.9) -- (-0.1, 0.1);
\draw[thick] (0.9, 0.9) -- (0.1, 0.1);
\draw[thick] (-1.1, -1.4) -- (-1.9, -0.1);
\draw[thick] (0.9, -1.4) -- (0.1, -0.1);
\draw[thick] (-0.9, -1.4) -- (-0.1, -0.1);
\draw[fill] (0,2) circle [radius=0.2];
\draw (-1,1) circle [radius=0.2];
\draw (1,1) circle [radius=0.2];
\draw (0,1) circle [radius=0.2];
\draw[red] (-2,0) circle [radius=0.2];
\draw (-1,0) circle [radius=0.2];
\draw (0,0) circle [radius=0.2];
\draw (-1,-1.5) circle [radius=0.2];
\draw (1,-1.5) circle [radius=0.2];
\node[left] at (-1,1) {\footnotesize{$e$}};
\node[below] at (-2,0) {\footnotesize{$e_1$}};
\node[below] at (-1,0) {\footnotesize{$e_2$}};
\node[below] at (0,0) {\footnotesize{$e_3$}};
\draw[thick, ->] (1.5,0.4) -- (3,0.4);
\node[below] at (2.25,0.4) {$\tau_{e_1}$};
\end{scope}
\begin{scope}[shift={(38.5,0)}]
\draw[thick] (-0.1, 1.9) -- (-0.9, 1.1);
\draw[thick] (0, 0.1) -- (0,0.9);
\draw[thick] (0, 1.1) -- (0,1.9);
\draw[thick] (-1, 0.1) -- (-1,0.9);
\draw[thick] (0.1, 1.9) -- (0.9, 1.1);
\draw[thick] (-1.1, 0.9) -- (-1.9, 0.1);
\draw[thick] (-0.9, 0.9) -- (-0.1, 0.1);
\draw[thick] (0.9, 0.9) -- (0.1, 0.1);
\draw[thick] (-1.1, -1.4) -- (-1.9, -0.1);
\draw[thick] (0.9, -1.4) -- (0.1, -0.1);
\draw[thick] (-0.9, -1.4) -- (-0.1, -0.1);
\draw[fill] (0,2) circle [radius=0.2];
\draw (-1,1) circle [radius=0.2];
\draw (1,1) circle [radius=0.2];
\draw (0,1) circle [radius=0.2];
\draw (-2,0) circle [radius=0.2];
\draw (-1,0) circle [radius=0.2];
\draw (0,0) circle [radius=0.2];
\draw (-1,-1.5) circle [radius=0.2];
\draw (1,-1.5) circle [radius=0.2];
\node[left] at (-1,1) {\footnotesize{$e$}};
\node[below] at (-2,0) {\footnotesize{$e_1$}};
\node[below] at (-1,0) {\footnotesize{$e_2$}};
\node[below] at (0,0) {\footnotesize{$e_3$}};
%\draw[thick, ->] (1.5,0.4) -- (3,0.4);
%\node[below] at (2.25,0.4) {$\tau_{e_3}$};
\end{scope}
\draw[thick, ->] (-0.5,-2) -- (-0.5,-3.5);
\draw[thick, ->] (38,-2) -- (38,-3.5);
\node[left] at (-0.5,-2.75) {$\bfI$};
\node[right] at (38,-2.75) {$\bfI$};
\begin{scope}[shift={(0,-6)}]
\draw[thick] (-0.1, 1.9) -- (-0.9, 1.1);
\draw[thick] (0, 0.1) -- (0,0.9);
\draw[thick] (0, 1.1) -- (0,1.9);
\draw[thick] (-1, 0.1) -- (-1,0.9);
\draw[thick] (0.1, 1.9) -- (0.9, 1.1);
\draw[thick] (-1.1, 0.9) -- (-1.9, 0.1);
\draw[thick] (-0.9, 0.9) -- (-0.1, 0.1);
\draw[thick] (0.9, 0.9) -- (0.1, 0.1);
\draw[thick] (-1.1, -1.4) -- (-1.9, -0.1);
\draw[thick] (0.9, -1.4) -- (0.1, -0.1);
\draw[thick] (-0.9, -1.4) -- (-0.1, -0.1);
\draw[fill] (0,2) circle [radius=0.2];
\draw[red,fill] (-1,1) circle [radius=0.2];
\draw[fill] (1,1) circle [radius=0.2];
\draw[fill] (0,1) circle [radius=0.2];
\draw[fill] (-2,0) circle [radius=0.2];
\draw[fill] (-1,0) circle [radius=0.2];
\draw[fill] (0,0) circle [radius=0.2];
\draw[fill] (-1,-1.5) circle [radius=0.2];
\draw[fill] (1,-1.5) circle [radius=0.2];
\node[left] at (-1,1) {\footnotesize{$e$}};
\node[below] at (-2,0) {\footnotesize{$e_1$}};
\node[below] at (-1,0) {\footnotesize{$e_2$}};
\node[below] at (0,0) {\footnotesize{$e_3$}};
\draw[thick, ->] (1.5,0.4) -- (36,0.4);
\node[below] at (18.75,0.4) {$t_e$};
\end{scope}
\begin{scope}[shift={(38.5,-6)}]
\draw[thick] (-0.1, 1.9) -- (-0.9, 1.1);
\draw[thick] (0, 0.1) -- (0,0.9);
\draw[thick] (0, 1.1) -- (0,1.9);
\draw[thick] (-1, 0.1) -- (-1,0.9);
\draw[thick] (0.1, 1.9) -- (0.9, 1.1);
\draw[thick] (-1.1, 0.9) -- (-1.9, 0.1);
\draw[thick] (-0.9, 0.9) -- (-0.1, 0.1);
\draw[thick] (0.9, 0.9) -- (0.1, 0.1);
\draw[thick] (-1.1, -1.4) -- (-1.9, -0.1);
\draw[thick] (0.9, -1.4) -- (0.1, -0.1);
\draw[thick] (-0.9, -1.4) -- (-0.1, -0.1);
\draw[fill] (0,2) circle [radius=0.2];
\draw[red,fill] (-1,1) circle [radius=0.2];
\draw[fill] (1,1) circle [radius=0.2];
\draw[fill] (0,1) circle [radius=0.2];
\draw[fill] (-2,0) circle [radius=0.2];
\draw[fill] (-1,0) circle [radius=0.2];
\draw[fill] (0,0) circle [radius=0.2];
\draw[fill] (-1,-1.5) circle [radius=0.2];
\draw[fill] (1,-1.5) circle [radius=0.2];
\node[left] at (-1,1) {\footnotesize{$e$}};
\node[below] at (-2,0) {\footnotesize{$e_1$}};
\node[below] at (-1,0) {\footnotesize{$e_2$}};
\node[below] at (0,0) {\footnotesize{$e_3$}};
%\draw[thick, ->] (1.5,0.4) -- (36,0.4);
%\node[below] at (18.75,0.4) {$t_e$};
\end{scope}
\end{tikzpicture}
\\\hline
\begin{tikzpicture}[scale=0.38]
\draw[thick] (-0.1, 1.9) -- (-0.9, 1.1);
\draw[thick] (0, 0.1) -- (0,0.9);
\draw[thick] (0, 1.1) -- (0,1.9);
\draw[thick] (-1, 0.1) -- (-1,0.9);
\draw[thick] (0.1, 1.9) -- (0.9, 1.1);
\draw[thick] (-1.1, 0.9) -- (-1.9, 0.1);
\draw[thick] (-0.9, 0.9) -- (-0.1, 0.1);
\draw[thick] (0.9, 0.9) -- (0.1, 0.1);
\draw[thick] (-1.1, -1.4) -- (-1.9, -0.1);
\draw[thick] (0.9, -1.4) -- (0.1, -0.1);
\draw[thick] (-0.9, -1.4) -- (-0.1, -0.1);
\draw (0,2) circle [radius=0.2];
\draw[fill] (-1,1) circle [radius=0.2];
\draw[fill] (1,1) circle [radius=0.2];
\draw (0,1) circle [radius=0.2];
\draw (-2,0) circle [radius=0.2];
\draw (-1,0) circle [radius=0.2];
\draw[red] (0,0) circle [radius=0.2];
\draw (-1,-1.5) circle [radius=0.2];
\draw (1,-1.5) circle [radius=0.2];
\node[left] at (-1,1) {\footnotesize{$e$}};
\node[below] at (-2,0) {\footnotesize{$e_1$}};
\node[below] at (-1,0) {\footnotesize{$e_2$}};
\node[below] at (0,0) {\footnotesize{$e_3$}};
\draw[thick, ->] (1.5,0.4) -- (3,0.4);
\node[below] at (2.25,0.4) {$\tau_{e_3}$};
\begin{scope}[shift={(5.5,0)}]
\draw[thick] (-0.1, 1.9) -- (-0.9, 1.1);
\draw[thick] (0, 0.1) -- (0,0.9);
\draw[thick] (0, 1.1) -- (0,1.9);
\draw[thick] (-1, 0.1) -- (-1,0.9);
\draw[thick] (0.1, 1.9) -- (0.9, 1.1);
\draw[thick] (-1.1, 0.9) -- (-1.9, 0.1);
\draw[thick] (-0.9, 0.9) -- (-0.1, 0.1);
\draw[thick] (0.9, 0.9) -- (0.1, 0.1);
\draw[thick] (-1.1, -1.4) -- (-1.9, -0.1);
\draw[thick] (0.9, -1.4) -- (0.1, -0.1);
\draw[thick] (-0.9, -1.4) -- (-0.1, -0.1);
\draw (0,2) circle [radius=0.2];
\draw[fill] (-1,1) circle [radius=0.2];
\draw[fill] (1,1) circle [radius=0.2];
\draw (0,1) circle [radius=0.2];
\draw (-2,0) circle [radius=0.2];
\draw[red] (-1,0) circle [radius=0.2];
\draw (0,0) circle [radius=0.2];
\draw (-1,-1.5) circle [radius=0.2];
\draw (1,-1.5) circle [radius=0.2];
\node[left] at (-1,1) {\footnotesize{$e$}};
\node[below] at (-2,0) {\footnotesize{$e_1$}};
\node[below] at (-1,0) {\footnotesize{$e_2$}};
\node[below] at (0,0) {\footnotesize{$e_3$}};
\draw[thick, ->] (1.5,0.4) -- (3,0.4);
\node[below] at (2.25,0.4) {$\tau_{e_2}$};
\end{scope}
\begin{scope}[shift={(11,0)}]
\draw[thick] (-0.1, 1.9) -- (-0.9, 1.1);
\draw[thick] (0, 0.1) -- (0,0.9);
\draw[thick] (0, 1.1) -- (0,1.9);
\draw[thick] (-1, 0.1) -- (-1,0.9);
\draw[thick] (0.1, 1.9) -- (0.9, 1.1);
\draw[thick] (-1.1, 0.9) -- (-1.9, 0.1);
\draw[thick] (-0.9, 0.9) -- (-0.1, 0.1);
\draw[thick] (0.9, 0.9) -- (0.1, 0.1);
\draw[thick] (-1.1, -1.4) -- (-1.9, -0.1);
\draw[thick] (0.9, -1.4) -- (0.1, -0.1);
\draw[thick] (-0.9, -1.4) -- (-0.1, -0.1);
\draw (0,2) circle [radius=0.2];
\draw[fill] (-1,1) circle [radius=0.2];
\draw[fill] (1,1) circle [radius=0.2];
\draw (0,1) circle [radius=0.2];
\draw[red] (-2,0) circle [radius=0.2];
\draw (-1,0) circle [radius=0.2];
\draw (0,0) circle [radius=0.2];
\draw (-1,-1.5) circle [radius=0.2];
\draw (1,-1.5) circle [radius=0.2];
\node[left] at (-1,1) {\footnotesize{$e$}};
\node[below] at (-2,0) {\footnotesize{$e_1$}};
\node[below] at (-1,0) {\footnotesize{$e_2$}};
\node[below] at (0,0) {\footnotesize{$e_3$}};
\draw[thick, ->] (1.5,0.4) -- (3,0.4);
\node[below] at (2.25,0.4) {$\tau_{e_1}$};
\end{scope}
\begin{scope}[shift={(16.5,0)}]
\draw[thick] (-0.1, 1.9) -- (-0.9, 1.1);
\draw[thick] (0, 0.1) -- (0,0.9);
\draw[thick] (0, 1.1) -- (0,1.9);
\draw[thick] (-1, 0.1) -- (-1,0.9);
\draw[thick] (0.1, 1.9) -- (0.9, 1.1);
\draw[thick] (-1.1, 0.9) -- (-1.9, 0.1);
\draw[thick] (-0.9, 0.9) -- (-0.1, 0.1);
\draw[thick] (0.9, 0.9) -- (0.1, 0.1);
\draw[thick] (-1.1, -1.4) -- (-1.9, -0.1);
\draw[thick] (0.9, -1.4) -- (0.1, -0.1);
\draw[thick] (-0.9, -1.4) -- (-0.1, -0.1);
\draw (0,2) circle [radius=0.2];
\draw[red,fill] (-1,1) circle [radius=0.2];
\draw[fill] (1,1) circle [radius=0.2];
\draw (0,1) circle [radius=0.2];
\draw (-2,0) circle [radius=0.2];
\draw (-1,0) circle [radius=0.2];
\draw (0,0) circle [radius=0.2];
\draw (-1,-1.5) circle [radius=0.2];
\draw (1,-1.5) circle [radius=0.2];
\node[left] at (-1,1) {\footnotesize{$e$}};
\node[below] at (-2,0) {\footnotesize{$e_1$}};
\node[below] at (-1,0) {\footnotesize{$e_2$}};
\node[below] at (0,0) {\footnotesize{$e_3$}};
\draw[thick, ->] (1.5,0.4) -- (3,0.4);
\node[below] at (2.25,0.4) {$\tau_{e}$};
\end{scope}
\begin{scope}[shift={(22,0)}]
\draw[thick] (-0.1, 1.9) -- (-0.9, 1.1);
\draw[thick] (0, 0.1) -- (0,0.9);
\draw[thick] (0, 1.1) -- (0,1.9);
\draw[thick] (-1, 0.1) -- (-1,0.9);
\draw[thick] (0.1, 1.9) -- (0.9, 1.1);
\draw[thick] (-1.1, 0.9) -- (-1.9, 0.1);
\draw[thick] (-0.9, 0.9) -- (-0.1, 0.1);
\draw[thick] (0.9, 0.9) -- (0.1, 0.1);
\draw[thick] (-1.1, -1.4) -- (-1.9, -0.1);
\draw[thick] (0.9, -1.4) -- (0.1, -0.1);
\draw[thick] (-0.9, -1.4) -- (-0.1, -0.1);
\draw (0,2) circle [radius=0.2];
\draw (-1,1) circle [radius=0.2];
\draw[fill] (1,1) circle [radius=0.2];
\draw (0,1) circle [radius=0.2];
\draw (-2,0) circle [radius=0.2];
\draw (-1,0) circle [radius=0.2];
\draw[red] (0,0) circle [radius=0.2];
\draw (-1,-1.5) circle [radius=0.2];
\draw (1,-1.5) circle [radius=0.2];
\node[left] at (-1,1) {\footnotesize{$e$}};
\node[below] at (-2,0) {\footnotesize{$e_1$}};
\node[below] at (-1,0) {\footnotesize{$e_2$}};
\node[below] at (0,0) {\footnotesize{$e_3$}};
\draw[thick, ->] (1.5,0.4) -- (3,0.4);
\node[below] at (2.25,0.4) {$\tau_{e_3}$};
\end{scope}
\begin{scope}[shift={(27.5,0)}]
\draw[thick] (-0.1, 1.9) -- (-0.9, 1.1);
\draw[thick] (0, 0.1) -- (0,0.9);
\draw[thick] (0, 1.1) -- (0,1.9);
\draw[thick] (-1, 0.1) -- (-1,0.9);
\draw[thick] (0.1, 1.9) -- (0.9, 1.1);
\draw[thick] (-1.1, 0.9) -- (-1.9, 0.1);
\draw[thick] (-0.9, 0.9) -- (-0.1, 0.1);
\draw[thick] (0.9, 0.9) -- (0.1, 0.1);
\draw[thick] (-1.1, -1.4) -- (-1.9, -0.1);
\draw[thick] (0.9, -1.4) -- (0.1, -0.1);
\draw[thick] (-0.9, -1.4) -- (-0.1, -0.1);
\draw (0,2) circle [radius=0.2];
\draw (-1,1) circle [radius=0.2];
\draw[fill] (1,1) circle [radius=0.2];
\draw (0,1) circle [radius=0.2];
\draw (-2,0) circle [radius=0.2];
\draw[red] (-1,0) circle [radius=0.2];
\draw (0,0) circle [radius=0.2];
\draw (-1,-1.5) circle [radius=0.2];
\draw (1,-1.5) circle [radius=0.2];
\node[left] at (-1,1) {\footnotesize{$e$}};
\node[below] at (-2,0) {\footnotesize{$e_1$}};
\node[below] at (-1,0) {\footnotesize{$e_2$}};
\node[below] at (0,0) {\footnotesize{$e_3$}};
\draw[thick, ->] (1.5,0.4) -- (3,0.4);
\node[below] at (2.25,0.4) {$\tau_{e_2}$};
\end{scope}
\begin{scope}[shift={(33,0)}]
\draw[thick] (-0.1, 1.9) -- (-0.9, 1.1);
\draw[thick] (0, 0.1) -- (0,0.9);
\draw[thick] (0, 1.1) -- (0,1.9);
\draw[thick] (-1, 0.1) -- (-1,0.9);
\draw[thick] (0.1, 1.9) -- (0.9, 1.1);
\draw[thick] (-1.1, 0.9) -- (-1.9, 0.1);
\draw[thick] (-0.9, 0.9) -- (-0.1, 0.1);
\draw[thick] (0.9, 0.9) -- (0.1, 0.1);
\draw[thick] (-1.1, -1.4) -- (-1.9, -0.1);
\draw[thick] (0.9, -1.4) -- (0.1, -0.1);
\draw[thick] (-0.9, -1.4) -- (-0.1, -0.1);
\draw (0,2) circle [radius=0.2];
\draw (-1,1) circle [radius=0.2];
\draw[fill] (1,1) circle [radius=0.2];
\draw (0,1) circle [radius=0.2];
\draw[red] (-2,0) circle [radius=0.2];
\draw[fill] (-1,0) circle [radius=0.2];
\draw (0,0) circle [radius=0.2];
\draw (-1,-1.5) circle [radius=0.2];
\draw (1,-1.5) circle [radius=0.2];
\node[left] at (-1,1) {\footnotesize{$e$}};
\node[below] at (-2,0) {\footnotesize{$e_1$}};
\node[below] at (-1,0) {\footnotesize{$e_2$}};
\node[below] at (0,0) {\footnotesize{$e_3$}};
\draw[thick, ->] (1.5,0.4) -- (3,0.4);
\node[below] at (2.25,0.4) {$\tau_{e_1}$};
\end{scope}
\begin{scope}[shift={(38.5,0)}]
\draw[thick] (-0.1, 1.9) -- (-0.9, 1.1);
\draw[thick] (0, 0.1) -- (0,0.9);
\draw[thick] (0, 1.1) -- (0,1.9);
\draw[thick] (-1, 0.1) -- (-1,0.9);
\draw[thick] (0.1, 1.9) -- (0.9, 1.1);
\draw[thick] (-1.1, 0.9) -- (-1.9, 0.1);
\draw[thick] (-0.9, 0.9) -- (-0.1, 0.1);
\draw[thick] (0.9, 0.9) -- (0.1, 0.1);
\draw[thick] (-1.1, -1.4) -- (-1.9, -0.1);
\draw[thick] (0.9, -1.4) -- (0.1, -0.1);
\draw[thick] (-0.9, -1.4) -- (-0.1, -0.1);
\draw (0,2) circle [radius=0.2];
\draw (-1,1) circle [radius=0.2];
\draw[fill] (1,1) circle [radius=0.2];
\draw (0,1) circle [radius=0.2];
\draw[fill] (-2,0) circle [radius=0.2];
\draw[fill] (-1,0) circle [radius=0.2];
\draw (0,0) circle [radius=0.2];
\draw (-1,-1.5) circle [radius=0.2];
\draw (1,-1.5) circle [radius=0.2];
\node[left] at (-1,1) {\footnotesize{$e$}};
\node[below] at (-2,0) {\footnotesize{$e_1$}};
\node[below] at (-1,0) {\footnotesize{$e_2$}};
\node[below] at (0,0) {\footnotesize{$e_3$}};
%\draw[thick, ->] (1.5,0.4) -- (3,0.4);
%\node[below] at (2.25,0.4) {$\tau_{e_3}$};
\end{scope}
\draw[thick, ->] (-0.5,-2) -- (-0.5,-3.5);
\draw[thick, ->] (38,-2) -- (38,-3.5);
\node[left] at (-0.5,-2.75) {$\bfI$};
\node[right] at (38,-2.75) {$\bfI$};
\begin{scope}[shift={(0,-6)}]
\draw[thick] (-0.1, 1.9) -- (-0.9, 1.1);
\draw[thick] (0, 0.1) -- (0,0.9);
\draw[thick] (0, 1.1) -- (0,1.9);
\draw[thick] (-1, 0.1) -- (-1,0.9);
\draw[thick] (0.1, 1.9) -- (0.9, 1.1);
\draw[thick] (-1.1, 0.9) -- (-1.9, 0.1);
\draw[thick] (-0.9, 0.9) -- (-0.1, 0.1);
\draw[thick] (0.9, 0.9) -- (0.1, 0.1);
\draw[thick] (-1.1, -1.4) -- (-1.9, -0.1);
\draw[thick] (0.9, -1.4) -- (0.1, -0.1);
\draw[thick] (-0.9, -1.4) -- (-0.1, -0.1);
\draw (0,2) circle [radius=0.2];
\draw[red,fill] (-1,1) circle [radius=0.2];
\draw[fill] (1,1) circle [radius=0.2];
\draw (0,1) circle [radius=0.2];
\draw[fill] (-2,0) circle [radius=0.2];
\draw[fill] (-1,0) circle [radius=0.2];
\draw[fill] (0,0) circle [radius=0.2];
\draw[fill] (-1,-1.5) circle [radius=0.2];
\draw[fill] (1,-1.5) circle [radius=0.2];
\node[left] at (-1,1) {\footnotesize{$e$}};
\node[below] at (-2,0) {\footnotesize{$e_1$}};
\node[below] at (-1,0) {\footnotesize{$e_2$}};
\node[below] at (0,0) {\footnotesize{$e_3$}};
\draw[thick, ->] (1.5,0.4) -- (36,0.4);
\node[below] at (18.75,0.4) {$t_e$};
\end{scope}
\begin{scope}[shift={(38.5,-6)}]
\draw[thick] (-0.1, 1.9) -- (-0.9, 1.1);
\draw[thick] (0, 0.1) -- (0,0.9);
\draw[thick] (0, 1.1) -- (0,1.9);
\draw[thick] (-1, 0.1) -- (-1,0.9);
\draw[thick] (0.1, 1.9) -- (0.9, 1.1);
\draw[thick] (-1.1, 0.9) -- (-1.9, 0.1);
\draw[thick] (-0.9, 0.9) -- (-0.1, 0.1);
\draw[thick] (0.9, 0.9) -- (0.1, 0.1);
\draw[thick] (-1.1, -1.4) -- (-1.9, -0.1);
\draw[thick] (0.9, -1.4) -- (0.1, -0.1);
\draw[thick] (-0.9, -1.4) -- (-0.1, -0.1);
\draw (0,2) circle [radius=0.2];
\draw[red] (-1,1) circle [radius=0.2];
\draw[fill] (1,1) circle [radius=0.2];
\draw (0,1) circle [radius=0.2];
\draw[fill] (-2,0) circle [radius=0.2];
\draw[fill] (-1,0) circle [radius=0.2];
\draw[fill] (0,0) circle [radius=0.2];
\draw[fill] (-1,-1.5) circle [radius=0.2];
\draw[fill] (1,-1.5) circle [radius=0.2];
\node[left] at (-1,1) {\footnotesize{$e$}};
\node[below] at (-2,0) {\footnotesize{$e_1$}};
\node[below] at (-1,0) {\footnotesize{$e_2$}};
\node[below] at (0,0) {\footnotesize{$e_3$}};
%\draw[thick, ->] (1.5,0.4) -- (36,0.4);
%\node[below] at (18.75,0.4) {$t_e$};
\end{scope}
\end{tikzpicture}
\\\hline
\begin{tikzpicture}[scale=0.38]
\draw[thick] (-0.1, 1.9) -- (-0.9, 1.1);
\draw[thick] (0, 0.1) -- (0,0.9);
\draw[thick] (0, 1.1) -- (0,1.9);
\draw[thick] (-1, 0.1) -- (-1,0.9);
\draw[thick] (0.1, 1.9) -- (0.9, 1.1);
\draw[thick] (-1.1, 0.9) -- (-1.9, 0.1);
\draw[thick] (-0.9, 0.9) -- (-0.1, 0.1);
\draw[thick] (0.9, 0.9) -- (0.1, 0.1);
\draw[thick] (-1.1, -1.4) -- (-1.9, -0.1);
\draw[thick] (0.9, -1.4) -- (0.1, -0.1);
\draw[thick] (-0.9, -1.4) -- (-0.1, -0.1);
\draw (0,2) circle [radius=0.2];
\draw (-1,1) circle [radius=0.2];
\draw (1,1) circle [radius=0.2];
\draw (0,1) circle [radius=0.2];
\draw (-2,0) circle [radius=0.2];
\draw[fill] (-1,0) circle [radius=0.2];
\draw[red] (0,0) circle [radius=0.2];
\draw (-1,-1.5) circle [radius=0.2];
\draw[fill] (1,-1.5) circle [radius=0.2];
\node[left] at (-1,1) {\footnotesize{$e$}};
\node[below] at (-2,0) {\footnotesize{$e_1$}};
\node[below] at (-1,0) {\footnotesize{$e_2$}};
\node[below] at (0,0) {\footnotesize{$e_3$}};
\draw[thick, ->] (1.5,0.4) -- (3,0.4);
\node[below] at (2.25,0.4) {$\tau_{e_3}$};
\begin{scope}[shift={(5.5,0)}]
\draw[thick] (-0.1, 1.9) -- (-0.9, 1.1);
\draw[thick] (0, 0.1) -- (0,0.9);
\draw[thick] (0, 1.1) -- (0,1.9);
\draw[thick] (-1, 0.1) -- (-1,0.9);
\draw[thick] (0.1, 1.9) -- (0.9, 1.1);
\draw[thick] (-1.1, 0.9) -- (-1.9, 0.1);
\draw[thick] (-0.9, 0.9) -- (-0.1, 0.1);
\draw[thick] (0.9, 0.9) -- (0.1, 0.1);
\draw[thick] (-1.1, -1.4) -- (-1.9, -0.1);
\draw[thick] (0.9, -1.4) -- (0.1, -0.1);
\draw[thick] (-0.9, -1.4) -- (-0.1, -0.1);
\draw (0,2) circle [radius=0.2];
\draw (-1,1) circle [radius=0.2];
\draw (1,1) circle [radius=0.2];
\draw (0,1) circle [radius=0.2];
\draw (-2,0) circle [radius=0.2];
\draw[red,fill] (-1,0) circle [radius=0.2];
\draw (0,0) circle [radius=0.2];
\draw (-1,-1.5) circle [radius=0.2];
\draw[fill] (1,-1.5) circle [radius=0.2];
\node[left] at (-1,1) {\footnotesize{$e$}};
\node[below] at (-2,0) {\footnotesize{$e_1$}};
\node[below] at (-1,0) {\footnotesize{$e_2$}};
\node[below] at (0,0) {\footnotesize{$e_3$}};
\draw[thick, ->] (1.5,0.4) -- (3,0.4);
\node[below] at (2.25,0.4) {$\tau_{e_2}$};
\end{scope}
\begin{scope}[shift={(11,0)}]
\draw[thick] (-0.1, 1.9) -- (-0.9, 1.1);
\draw[thick] (0, 0.1) -- (0,0.9);
\draw[thick] (0, 1.1) -- (0,1.9);
\draw[thick] (-1, 0.1) -- (-1,0.9);
\draw[thick] (0.1, 1.9) -- (0.9, 1.1);
\draw[thick] (-1.1, 0.9) -- (-1.9, 0.1);
\draw[thick] (-0.9, 0.9) -- (-0.1, 0.1);
\draw[thick] (0.9, 0.9) -- (0.1, 0.1);
\draw[thick] (-1.1, -1.4) -- (-1.9, -0.1);
\draw[thick] (0.9, -1.4) -- (0.1, -0.1);
\draw[thick] (-0.9, -1.4) -- (-0.1, -0.1);
\draw (0,2) circle [radius=0.2];
\draw (-1,1) circle [radius=0.2];
\draw (1,1) circle [radius=0.2];
\draw (0,1) circle [radius=0.2];
\draw[red] (-2,0) circle [radius=0.2];
\draw (-1,0) circle [radius=0.2];
\draw (0,0) circle [radius=0.2];
\draw (-1,-1.5) circle [radius=0.2];
\draw[fill] (1,-1.5) circle [radius=0.2];
\node[left] at (-1,1) {\footnotesize{$e$}};
\node[below] at (-2,0) {\footnotesize{$e_1$}};
\node[below] at (-1,0) {\footnotesize{$e_2$}};
\node[below] at (0,0) {\footnotesize{$e_3$}};
\draw[thick, ->] (1.5,0.4) -- (3,0.4);
\node[below] at (2.25,0.4) {$\tau_{e_1}$};
\end{scope}
\begin{scope}[shift={(16.5,0)}]
\draw[thick] (-0.1, 1.9) -- (-0.9, 1.1);
\draw[thick] (0, 0.1) -- (0,0.9);
\draw[thick] (0, 1.1) -- (0,1.9);
\draw[thick] (-1, 0.1) -- (-1,0.9);
\draw[thick] (0.1, 1.9) -- (0.9, 1.1);
\draw[thick] (-1.1, 0.9) -- (-1.9, 0.1);
\draw[thick] (-0.9, 0.9) -- (-0.1, 0.1);
\draw[thick] (0.9, 0.9) -- (0.1, 0.1);
\draw[thick] (-1.1, -1.4) -- (-1.9, -0.1);
\draw[thick] (0.9, -1.4) -- (0.1, -0.1);
\draw[thick] (-0.9, -1.4) -- (-0.1, -0.1);
\draw (0,2) circle [radius=0.2];
\draw[red] (-1,1) circle [radius=0.2];
\draw (1,1) circle [radius=0.2];
\draw (0,1) circle [radius=0.2];
\draw[fill] (-2,0) circle [radius=0.2];
\draw (-1,0) circle [radius=0.2];
\draw (0,0) circle [radius=0.2];
\draw (-1,-1.5) circle [radius=0.2];
\draw[fill] (1,-1.5) circle [radius=0.2];
\node[left] at (-1,1) {\footnotesize{$e$}};
\node[below] at (-2,0) {\footnotesize{$e_1$}};
\node[below] at (-1,0) {\footnotesize{$e_2$}};
\node[below] at (0,0) {\footnotesize{$e_3$}};
\draw[thick, ->] (1.5,0.4) -- (3,0.4);
\node[below] at (2.25,0.4) {$\tau_{e}$};
\end{scope}
\begin{scope}[shift={(22,0)}]
\draw[thick] (-0.1, 1.9) -- (-0.9, 1.1);
\draw[thick] (0, 0.1) -- (0,0.9);
\draw[thick] (0, 1.1) -- (0,1.9);
\draw[thick] (-1, 0.1) -- (-1,0.9);
\draw[thick] (0.1, 1.9) -- (0.9, 1.1);
\draw[thick] (-1.1, 0.9) -- (-1.9, 0.1);
\draw[thick] (-0.9, 0.9) -- (-0.1, 0.1);
\draw[thick] (0.9, 0.9) -- (0.1, 0.1);
\draw[thick] (-1.1, -1.4) -- (-1.9, -0.1);
\draw[thick] (0.9, -1.4) -- (0.1, -0.1);
\draw[thick] (-0.9, -1.4) -- (-0.1, -0.1);
\draw (0,2) circle [radius=0.2];
\draw (-1,1) circle [radius=0.2];
\draw (1,1) circle [radius=0.2];
\draw (0,1) circle [radius=0.2];
\draw[fill] (-2,0) circle [radius=0.2];
\draw (-1,0) circle [radius=0.2];
\draw[red] (0,0) circle [radius=0.2];
\draw (-1,-1.5) circle [radius=0.2];
\draw[fill] (1,-1.5) circle [radius=0.2];
\node[left] at (-1,1) {\footnotesize{$e$}};
\node[below] at (-2,0) {\footnotesize{$e_1$}};
\node[below] at (-1,0) {\footnotesize{$e_2$}};
\node[below] at (0,0) {\footnotesize{$e_3$}};
\draw[thick, ->] (1.5,0.4) -- (3,0.4);
\node[below] at (2.25,0.4) {$\tau_{e_3}$};
\end{scope}
\begin{scope}[shift={(27.5,0)}]
\draw[thick] (-0.1, 1.9) -- (-0.9, 1.1);
\draw[thick] (0, 0.1) -- (0,0.9);
\draw[thick] (0, 1.1) -- (0,1.9);
\draw[thick] (-1, 0.1) -- (-1,0.9);
\draw[thick] (0.1, 1.9) -- (0.9, 1.1);
\draw[thick] (-1.1, 0.9) -- (-1.9, 0.1);
\draw[thick] (-0.9, 0.9) -- (-0.1, 0.1);
\draw[thick] (0.9, 0.9) -- (0.1, 0.1);
\draw[thick] (-1.1, -1.4) -- (-1.9, -0.1);
\draw[thick] (0.9, -1.4) -- (0.1, -0.1);
\draw[thick] (-0.9, -1.4) -- (-0.1, -0.1);
\draw (0,2) circle [radius=0.2];
\draw (-1,1) circle [radius=0.2];
\draw (1,1) circle [radius=0.2];
\draw (0,1) circle [radius=0.2];
\draw[fill] (-2,0) circle [radius=0.2];
\draw[red] (-1,0) circle [radius=0.2];
\draw (0,0) circle [radius=0.2];
\draw (-1,-1.5) circle [radius=0.2];
\draw[fill] (1,-1.5) circle [radius=0.2];
\node[left] at (-1,1) {\footnotesize{$e$}};
\node[below] at (-2,0) {\footnotesize{$e_1$}};
\node[below] at (-1,0) {\footnotesize{$e_2$}};
\node[below] at (0,0) {\footnotesize{$e_3$}};
\draw[thick, ->] (1.5,0.4) -- (3,0.4);
\node[below] at (2.25,0.4) {$\tau_{e_2}$};
\end{scope}
\begin{scope}[shift={(33,0)}]
\draw[thick] (-0.1, 1.9) -- (-0.9, 1.1);
\draw[thick] (0, 0.1) -- (0,0.9);
\draw[thick] (0, 1.1) -- (0,1.9);
\draw[thick] (-1, 0.1) -- (-1,0.9);
\draw[thick] (0.1, 1.9) -- (0.9, 1.1);
\draw[thick] (-1.1, 0.9) -- (-1.9, 0.1);
\draw[thick] (-0.9, 0.9) -- (-0.1, 0.1);
\draw[thick] (0.9, 0.9) -- (0.1, 0.1);
\draw[thick] (-1.1, -1.4) -- (-1.9, -0.1);
\draw[thick] (0.9, -1.4) -- (0.1, -0.1);
\draw[thick] (-0.9, -1.4) -- (-0.1, -0.1);
\draw (0,2) circle [radius=0.2];
\draw (-1,1) circle [radius=0.2];
\draw (1,1) circle [radius=0.2];
\draw (0,1) circle [radius=0.2];
\draw[red,fill] (-2,0) circle [radius=0.2];
\draw[fill] (-1,0) circle [radius=0.2];
\draw (0,0) circle [radius=0.2];
\draw (-1,-1.5) circle [radius=0.2];
\draw[fill] (1,-1.5) circle [radius=0.2];
\node[left] at (-1,1) {\footnotesize{$e$}};
\node[below] at (-2,0) {\footnotesize{$e_1$}};
\node[below] at (-1,0) {\footnotesize{$e_2$}};
\node[below] at (0,0) {\footnotesize{$e_3$}};
\draw[thick, ->] (1.5,0.4) -- (3,0.4);
\node[below] at (2.25,0.4) {$\tau_{e_1}$};
\end{scope}
\begin{scope}[shift={(38.5,0)}]
\draw[thick] (-0.1, 1.9) -- (-0.9, 1.1);
\draw[thick] (0, 0.1) -- (0,0.9);
\draw[thick] (0, 1.1) -- (0,1.9);
\draw[thick] (-1, 0.1) -- (-1,0.9);
\draw[thick] (0.1, 1.9) -- (0.9, 1.1);
\draw[thick] (-1.1, 0.9) -- (-1.9, 0.1);
\draw[thick] (-0.9, 0.9) -- (-0.1, 0.1);
\draw[thick] (0.9, 0.9) -- (0.1, 0.1);
\draw[thick] (-1.1, -1.4) -- (-1.9, -0.1);
\draw[thick] (0.9, -1.4) -- (0.1, -0.1);
\draw[thick] (-0.9, -1.4) -- (-0.1, -0.1);
\draw (0,2) circle [radius=0.2];
\draw (-1,1) circle [radius=0.2];
\draw (1,1) circle [radius=0.2];
\draw (0,1) circle [radius=0.2];
\draw (-2,0) circle [radius=0.2];
\draw[fill] (-1,0) circle [radius=0.2];
\draw (0,0) circle [radius=0.2];
\draw (-1,-1.5) circle [radius=0.2];
\draw[fill] (1,-1.5) circle [radius=0.2];
\node[left] at (-1,1) {\footnotesize{$e$}};
\node[below] at (-2,0) {\footnotesize{$e_1$}};
\node[below] at (-1,0) {\footnotesize{$e_2$}};
\node[below] at (0,0) {\footnotesize{$e_3$}};
%\draw[thick, ->] (1.5,0.4) -- (3,0.4);
%\node[below] at (2.25,0.4) {$\tau_{e_3}$};
\end{scope}
\draw[thick, ->] (-0.5,-2) -- (-0.5,-3.5);
\draw[thick, ->] (38,-2) -- (38,-3.5);
\node[left] at (-0.5,-2.75) {$\bfI$};
\node[right] at (38,-2.75) {$\bfI$};
\begin{scope}[shift={(0,-6)}]
\draw[thick] (-0.1, 1.9) -- (-0.9, 1.1);
\draw[thick] (0, 0.1) -- (0,0.9);
\draw[thick] (0, 1.1) -- (0,1.9);
\draw[thick] (-1, 0.1) -- (-1,0.9);
\draw[thick] (0.1, 1.9) -- (0.9, 1.1);
\draw[thick] (-1.1, 0.9) -- (-1.9, 0.1);
\draw[thick] (-0.9, 0.9) -- (-0.1, 0.1);
\draw[thick] (0.9, 0.9) -- (0.1, 0.1);
\draw[thick] (-1.1, -1.4) -- (-1.9, -0.1);
\draw[thick] (0.9, -1.4) -- (0.1, -0.1);
\draw[thick] (-0.9, -1.4) -- (-0.1, -0.1);
\draw (0,2) circle [radius=0.2];
\draw[red] (-1,1) circle [radius=0.2];
\draw (1,1) circle [radius=0.2];
\draw (0,1) circle [radius=0.2];
\draw (-2,0) circle [radius=0.2];
\draw[fill] (-1,0) circle [radius=0.2];
\draw (0,0) circle [radius=0.2];
\draw (-1,-1.5) circle [radius=0.2];
\draw[fill] (1,-1.5) circle [radius=0.2];
\node[left] at (-1,1) {\footnotesize{$e$}};
\node[below] at (-2,0) {\footnotesize{$e_1$}};
\node[below] at (-1,0) {\footnotesize{$e_2$}};
\node[below] at (0,0) {\footnotesize{$e_3$}};
\draw[thick, ->] (1.5,0.4) -- (36,0.4);
\node[below] at (18.75,0.4) {$t_e$};
\end{scope}
\begin{scope}[shift={(38.5,-6)}]
\draw[thick] (-0.1, 1.9) -- (-0.9, 1.1);
\draw[thick] (0, 0.1) -- (0,0.9);
\draw[thick] (0, 1.1) -- (0,1.9);
\draw[thick] (-1, 0.1) -- (-1,0.9);
\draw[thick] (0.1, 1.9) -- (0.9, 1.1);
\draw[thick] (-1.1, 0.9) -- (-1.9, 0.1);
\draw[thick] (-0.9, 0.9) -- (-0.1, 0.1);
\draw[thick] (0.9, 0.9) -- (0.1, 0.1);
\draw[thick] (-1.1, -1.4) -- (-1.9, -0.1);
\draw[thick] (0.9, -1.4) -- (0.1, -0.1);
\draw[thick] (-0.9, -1.4) -- (-0.1, -0.1);
\draw (0,2) circle [radius=0.2];
\draw[red] (-1,1) circle [radius=0.2];
\draw (1,1) circle [radius=0.2];
\draw (0,1) circle [radius=0.2];
\draw (-2,0) circle [radius=0.2];
\draw[fill] (-1,0) circle [radius=0.2];
\draw (0,0) circle [radius=0.2];
\draw (-1,-1.5) circle [radius=0.2];
\draw[fill] (1,-1.5) circle [radius=0.2];
\node[left] at (-1,1) {\footnotesize{$e$}};
\node[below] at (-2,0) {\footnotesize{$e_1$}};
\node[below] at (-1,0) {\footnotesize{$e_2$}};
\node[below] at (0,0) {\footnotesize{$e_3$}};
%\draw[thick, ->] (1.5,0.4) -- (36,0.4);
%\node[below] at (18.75,0.4) {$t_e$};
\end{scope}
\end{tikzpicture}
\\\hline
\begin{tikzpicture}[scale=0.38]
\draw[thick] (-0.1, 1.9) -- (-0.9, 1.1);
\draw[thick] (0, 0.1) -- (0,0.9);
\draw[thick] (0, 1.1) -- (0,1.9);
\draw[thick] (-1, 0.1) -- (-1,0.9);
\draw[thick] (0.1, 1.9) -- (0.9, 1.1);
\draw[thick] (-1.1, 0.9) -- (-1.9, 0.1);
\draw[thick] (-0.9, 0.9) -- (-0.1, 0.1);
\draw[thick] (0.9, 0.9) -- (0.1, 0.1);
\draw[thick] (-1.1, -1.4) -- (-1.9, -0.1);
\draw[thick] (0.9, -1.4) -- (0.1, -0.1);
\draw[thick] (-0.9, -1.4) -- (-0.1, -0.1);
\draw (0,2) circle [radius=0.2];
\draw (-1,1) circle [radius=0.2];
\draw[fill] (1,1) circle [radius=0.2];
\draw (0,1) circle [radius=0.2];
\draw[fill] (-2,0) circle [radius=0.2];
\draw[fill] (-1,0) circle [radius=0.2];
\draw[red] (0,0) circle [radius=0.2];
\draw (-1,-1.5) circle [radius=0.2];
\draw (1,-1.5) circle [radius=0.2];
\node[left] at (-1,1) {\footnotesize{$e$}};
\node[below] at (-2,0) {\footnotesize{$e_1$}};
\node[below] at (-1,0) {\footnotesize{$e_2$}};
\node[below] at (0,0) {\footnotesize{$e_3$}};
\draw[thick, ->] (1.5,0.4) -- (3,0.4);
\node[below] at (2.25,0.4) {$\tau_{e_3}$};
\begin{scope}[shift={(5.5,0)}]
\draw[thick] (-0.1, 1.9) -- (-0.9, 1.1);
\draw[thick] (0, 0.1) -- (0,0.9);
\draw[thick] (0, 1.1) -- (0,1.9);
\draw[thick] (-1, 0.1) -- (-1,0.9);
\draw[thick] (0.1, 1.9) -- (0.9, 1.1);
\draw[thick] (-1.1, 0.9) -- (-1.9, 0.1);
\draw[thick] (-0.9, 0.9) -- (-0.1, 0.1);
\draw[thick] (0.9, 0.9) -- (0.1, 0.1);
\draw[thick] (-1.1, -1.4) -- (-1.9, -0.1);
\draw[thick] (0.9, -1.4) -- (0.1, -0.1);
\draw[thick] (-0.9, -1.4) -- (-0.1, -0.1);
\draw (0,2) circle [radius=0.2];
\draw (-1,1) circle [radius=0.2];
\draw[fill] (1,1) circle [radius=0.2];
\draw (0,1) circle [radius=0.2];
\draw[fill] (-2,0) circle [radius=0.2];
\draw[red,fill] (-1,0) circle [radius=0.2];
\draw (0,0) circle [radius=0.2];
\draw (-1,-1.5) circle [radius=0.2];
\draw (1,-1.5) circle [radius=0.2];
\node[left] at (-1,1) {\footnotesize{$e$}};
\node[below] at (-2,0) {\footnotesize{$e_1$}};
\node[below] at (-1,0) {\footnotesize{$e_2$}};
\node[below] at (0,0) {\footnotesize{$e_3$}};
\draw[thick, ->] (1.5,0.4) -- (3,0.4);
\node[below] at (2.25,0.4) {$\tau_{e_2}$};
\end{scope}
\begin{scope}[shift={(11,0)}]
\draw[thick] (-0.1, 1.9) -- (-0.9, 1.1);
\draw[thick] (0, 0.1) -- (0,0.9);
\draw[thick] (0, 1.1) -- (0,1.9);
\draw[thick] (-1, 0.1) -- (-1,0.9);
\draw[thick] (0.1, 1.9) -- (0.9, 1.1);
\draw[thick] (-1.1, 0.9) -- (-1.9, 0.1);
\draw[thick] (-0.9, 0.9) -- (-0.1, 0.1);
\draw[thick] (0.9, 0.9) -- (0.1, 0.1);
\draw[thick] (-1.1, -1.4) -- (-1.9, -0.1);
\draw[thick] (0.9, -1.4) -- (0.1, -0.1);
\draw[thick] (-0.9, -1.4) -- (-0.1, -0.1);
\draw (0,2) circle [radius=0.2];
\draw (-1,1) circle [radius=0.2];
\draw[fill] (1,1) circle [radius=0.2];
\draw (0,1) circle [radius=0.2];
\draw[red,fill] (-2,0) circle [radius=0.2];
\draw (-1,0) circle [radius=0.2];
\draw (0,0) circle [radius=0.2];
\draw (-1,-1.5) circle [radius=0.2];
\draw (1,-1.5) circle [radius=0.2];
\node[left] at (-1,1) {\footnotesize{$e$}};
\node[below] at (-2,0) {\footnotesize{$e_1$}};
\node[below] at (-1,0) {\footnotesize{$e_2$}};
\node[below] at (0,0) {\footnotesize{$e_3$}};
\draw[thick, ->] (1.5,0.4) -- (3,0.4);
\node[below] at (2.25,0.4) {$\tau_{e_1}$};
\end{scope}
\begin{scope}[shift={(16.5,0)}]
\draw[thick] (-0.1, 1.9) -- (-0.9, 1.1);
\draw[thick] (0, 0.1) -- (0,0.9);
\draw[thick] (0, 1.1) -- (0,1.9);
\draw[thick] (-1, 0.1) -- (-1,0.9);
\draw[thick] (0.1, 1.9) -- (0.9, 1.1);
\draw[thick] (-1.1, 0.9) -- (-1.9, 0.1);
\draw[thick] (-0.9, 0.9) -- (-0.1, 0.1);
\draw[thick] (0.9, 0.9) -- (0.1, 0.1);
\draw[thick] (-1.1, -1.4) -- (-1.9, -0.1);
\draw[thick] (0.9, -1.4) -- (0.1, -0.1);
\draw[thick] (-0.9, -1.4) -- (-0.1, -0.1);
\draw (0,2) circle [radius=0.2];
\draw[red] (-1,1) circle [radius=0.2];
\draw[fill] (1,1) circle [radius=0.2];
\draw (0,1) circle [radius=0.2];
\draw (-2,0) circle [radius=0.2];
\draw (-1,0) circle [radius=0.2];
\draw (0,0) circle [radius=0.2];
\draw (-1,-1.5) circle [radius=0.2];
\draw (1,-1.5) circle [radius=0.2];
\node[left] at (-1,1) {\footnotesize{$e$}};
\node[below] at (-2,0) {\footnotesize{$e_1$}};
\node[below] at (-1,0) {\footnotesize{$e_2$}};
\node[below] at (0,0) {\footnotesize{$e_3$}};
\draw[thick, ->] (1.5,0.4) -- (3,0.4);
\node[below] at (2.25,0.4) {$\tau_{e}$};
\end{scope}
\begin{scope}[shift={(22,0)}]
\draw[thick] (-0.1, 1.9) -- (-0.9, 1.1);
\draw[thick] (0, 0.1) -- (0,0.9);
\draw[thick] (0, 1.1) -- (0,1.9);
\draw[thick] (-1, 0.1) -- (-1,0.9);
\draw[thick] (0.1, 1.9) -- (0.9, 1.1);
\draw[thick] (-1.1, 0.9) -- (-1.9, 0.1);
\draw[thick] (-0.9, 0.9) -- (-0.1, 0.1);
\draw[thick] (0.9, 0.9) -- (0.1, 0.1);
\draw[thick] (-1.1, -1.4) -- (-1.9, -0.1);
\draw[thick] (0.9, -1.4) -- (0.1, -0.1);
\draw[thick] (-0.9, -1.4) -- (-0.1, -0.1);
\draw (0,2) circle [radius=0.2];
\draw[fill] (-1,1) circle [radius=0.2];
\draw[fill] (1,1) circle [radius=0.2];
\draw (0,1) circle [radius=0.2];
\draw (-2,0) circle [radius=0.2];
\draw (-1,0) circle [radius=0.2];
\draw[red] (0,0) circle [radius=0.2];
\draw (-1,-1.5) circle [radius=0.2];
\draw (1,-1.5) circle [radius=0.2];
\node[left] at (-1,1) {\footnotesize{$e$}};
\node[below] at (-2,0) {\footnotesize{$e_1$}};
\node[below] at (-1,0) {\footnotesize{$e_2$}};
\node[below] at (0,0) {\footnotesize{$e_3$}};
\draw[thick, ->] (1.5,0.4) -- (3,0.4);
\node[below] at (2.25,0.4) {$\tau_{e_3}$};
\end{scope}
\begin{scope}[shift={(27.5,0)}]
\draw[thick] (-0.1, 1.9) -- (-0.9, 1.1);
\draw[thick] (0, 0.1) -- (0,0.9);
\draw[thick] (0, 1.1) -- (0,1.9);
\draw[thick] (-1, 0.1) -- (-1,0.9);
\draw[thick] (0.1, 1.9) -- (0.9, 1.1);
\draw[thick] (-1.1, 0.9) -- (-1.9, 0.1);
\draw[thick] (-0.9, 0.9) -- (-0.1, 0.1);
\draw[thick] (0.9, 0.9) -- (0.1, 0.1);
\draw[thick] (-1.1, -1.4) -- (-1.9, -0.1);
\draw[thick] (0.9, -1.4) -- (0.1, -0.1);
\draw[thick] (-0.9, -1.4) -- (-0.1, -0.1);
\draw (0,2) circle [radius=0.2];
\draw[fill] (-1,1) circle [radius=0.2];
\draw[fill] (1,1) circle [radius=0.2];
\draw (0,1) circle [radius=0.2];
\draw (-2,0) circle [radius=0.2];
\draw[red] (-1,0) circle [radius=0.2];
\draw (0,0) circle [radius=0.2];
\draw (-1,-1.5) circle [radius=0.2];
\draw (1,-1.5) circle [radius=0.2];
\node[left] at (-1,1) {\footnotesize{$e$}};
\node[below] at (-2,0) {\footnotesize{$e_1$}};
\node[below] at (-1,0) {\footnotesize{$e_2$}};
\node[below] at (0,0) {\footnotesize{$e_3$}};
\draw[thick, ->] (1.5,0.4) -- (3,0.4);
\node[below] at (2.25,0.4) {$\tau_{e_2}$};
\end{scope}
\begin{scope}[shift={(33,0)}]
\draw[thick] (-0.1, 1.9) -- (-0.9, 1.1);
\draw[thick] (0, 0.1) -- (0,0.9);
\draw[thick] (0, 1.1) -- (0,1.9);
\draw[thick] (-1, 0.1) -- (-1,0.9);
\draw[thick] (0.1, 1.9) -- (0.9, 1.1);
\draw[thick] (-1.1, 0.9) -- (-1.9, 0.1);
\draw[thick] (-0.9, 0.9) -- (-0.1, 0.1);
\draw[thick] (0.9, 0.9) -- (0.1, 0.1);
\draw[thick] (-1.1, -1.4) -- (-1.9, -0.1);
\draw[thick] (0.9, -1.4) -- (0.1, -0.1);
\draw[thick] (-0.9, -1.4) -- (-0.1, -0.1);
\draw (0,2) circle [radius=0.2];
\draw[fill] (-1,1) circle [radius=0.2];
\draw[fill] (1,1) circle [radius=0.2];
\draw (0,1) circle [radius=0.2];
\draw[red] (-2,0) circle [radius=0.2];
\draw (-1,0) circle [radius=0.2];
\draw (0,0) circle [radius=0.2];
\draw (-1,-1.5) circle [radius=0.2];
\draw (1,-1.5) circle [radius=0.2];
\node[left] at (-1,1) {\footnotesize{$e$}};
\node[below] at (-2,0) {\footnotesize{$e_1$}};
\node[below] at (-1,0) {\footnotesize{$e_2$}};
\node[below] at (0,0) {\footnotesize{$e_3$}};
\draw[thick, ->] (1.5,0.4) -- (3,0.4);
\node[below] at (2.25,0.4) {$\tau_{e_1}$};
\end{scope}
\begin{scope}[shift={(38.5,0)}]
\draw[thick] (-0.1, 1.9) -- (-0.9, 1.1);
\draw[thick] (0, 0.1) -- (0,0.9);
\draw[thick] (0, 1.1) -- (0,1.9);
\draw[thick] (-1, 0.1) -- (-1,0.9);
\draw[thick] (0.1, 1.9) -- (0.9, 1.1);
\draw[thick] (-1.1, 0.9) -- (-1.9, 0.1);
\draw[thick] (-0.9, 0.9) -- (-0.1, 0.1);
\draw[thick] (0.9, 0.9) -- (0.1, 0.1);
\draw[thick] (-1.1, -1.4) -- (-1.9, -0.1);
\draw[thick] (0.9, -1.4) -- (0.1, -0.1);
\draw[thick] (-0.9, -1.4) -- (-0.1, -0.1);
\draw (0,2) circle [radius=0.2];
\draw[fill] (-1,1) circle [radius=0.2];
\draw[fill] (1,1) circle [radius=0.2];
\draw (0,1) circle [radius=0.2];
\draw (-2,0) circle [radius=0.2];
\draw (-1,0) circle [radius=0.2];
\draw (0,0) circle [radius=0.2];
\draw (-1,-1.5) circle [radius=0.2];
\draw (1,-1.5) circle [radius=0.2];
\node[left] at (-1,1) {\footnotesize{$e$}};
\node[below] at (-2,0) {\footnotesize{$e_1$}};
\node[below] at (-1,0) {\footnotesize{$e_2$}};
\node[below] at (0,0) {\footnotesize{$e_3$}};
%\draw[thick, ->] (1.5,0.4) -- (3,0.4);
%\node[below] at (2.25,0.4) {$\tau_{e_3}$};
\end{scope}
\draw[thick, ->] (-0.5,-2) -- (-0.5,-3.5);
\draw[thick, ->] (38,-2) -- (38,-3.5);
\node[left] at (-0.5,-2.75) {$\bfI$};
\node[right] at (38,-2.75) {$\bfI$};
\begin{scope}[shift={(0,-6)}]
\draw[thick] (-0.1, 1.9) -- (-0.9, 1.1);
\draw[thick] (0, 0.1) -- (0,0.9);
\draw[thick] (0, 1.1) -- (0,1.9);
\draw[thick] (-1, 0.1) -- (-1,0.9);
\draw[thick] (0.1, 1.9) -- (0.9, 1.1);
\draw[thick] (-1.1, 0.9) -- (-1.9, 0.1);
\draw[thick] (-0.9, 0.9) -- (-0.1, 0.1);
\draw[thick] (0.9, 0.9) -- (0.1, 0.1);
\draw[thick] (-1.1, -1.4) -- (-1.9, -0.1);
\draw[thick] (0.9, -1.4) -- (0.1, -0.1);
\draw[thick] (-0.9, -1.4) -- (-0.1, -0.1);
\draw (0,2) circle [radius=0.2];
\draw[red] (-1,1) circle [radius=0.2];
\draw[fill] (1,1) circle [radius=0.2];
\draw (0,1) circle [radius=0.2];
\draw[fill] (-2,0) circle [radius=0.2];
\draw[fill] (-1,0) circle [radius=0.2];
\draw[fill] (0,0) circle [radius=0.2];
\draw[fill] (-1,-1.5) circle [radius=0.2];
\draw[fill] (1,-1.5) circle [radius=0.2];
\node[left] at (-1,1) {\footnotesize{$e$}};
\node[below] at (-2,0) {\footnotesize{$e_1$}};
\node[below] at (-1,0) {\footnotesize{$e_2$}};
\node[below] at (0,0) {\footnotesize{$e_3$}};
\draw[thick, ->] (1.5,0.4) -- (36,0.4);
\node[below] at (18.75,0.4) {$t_e$};
\end{scope}
\begin{scope}[shift={(38.5,-6)}]
\draw[thick] (-0.1, 1.9) -- (-0.9, 1.1);
\draw[thick] (0, 0.1) -- (0,0.9);
\draw[thick] (0, 1.1) -- (0,1.9);
\draw[thick] (-1, 0.1) -- (-1,0.9);
\draw[thick] (0.1, 1.9) -- (0.9, 1.1);
\draw[thick] (-1.1, 0.9) -- (-1.9, 0.1);
\draw[thick] (-0.9, 0.9) -- (-0.1, 0.1);
\draw[thick] (0.9, 0.9) -- (0.1, 0.1);
\draw[thick] (-1.1, -1.4) -- (-1.9, -0.1);
\draw[thick] (0.9, -1.4) -- (0.1, -0.1);
\draw[thick] (-0.9, -1.4) -- (-0.1, -0.1);
\draw (0,2) circle [radius=0.2];
\draw[red,fill] (-1,1) circle [radius=0.2];
\draw[fill] (1,1) circle [radius=0.2];
\draw (0,1) circle [radius=0.2];
\draw[fill] (-2,0) circle [radius=0.2];
\draw[fill] (-1,0) circle [radius=0.2];
\draw[fill] (0,0) circle [radius=0.2];
\draw[fill] (-1,-1.5) circle [radius=0.2];
\draw[fill] (1,-1.5) circle [radius=0.2];
\node[left] at (-1,1) {\footnotesize{$e$}};
\node[below] at (-2,0) {\footnotesize{$e_1$}};
\node[below] at (-1,0) {\footnotesize{$e_2$}};
\node[below] at (0,0) {\footnotesize{$e_3$}};
%\draw[thick, ->] (1.5,0.4) -- (36,0.4);
%\node[below] at (18.75,0.4) {$t_e$};
\end{scope}
\end{tikzpicture}
\\\hline
\end{tabular}
\caption{Four examples of Theorem~\ref{thm:t-star} corresponding (in order) to the four cases of the proof.}
\label{fig:4ex}
\end{figure}

\begin{proof}
%Let $I'=t_e(I)$ and $A'=t_e^*(A)$.  We wish to show that $\bfI(A')=I'$.  
We have four cases to consider.  The four examples in Figure~\ref{fig:4ex} correspond in order to the cases in this proof.

\textbf{Case 1: $e\in I$ and $I\sm\{e\}\not\in \calj(P)$.}
Then $t_e(I)=I$ so we wish to show that $t_e^*(A)=A$.  In this case $e$ is not a maximal element of $I$ so there exists a maximal element $y\in I$ for which $e<y$.  Then each $e_i<y$ for $1\leq i\leq k$.  Also $y\in A$ so each of $e_1,\dots,e_k,e$ is not in $A$ and cannot be toggled in.  So $t_e^*(A)=A$.

\textbf{Case 2: $e\in I$ and $I\sm\{e\}\in \calj(P)$.}
Then $e$ is a maximal element of $I$ so $e\in A$ but no $e_i$ covered by $e$ is.  Clearly $e$ is not a maximal element of $t_e(I)=I\sm\{e\}$.  Any $e_i\lessdot e$ is a maximal element of $t_e(I)$ if and only if the only $x > e_i$ in $I$ is $x=e$.  Other than these, the maximal elements of $t_e(I)$ and $I$ are the same.

Applying $\tau_{e_1}\cdots \tau_{e_k}$ to $A$ does nothing because $e\in A$.  Then applying $\tau_e$ to $A$ removes $e$ from $A$.  Then applying $\tau_{e_1}\cdots \tau_{e_k}$ to $A\sm \{e\}$ adds in any $e_i$ for which no $y > e_i$ is in $A\sm \{e\}$.  These are precisely the elements $e_i$ for which the only $x > e_i$ in $I$ is $x=e$.  Thus, $t_e^*(A)$ is the set of maximal elements of $t_e(I)$.

\textbf{Case 3: $e\not\in I$ and $I\cup\{e\}\not\in \calj(P)$.}  Then $t_e(I)=I$ so we wish to show that $t_e^*(A)=A$.  In this case there exists some $e_i\lessdot e$ not in $I$, so in particular this case cannot happen when $e$ is a minimal element of $P$.
Fix such an $e_i$.  Then $e_i\not\in A$.  If there were $y>e_i$ in $A$, then $y$ would be in $I$ and thus $e_i$ would be in $I$, a contradiction.  So no element greater than $e_i$ is in $A$.

Then when applying $\tau_{e_1}\cdots \tau_{e_k}$ to $A$, either $e_i$ gets toggled into the antichain or there is some $x< e_i<e$ that is in $A$.  In either scenario, there exists an element less than $e$ in $\tau_{e_1}\cdots \tau_{e_k}(A)$.  So applying $\tau_e$ leaves $\tau_{e_1}\cdots \tau_{e_k}(A)$ unchanged.  Then applying $\tau_{e_1}\cdots \tau_{e_k}$ again undoes the effect of applying $\tau_{e_1}\cdots \tau_{e_k}$ in the first place.  Thus, $t_e^*(A)=A$.

\textbf{Case 4: $e\not\in I$ and $I\cup\{e\}\in \calj(P)$.}
Then every $e_i$ is in $I$.  Each $e_i$ is either a maximal element of $I$ or less than some $y\not=e$ in $I$.  Also any element of $A$ comparable with $e$ must be one of $e_1,\dots,e_k$. So $e$ is a maximal element of $t_e(I)$, while none of $e_1,\dots,e_k$ are.  Other than these, the maximal elements of $I$ and $t_e(I)$ are identical.
Applying $\tau_{e_1}\cdots \tau_{e_k}$ to $A$ removes any $e_i$ that is in $A$.  However, it does not insert any $e_i$ that is not in $A$ because such an element is less than some $y\in A$.  Thus, $\tau_{e_1}\cdots \tau_{e_k}(A)$ contains no element that is comparable with $e$, so applying $\tau_e$ adds $e$ to the antichain.  Since $e$ is in $\tau_e\tau_{e_1}\cdots \tau_{e_k}(A)$, none of $e_1,\dots,e_k$ can be added to it.  So $t_e^*(A)=A\cup\{e\}\sm\{e_1,\dots,e_k\}$, exactly the set of maximal elements of $t_e(I)$.
\end{proof}

\begin{defn}
Let $S\subseteq P$.  Let $\eta_S:=t_{x_1}t_{x_2}\cdots t_{x_k}$
%and $\eta_S^*:=t_{x_1}^*t_{x_2}^*\cdots t_{x_k}^*$
where $(x_1,x_2,\dots,x_k)$ is a linear extension of the subposet $\{x\in P\;|\;x<y,y\in S\}$.  (In the special case where every element of $S$ is minimal in $P$, $\eta_S$ is the identity.)  For $e\in P$, we write $\eta_e:=\eta_{\{e\}}$.
\end{defn}

\begin{remark}\label{rem:etienne-84}
Any two linear extensions of a poset differ by a sequence of swaps between adjacent incomparable elements~\cite{etienne-84}.  So $\eta_S$ is well-defined
(since Proposition~\ref{prop:J-toggle-inv-commute}
shows that such swaps do not change the product $t_{x_1}t_{x_2}\cdots t_{x_k}$).
\end{remark}

\begin{defn}
For $e\in P$, define $\tau_e^*\in\tog_\calj(P)$ as $\tau_e^* := \eta_e t_e \eta_e^{-1}$.
\end{defn}

\begin{thm}\label{thm:tau-star}
Let $A\in \cala(P)$, $e\in P$, and $I=\bfI(A)$ be the order ideal generated by $A$.  Then the order ideal $\bfI(\tau_e(A))$ generated by $\tau_e(A)$ is $\tau_e^*(I)$.  That is, the following diagram commutes.

\begin{center}
\begin{tikzpicture}
\node at (0,1.8) {$\cala(P)$};
\node at (0,0) {$\calj(P)$};
\node at (3.25,1.8) {$\cala(P)$};
\node at (3.25,0) {$\calj(P)$};
\draw[semithick, ->] (0,1.3) -- (0,0.5);
\node[left] at (0,0.9) {$\bfI$};
\draw[semithick, ->] (0.7,0) -- (2.5,0);
\node[below] at (1.5,0) {$\tau_e^*$};
\draw[semithick, ->] (0.7,1.8) -- (2.5,1.8);
\node[above] at (1.5,1.8) {$\tau_e$};
\draw[semithick, ->] (3.25,1.3) -- (3.25,0.5);
\node[right] at (3.25,0.9) {$\bfI$};
\end{tikzpicture}
\end{center}
\end{thm}

\begin{ex}
In the product of two chains poset $P=[3]\times[2]$ given by
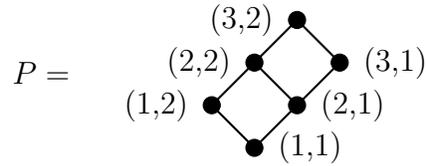
\begin{center}
\begin{tikzpicture}[scale=0.567]
\draw[thick] (-0.1, 1.9) -- (-0.9, 1.1);
\draw[thick] (0.1, 1.9) -- (0.9, 1.1);
\draw[thick] (-1.1, 0.9) -- (-1.9, 0.1);
\draw[thick] (-0.9, 0.9) -- (-0.1, 0.1);
\draw[thick] (0.9, 0.9) -- (0.1, 0.1);
\draw[thick] (-1.1, -0.9) -- (-1.9, -0.1);
\draw[thick] (-0.9, -0.9) -- (-0.1, -0.1);
\draw[fill] (0,2) circle [radius=0.2];
\draw[fill] (-1,1) circle [radius=0.2];
\draw[fill] (1,1) circle [radius=0.2];
\draw[fill] (-2,0) circle [radius=0.2];
\draw[fill] (0,0) circle [radius=0.2];
\draw[fill] (-1,-1) circle [radius=0.2];
\node[right] at (-0.7,-1) {(1,1)};
\node[right] at (0.3,0) {(2,1)};
\node[right] at (1.3,1) {(3,1)};
\node[left] at (-2.3,0) {(1,2)};
\node[left] at (-1.3,1) {(2,2)};
\node[left] at (-0.3,2) {(3,2)};
\node at (-6,0.75) {$P=$};
\end{tikzpicture}
\end{center}
we have $\eta_{(2,2)}=t_{(1,1)}t_{(1,2)}t_{(2,1)}$, so $\tau_{(2,2)}^*=t_{(1,1)}t_{(1,2)}t_{(2,1)}t_{(2,2)}t_{(2,1)}t_{(1,2)}t_{(1,1)}$.  An illustration of Theorem~\ref{thm:tau-star} for an antichain of this poset is below.

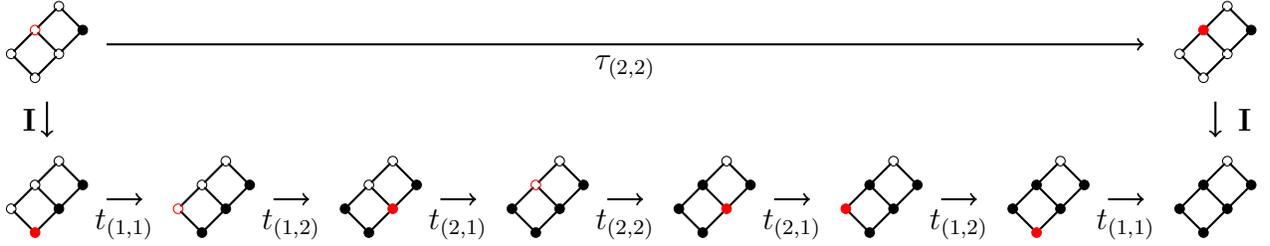
\begin{center}
\begin{tikzpicture}[scale=0.317]
\begin{scope}
\draw[thick] (-0.1, 1.9) -- (-0.9, 1.1);
\draw[thick] (0.1, 1.9) -- (0.9, 1.1);
\draw[thick] (-1.1, 0.9) -- (-1.9, 0.1);
\draw[thick] (-0.9, 0.9) -- (-0.1, 0.1);
\draw[thick] (0.9, 0.9) -- (0.1, 0.1);
\draw[thick] (-1.1, -0.9) -- (-1.9, -0.1);
\draw[thick] (-0.9, -0.9) -- (-0.1, -0.1);
\draw (0,2) circle [radius=0.2];
\draw (-1,1) circle [radius=0.2];
\draw[fill] (1,1) circle [radius=0.2];
\draw (-2,0) circle [radius=0.2];
\draw[fill] (0,0) circle [radius=0.2];
\draw[red,fill] (-1,-1) circle [radius=0.2];
\draw[thick, ->] (2,0.4) -- (3.5,0.4);
\node[below] at (2.75,0.4) {$t_{(1,1)}$};
\end{scope}
\begin{scope}[shift={(7,0)}]
\draw[thick] (-0.1, 1.9) -- (-0.9, 1.1);
\draw[thick] (0.1, 1.9) -- (0.9, 1.1);
\draw[thick] (-1.1, 0.9) -- (-1.9, 0.1);
\draw[thick] (-0.9, 0.9) -- (-0.1, 0.1);
\draw[thick] (0.9, 0.9) -- (0.1, 0.1);
\draw[thick] (-1.1, -0.9) -- (-1.9, -0.1);
\draw[thick] (-0.9, -0.9) -- (-0.1, -0.1);
\draw (0,2) circle [radius=0.2];
\draw (-1,1) circle [radius=0.2];
\draw[fill] (1,1) circle [radius=0.2];
\draw[red] (-2,0) circle [radius=0.2];
\draw[fill] (0,0) circle [radius=0.2];
\draw[fill] (-1,-1) circle [radius=0.2];
\draw[thick, ->] (2,0.4) -- (3.5,0.4);
\node[below] at (2.75,0.4) {$t_{(1,2)}$};
\end{scope}
\begin{scope}[shift={(14,0)}]
\draw[thick] (-0.1, 1.9) -- (-0.9, 1.1);
\draw[thick] (0.1, 1.9) -- (0.9, 1.1);
\draw[thick] (-1.1, 0.9) -- (-1.9, 0.1);
\draw[thick] (-0.9, 0.9) -- (-0.1, 0.1);
\draw[thick] (0.9, 0.9) -- (0.1, 0.1);
\draw[thick] (-1.1, -0.9) -- (-1.9, -0.1);
\draw[thick] (-0.9, -0.9) -- (-0.1, -0.1);
\draw (0,2) circle [radius=0.2];
\draw (-1,1) circle [radius=0.2];
\draw[fill] (1,1) circle [radius=0.2];
\draw[fill] (-2,0) circle [radius=0.2];
\draw[red,fill] (0,0) circle [radius=0.2];
\draw[fill] (-1,-1) circle [radius=0.2];
\draw[thick, ->] (2,0.4) -- (3.5,0.4);
\node[below] at (2.75,0.4) {$t_{(2,1)}$};
\end{scope}
\begin{scope}[shift={(21,0)}]
\draw[thick] (-0.1, 1.9) -- (-0.9, 1.1);
\draw[thick] (0.1, 1.9) -- (0.9, 1.1);
\draw[thick] (-1.1, 0.9) -- (-1.9, 0.1);
\draw[thick] (-0.9, 0.9) -- (-0.1, 0.1);
\draw[thick] (0.9, 0.9) -- (0.1, 0.1);
\draw[thick] (-1.1, -0.9) -- (-1.9, -0.1);
\draw[thick] (-0.9, -0.9) -- (-0.1, -0.1);
\draw (0,2) circle [radius=0.2];
\draw[red] (-1,1) circle [radius=0.2];
\draw[fill] (1,1) circle [radius=0.2];
\draw[fill] (-2,0) circle [radius=0.2];
\draw[fill] (0,0) circle [radius=0.2];
\draw[fill] (-1,-1) circle [radius=0.2];
\draw[thick, ->] (2,0.4) -- (3.5,0.4);
\node[below] at (2.75,0.4) {$t_{(2,2)}$};
\end{scope}
\begin{scope}[shift={(28,0)}]
\draw[thick] (-0.1, 1.9) -- (-0.9, 1.1);
\draw[thick] (0.1, 1.9) -- (0.9, 1.1);
\draw[thick] (-1.1, 0.9) -- (-1.9, 0.1);
\draw[thick] (-0.9, 0.9) -- (-0.1, 0.1);
\draw[thick] (0.9, 0.9) -- (0.1, 0.1);
\draw[thick] (-1.1, -0.9) -- (-1.9, -0.1);
\draw[thick] (-0.9, -0.9) -- (-0.1, -0.1);
\draw (0,2) circle [radius=0.2];
\draw[fill] (-1,1) circle [radius=0.2];
\draw[fill] (1,1) circle [radius=0.2];
\draw[fill] (-2,0) circle [radius=0.2];
\draw[red,fill] (0,0) circle [radius=0.2];
\draw[fill] (-1,-1) circle [radius=0.2];
\draw[thick, ->] (2,0.4) -- (3.5,0.4);
\node[below] at (2.75,0.4) {$t_{(2,1)}$};
\end{scope}
\begin{scope}[shift={(35,0)}]
\draw[thick] (-0.1, 1.9) -- (-0.9, 1.1);
\draw[thick] (0.1, 1.9) -- (0.9, 1.1);
\draw[thick] (-1.1, 0.9) -- (-1.9, 0.1);
\draw[thick] (-0.9, 0.9) -- (-0.1, 0.1);
\draw[thick] (0.9, 0.9) -- (0.1, 0.1);
\draw[thick] (-1.1, -0.9) -- (-1.9, -0.1);
\draw[thick] (-0.9, -0.9) -- (-0.1, -0.1);
\draw (0,2) circle [radius=0.2];
\draw[fill] (-1,1) circle [radius=0.2];
\draw[fill] (1,1) circle [radius=0.2];
\draw[red,fill] (-2,0) circle [radius=0.2];
\draw[fill] (0,0) circle [radius=0.2];
\draw[fill] (-1,-1) circle [radius=0.2];
\draw[thick, ->] (2,0.4) -- (3.5,0.4);
\node[below] at (2.75,0.4) {$t_{(1,2)}$};
\end{scope}
\begin{scope}[shift={(42,0)}]
\draw[thick] (-0.1, 1.9) -- (-0.9, 1.1);
\draw[thick] (0.1, 1.9) -- (0.9, 1.1);
\draw[thick] (-1.1, 0.9) -- (-1.9, 0.1);
\draw[thick] (-0.9, 0.9) -- (-0.1, 0.1);
\draw[thick] (0.9, 0.9) -- (0.1, 0.1);
\draw[thick] (-1.1, -0.9) -- (-1.9, -0.1);
\draw[thick] (-0.9, -0.9) -- (-0.1, -0.1);
\draw (0,2) circle [radius=0.2];
\draw[fill] (-1,1) circle [radius=0.2];
\draw[fill] (1,1) circle [radius=0.2];
\draw[fill] (-2,0) circle [radius=0.2];
\draw[fill] (0,0) circle [radius=0.2];
\draw[red,fill] (-1,-1) circle [radius=0.2];
\draw[thick, ->] (2,0.4) -- (3.5,0.4);
\node[below] at (2.75,0.4) {$t_{(1,1)}$};
\end{scope}
\begin{scope}[shift={(49,0)}]
\draw[thick] (-0.1, 1.9) -- (-0.9, 1.1);
\draw[thick] (0.1, 1.9) -- (0.9, 1.1);
\draw[thick] (-1.1, 0.9) -- (-1.9, 0.1);
\draw[thick] (-0.9, 0.9) -- (-0.1, 0.1);
\draw[thick] (0.9, 0.9) -- (0.1, 0.1);
\draw[thick] (-1.1, -0.9) -- (-1.9, -0.1);
\draw[thick] (-0.9, -0.9) -- (-0.1, -0.1);
\draw (0,2) circle [radius=0.2];
\draw[fill] (-1,1) circle [radius=0.2];
\draw[fill] (1,1) circle [radius=0.2];
\draw[fill] (-2,0) circle [radius=0.2];
\draw[fill] (0,0) circle [radius=0.2];
\draw[fill] (-1,-1) circle [radius=0.2];
\end{scope}
\draw[thick, ->] (-0.5,4.5) -- (-0.5,3);
\draw[thick, ->] (48.5,4.5) -- (48.5,3);
\node[left] at (-0.5,3.75) {$\bfI$};
\node[right] at (49,3.75) {$\bfI$};
\draw[thick, ->] (2,6.9) -- (45.5,6.9);
\node[below] at (23.75,6.9) {$\tau_{(2,2)}$};
\begin{scope}[shift={(0,6.5)}]
\draw[thick] (-0.1, 1.9) -- (-0.9, 1.1);
\draw[thick] (0.1, 1.9) -- (0.9, 1.1);
\draw[thick] (-1.1, 0.9) -- (-1.9, 0.1);
\draw[thick] (-0.9, 0.9) -- (-0.1, 0.1);
\draw[thick] (0.9, 0.9) -- (0.1, 0.1);
\draw[thick] (-1.1, -0.9) -- (-1.9, -0.1);
\draw[thick] (-0.9, -0.9) -- (-0.1, -0.1);
\draw (0,2) circle [radius=0.2];
\draw[red] (-1,1) circle [radius=0.2];
\draw[fill] (1,1) circle [radius=0.2];
\draw (-2,0) circle [radius=0.2];
\draw (0,0) circle [radius=0.2];
\draw (-1,-1) circle [radius=0.2];
\end{scope}
\begin{scope}[shift={(49,6.5)}]
\draw[thick] (-0.1, 1.9) -- (-0.9, 1.1);
\draw[thick] (0.1, 1.9) -- (0.9, 1.1);
\draw[thick] (-1.1, 0.9) -- (-1.9, 0.1);
\draw[thick] (-0.9, 0.9) -- (-0.1, 0.1);
\draw[thick] (0.9, 0.9) -- (0.1, 0.1);
\draw[thick] (-1.1, -0.9) -- (-1.9, -0.1);
\draw[thick] (-0.9, -0.9) -- (-0.1, -0.1);
\draw (0,2) circle [radius=0.2];
\draw[red,fill] (-1,1) circle [radius=0.2];
\draw[fill] (1,1) circle [radius=0.2];
\draw (-2,0) circle [radius=0.2];
\draw (0,0) circle [radius=0.2];
\draw (-1,-1) circle [radius=0.2];
\end{scope}
\end{tikzpicture}
\end{center}
\end{ex}

To prove Theorem~\ref{thm:tau-star}, we first need a lemma.
The proof of Lemma~\ref{lem:tau-star} and Theorem~\ref{thm:tau-star} will both be purely at the group-theoretic level,
using properties of $\tog_\calj(P)$ and $\tog_\cala(P)$ proved earlier in the paper, and not the definitions of toggles themselves.
This will allow us to use the same proof in the generalization to the piecewise-linear setting after proving the analogue of Theorem~\ref{thm:t-star} and commutativity of toggles.  This will be Theorem~\ref{thm:iso-cpl}.

\begin{lemma}\label{lem:tau-star}
Let $e_1,\dots,e_k$ be pairwise incomparable elements of $P$.  Then for $1\leq i\leq k$,
$$\tau_{e_1}^*\tau_{e_2}^*\dots \tau_{e_i}^* = \eta_{\{e_1,\dots,e_i\}} t_{e_1}t_{e_2}\dots t_{e_i} \eta_{\{e_1,\dots,e_i\}}^{-1}.$$
\end{lemma}

\begin{proof}
This claim is true by definition for $i=1$ and we proceed inductively.  Suppose it is true for some given $i\leq k-1$.  Let
\begin{itemize}
\item $x_1,\dots,x_a$ be the elements that are both less than $e_{i+1}$ and less than at least one of $e_1,\dots,e_i$,
\item $y_1,\dots,y_b$ be the elements that are less than at least one of $e_1,\dots,e_i$ but not less than $e_{i+1}$,
\item $z_1,\dots,z_c$ be the elements that are less than $e_{i+1}$ but not less than any of $e_1,\dots,e_i$.
\end{itemize}
Clearly, it is possible for one or more of the sets $\{x_1,\dots,x_a\}$, $\{y_1,\dots,y_b\}$, and $\{z_1,\dots,z_c\}$ to be empty.  For example, if $b=0$, then the product $T_{y_1}\cdots t_{y_b}$ is just the identity.

Note than none of $y_1,\dots,y_b$ are less than any of $x_1,\dots,x_a$ because any element less than some $x_j$ is automatically less than $e_{i+1}$.  By similar reasoning, none of $z_1,\dots,z_c$ are less than any of $x_1,\dots,x_a$ either.  Also any pair $y_m,z_n$ are incomparable, because $z_n \leq y_m$ would imply $y_m$ is less than some $e_j$, while $y_m \leq z_n$ would imply $z_n<e_{i+1}$.  By transitivity and the pairwise incomparability of $e_1,\dots,e_{i+1}$, each $y_m$ is incomparable with $e_{i+1}$, and each $z_m$ is incomparable with any of $e_1,\dots,e_i$.

We will pick the indices so that $(x_1,\dots, x_a)$, $(y_1,\dots,y_b)$, and $(z_1,\dots,z_c)$ are linear extensions of the subposets $\{x_1,\dots,x_a\}$, $\{y_1,\dots,y_b\}$, and $\{z_1,\dots,z_c\}$, respectively.  Then we have the following
\begin{itemize}
\item $(x_1,\dots,x_a,y_1,\dots,y_b)$ is a linear extension of $\big\{p\in P\;|\; p<q,q\in\{e_1,\dots, e_i\}\big\}$.
\begin{itemize}
\item This yields $\eta_{\{e_1,\dots,e_i\}}=t_{x_1}\cdots t_{x_a}t_{y_1}\cdots t_{y_b}$.
\end{itemize}
\item $(x_1,\dots,x_a,z_1,\dots,z_c)$ is a linear extension of $\{p\in P \;|\; p<e_{i+1}\}$.
\begin{itemize}
\item This yields $\eta_{e_{i+1}}=t_{x_1}\cdots t_{x_a}t_{z_1}\cdots t_{z_c}$.
\end{itemize}
\item $(x_1,\dots,x_a,y_1,\dots,y_b,z_1,\dots,z_c)$ and 
$(x_1,\dots,x_a,z_1,\dots,z_c,y_1,\dots,y_b)$ are both linear extensions of $\big\{p\in P \;|\; p<q,q\in\{e_1,\dots, e_{i+1}\}\big\}$.
\begin{itemize}
\item This yields $\eta_{\{e_1,\dots,e_{i+1}\}}=t_{x_1}\cdots t_{x_a}t_{y_1}\cdots t_{y_b}t_{z_1}\cdots t_{z_c}$.
\end{itemize}
\end{itemize}

Using the induction hypothesis,
\begin{align*}
&  \tau_{e_1}^*\dots \tau_{e_i}^* \tau_{e_{i+1}}^* \\
&=
\eta_{\{e_1,\dots,e_i\}} t_{e_1}\dots t_{e_i} \eta_{\{e_1,\dots,e_i\}}^{-1}\eta_{e_{i+1}} t_{e_{i+1}} \eta_{e_{i+1}}^{-1}\\
&=
t_{x_1}\cdots t_{x_a} t_{y_1} \cdots t_{y_b} t_{e_1} \cdots t_{e_i} t_{y_b} \cdots t_{y_1} t_{x_a}\cdots t_{x_1} t_{x_1}\cdots t_{x_a} t_{z_1} \cdots t_{z_c} t_{e_{i+1}} t_{z_c} \cdots t_{z_1} t_{x_a}\cdots t_{x_1}\\
&=
t_{x_1}\cdots t_{x_a} t_{y_1} \cdots t_{y_b} t_{e_1} \cdots t_{e_i} t_{y_b} \cdots t_{y_1} t_{z_1} \cdots t_{z_c} t_{e_{i+1}} t_{z_c} \cdots t_{z_1} t_{x_a}\cdots t_{x_1}\\
&=
t_{x_1}\cdots t_{x_a} t_{y_1} \cdots t_{y_b} t_{e_1} \cdots t_{e_i} t_{z_1} \cdots t_{z_c} t_{e_{i+1}} t_{z_c} \cdots t_{z_1} t_{y_b} \cdots t_{y_1} t_{x_a}\cdots t_{x_1}\\
&=
t_{x_1}\cdots t_{x_a} t_{y_1} \cdots t_{y_b} t_{z_1} \cdots t_{z_c} t_{e_1} \cdots t_{e_i} t_{e_{i+1}} t_{z_c} \cdots t_{z_1} t_{y_b} \cdots t_{y_1} t_{x_a}\cdots t_{x_1}\\
&=
\eta_{\{e_1,\dots,e_{i+1}\}} t_{e_1}\dots t_{e_i}
t_{e_{i+1}}\eta_{\{e_1,\dots,e_{i+1}\}}^{-1}
\end{align*}
where each commutation above is between toggles for pairwise incomparable elements.
\end{proof}

We are now ready to prove Theorem~\ref{thm:tau-star}.

\begin{proof}[Proof of Theorem~\ref{thm:tau-star}]
We use induction on $e$.  If $e$ is a minimal element of $P$, then $\tau_e^*=t_e$, so the diagram commutes by Theorem~\ref{thm:t-star}.

Now suppose $e$ is not minimal.  Let $e_1,\dots,e_k$ be the elements of $P$ covered by $e$, and suppose that the theorem is true for every $e_i$.  That is, for every antichain $A$ with $I=\bfI(A)$, the order ideal generated by $\tau_{e_i}(A)$ is $\tau_{e_i}^*(I)$.
Then the order ideal generated by $\tau_{e_1}\tau_{e_2}\dots \tau_{e_k}(A)$ is
$\tau_{e_1}^*\tau_{e_2}^*\dots \tau_{e_k}^*(I)
= \eta_{\{e_1,\dots,e_k\}} t_{e_1}t_{e_2}\dots t_{e_k} \eta_{\{e_1,\dots,e_k\}}^{-1}(I)$ by Lemma~\ref{lem:tau-star}.

From the definition of $t_e^*$, it follows that $\tau_{e_1}\tau_{e_2}\cdots \tau_{e_k} t_e^* \tau_{e_1}\tau_{e_2}\cdots \tau_{e_k} 
=\tau_e$.
Then the order ideal generated by $\tau_e(A)=\tau_{e_1}\tau_{e_2}\cdots \tau_{e_k} t_e^* \tau_{e_1}\tau_{e_2}\cdots \tau_{e_k} (A)$ is 
$$\eta_{\{e_1,\dots,e_k\}} t_{e_1}t_{e_2}\cdots t_{e_k} \eta_{\{e_1,\dots,e_k\}}^{-1} t_e \eta_{\{e_1,\dots,e_k\}} t_{e_1}t_{e_2}\cdots t_{e_k} \eta_{\{e_1,\dots,e_k\}}^{-1}(I)$$ by Theorem~\ref{thm:t-star} (for $t_e^*$) and the induction hypothesis (for $\tau_{e_1}\tau_{e_2}\cdots \tau_{e_k}$).  Thus, it suffices to show that
\begin{equation}\label{eq:tau-star}
\eta_{\{e_1,\dots,e_k\}} t_{e_1}t_{e_2}\cdots t_{e_k} \eta_{\{e_1,\dots,e_k\}}^{-1} t_e \eta_{\{e_1,\dots,e_k\}} t_{e_1}t_{e_2}\cdots t_{e_k} \eta_{\{e_1,\dots,e_k\}}^{-1}=\tau_e^*.
\end{equation}
The toggles in the product $\eta_{\{e_1,\dots,e_k\}}$ correspond to elements strictly less than $e_1,\dots,e_k$; none of these cover nor are covered by $e$.  Thus by Proposition~\ref{prop:J-toggle-inv-commute}, we can commute $t_e$ with $\eta_{\{e_1,\dots,e_k\}}$ on the left side of~(\ref{eq:tau-star}) and then cancel $\eta_{\{e_1,\dots,e_k\}}^{-1} \eta_{\{e_1,\dots,e_k\}}$.  Also, since $e_1,\dots,e_k$ are pairwise incomparable, we can commute $t_{e_1},\dots, t_{e_k}$.  Thus the left side of~(\ref{eq:tau-star}) is
$\eta_{\{e_1,\dots,e_k\}} t_{e_1}t_{e_2}\cdots t_{e_k} t_e t_{e_k} \cdots t_{e_2}t_{e_1} \eta_{\{e_1,\dots,e_k\}}^{-1}$.
Note that $$\{x\in P \;|\; x<e\}=\{x\in P \;|\; x<y ,y\in\{e_1,\dots,e_k\}\}\cup \{e_1,\dots,e_k\}$$ where the union is disjoint and that $e_1,\dots,e_k$ are maximal elements of this set.  Thus for any linear extension $(x_1,\dots,x_n)$ of $\{x\in P \;|\; x<y ,y\in\{e_1,\dots,e_k\}\}$, a linear extension of
$\{x\in P \;|\; x<e\}$ is
$(x_1,\dots,x_n,e_1,\dots,e_k)$.  So $\eta_{\{e_1,\dots,e_k\}} t_{e_1}t_{e_2}\cdots t_{e_k}
=\eta_e$ which means the left side of~(\ref{eq:tau-star}) is $\eta_e t_e \eta_e^{-1}= \tau_e^*$, same as the right side.
\end{proof}

The following is a corollary of Theorems~\ref{thm:t-star} and~\ref{thm:tau-star}.

\begin{cor}\label{cor:iso}
There is an isomorphism from $\tog_\cala(P)$ to $\tog_\calj(P)$ given by $\tau_e\mapsto \tau_e^*$, with inverse given by $t_e\mapsto t_e^*$.
\end{cor}

\begin{remark}\label{rem:indep-sets}
Striker has proven that toggle groups on many families of subsets are either symmetric or alternating groups, including independent sets of connected graphs~\cite[\S3.6]{strikergentog}.  An \textbf{independent set} of a graph is a subset of the vertices, for which no two are connected by an edge.  Antichains of $P$ are the same as independent sets of the \textbf{comparability graph} of $P$, in which two elements are connected by an edge if they are comparable (different from the Hasse diagram that only includes \textit{cover} relations).  So any result that holds in general for toggling independent sets of graphs also does for toggling
antichains\footnote{And similarly chains of posets are the independent sets of the \textit{incomparability graph} in which two elements are connected by an edge if they are incomparable.  So any result that holds in general for toggling independent sets also holds for toggling chains.},
but not necessarily vice versa, since it is straightforward to show that e.g.\ a cycle graph with five vertices is not the comparability graph for any poset.
\end{remark}

The following proposition explains that we can state $\row_\cala$ by performing antichain toggles at every element, but in the \emph{opposite} order as that of $\row_\calj$ in Proposition~\ref{prop:row-toggles}.

\begin{prop}\label{prop:row-toggles-anti}
Let $(x_1,x_2,\dots,x_n)$ be any linear extension of a finite poset $P$.  Then $\row_\cala=\tau_{x_n}\cdots \tau_{x_2} \tau_{x_1}$.
\end{prop}

Like the proofs of Theorem~\ref{thm:tau-star} and Lemma~\ref{lem:tau-star},
we could prove this proposition algebraically using Theorems~\ref{thm:t-star} and~\ref{thm:tau-star}, which is what we will do in Section~\ref{sec:cpl} for the piecewise-linear generalization (Theorem~\ref{thm:row-C}).  However, the following is a much more elegant proof.

\begin{proof}
Let $A=\{a_1,\dots,a_k\}$ be an antichain.
Recall that $\row_\cala(A)$ is the set of minimal elements of the complement of the order ideal generated by $A$.  Let us consider what happens when we apply $\tau_{x_i}$ in the product $\tau_{x_n}\cdots \tau_{x_2} \tau_{x_1}$.
\begin{itemize}
\item If $x_i<a_j\in A$, then $\tau_{x_i}$ is performed before $\tau_{a_j}$ so $\tau_{x_i}$ cannot add $x_i$ to the antichain.
\item If $x_i\in A$, then $\tau_{x_i}$ removes $x_i$ from $A$.
\item Otherwise, $x_i\in P\sm \bfI(A)$.  In this case $\tau_{x_i}$ is performed after any element of $A$ less than $x_i$ (if any) has been toggled out.  If $x_i$ is a minimal element of $P\sm \bfI(A)$ (i.e., $x_i\in\row_\cala(A)$), then $\tau_{x_i}$ adds $x_i$ to the antichain.  If $x_i$ is not a minimal element of $P\sm \bfI(A)$, then when it is time to toggle $\tau_{x_i}$, some $z\in \row_\cala(A)$ with $z<x_i$ is in the antichain, so we cannot add $x_i$.
\end{itemize}
Thus, $\tau_{x_n}\cdots \tau_{x_2} \tau_{x_1}(A)=\row_\cala(A)$.
\end{proof}

\begin{example}\label{ex:row-toggles-anti}
For the poset of Example~\ref{ex:a3-row}, with elements as named below, $(a,b,c,d,e,f)$ is a linear extension.
We show the effect of applying $\tau_f \tau_e \tau_d \tau_c \tau_b \tau_a$ to the antichain considered in Example~\ref{ex:a3-row}.  Notice that the outcome is the same antichain we obtained by the three step process, demonstrating Proposition~\ref{prop:row-toggles-anti}.

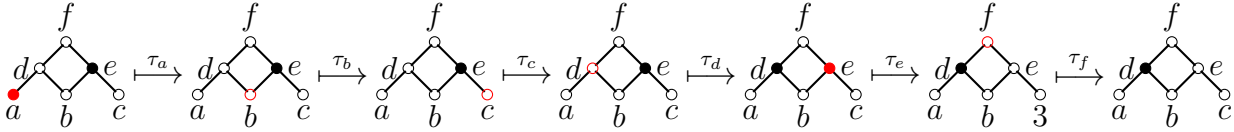
\begin{center}
\begin{tikzpicture}[scale=0.35]
\begin{scope}[shift={(0,-4)}]
\draw[thick] (-0.1, 1.9) -- (-0.9, 1.1);
\draw[thick] (0.1, 1.9) -- (0.9, 1.1);
\draw[thick] (-1.1, 0.9) -- (-1.9, 0.1);
\draw[thick] (-0.9, 0.9) -- (-0.1, 0.1);
\draw[thick] (0.9, 0.9) -- (0.1, 0.1);
\draw[thick] (1.1, 0.9) -- (1.9, 0.1);
\draw (0,2) circle [radius=0.2];
\draw (-1,1) circle [radius=0.2];
\draw[fill] (1,1) circle [radius=0.2];
\draw[red,fill] (-2,0) circle [radius=0.2];
\draw (0,0) circle [radius=0.2];
\draw (2,0) circle [radius=0.2];
\node[above] at (0,2) {$f$};
\node[left] at (-1,1) {$d$};
\node[right] at (1,1) {$e$};
\node[below] at (-2,0) {$a$};
\node[below] at (0,0) {$b$};
\node[below] at (2,0) {$c$};
\end{scope}
\node at (3.5,-3) {$\stackrel{\tau_a}{\longmapsto}$};
\begin{scope}[shift={(7,-4)}]
\draw[thick] (-0.1, 1.9) -- (-0.9, 1.1);
\draw[thick] (0.1, 1.9) -- (0.9, 1.1);
\draw[thick] (-1.1, 0.9) -- (-1.9, 0.1);
\draw[thick] (-0.9, 0.9) -- (-0.1, 0.1);
\draw[thick] (0.9, 0.9) -- (0.1, 0.1);
\draw[thick] (1.1, 0.9) -- (1.9, 0.1);
\draw (0,2) circle [radius=0.2];
\draw (-1,1) circle [radius=0.2];
\draw[fill] (1,1) circle [radius=0.2];
\draw (-2,0) circle [radius=0.2];
\draw[red] (0,0) circle [radius=0.2];
\draw (2,0) circle [radius=0.2];
\node[above] at (0,2) {$f$};
\node[left] at (-1,1) {$d$};
\node[right] at (1,1) {$e$};
\node[below] at (-2,0) {$a$};
\node[below] at (0,0) {$b$};
\node[below] at (2,0) {$c$};
\end{scope}
\node at (10.5,-3) {$\stackrel{\tau_b}{\longmapsto}$};
\begin{scope}[shift={(14,-4)}]
\draw[thick] (-0.1, 1.9) -- (-0.9, 1.1);
\draw[thick] (0.1, 1.9) -- (0.9, 1.1);
\draw[thick] (-1.1, 0.9) -- (-1.9, 0.1);
\draw[thick] (-0.9, 0.9) -- (-0.1, 0.1);
\draw[thick] (0.9, 0.9) -- (0.1, 0.1);
\draw[thick] (1.1, 0.9) -- (1.9, 0.1);
\draw (0,2) circle [radius=0.2];
\draw (-1,1) circle [radius=0.2];
\draw[fill] (1,1) circle [radius=0.2];
\draw (-2,0) circle [radius=0.2];
\draw (0,0) circle [radius=0.2];
\draw[red] (2,0) circle [radius=0.2];
\node[above] at (0,2) {$f$};
\node[left] at (-1,1) {$d$};
\node[right] at (1,1) {$e$};
\node[below] at (-2,0) {$a$};
\node[below] at (0,0) {$b$};
\node[below] at (2,0) {$c$};
\end{scope}
\node at (17.5,-3) {$\stackrel{\tau_c}{\longmapsto}$};
\begin{scope}[shift={(21,-4)}]
\draw[thick] (-0.1, 1.9) -- (-0.9, 1.1);
\draw[thick] (0.1, 1.9) -- (0.9, 1.1);
\draw[thick] (-1.1, 0.9) -- (-1.9, 0.1);
\draw[thick] (-0.9, 0.9) -- (-0.1, 0.1);
\draw[thick] (0.9, 0.9) -- (0.1, 0.1);
\draw[thick] (1.1, 0.9) -- (1.9, 0.1);
\draw (0,2) circle [radius=0.2];
\draw[red] (-1,1) circle [radius=0.2];
\draw[fill] (1,1) circle [radius=0.2];
\draw (-2,0) circle [radius=0.2];
\draw (0,0) circle [radius=0.2];
\draw (2,0) circle [radius=0.2];
\node[above] at (0,2) {$f$};
\node[left] at (-1,1) {$d$};
\node[right] at (1,1) {$e$};
\node[below] at (-2,0) {$a$};
\node[below] at (0,0) {$b$};
\node[below] at (2,0) {$c$};
\end{scope}
\node at (24.5,-3) {$\stackrel{\tau_d}{\longmapsto}$};
\begin{scope}[shift={(28,-4)}]
\draw[thick] (-0.1, 1.9) -- (-0.9, 1.1);
\draw[thick] (0.1, 1.9) -- (0.9, 1.1);
\draw[thick] (-1.1, 0.9) -- (-1.9, 0.1);
\draw[thick] (-0.9, 0.9) -- (-0.1, 0.1);
\draw[thick] (0.9, 0.9) -- (0.1, 0.1);
\draw[thick] (1.1, 0.9) -- (1.9, 0.1);
\draw (0,2) circle [radius=0.2];
\draw[fill] (-1,1) circle [radius=0.2];
\draw[red,fill] (1,1) circle [radius=0.2];
\draw (-2,0) circle [radius=0.2];
\draw (0,0) circle [radius=0.2];
\draw (2,0) circle [radius=0.2];
\node[above] at (0,2) {$f$};
\node[left] at (-1,1) {$d$};
\node[right] at (1,1) {$e$};
\node[below] at (-2,0) {$a$};
\node[below] at (0,0) {$b$};
\node[below] at (2,0) {$c$};
\end{scope}
\node at (31.5,-3) {$\stackrel{\tau_e}{\longmapsto}$};
\begin{scope}[shift={(35,-4)}]
\draw[thick] (-0.1, 1.9) -- (-0.9, 1.1);
\draw[thick] (0.1, 1.9) -- (0.9, 1.1);
\draw[thick] (-1.1, 0.9) -- (-1.9, 0.1);
\draw[thick] (-0.9, 0.9) -- (-0.1, 0.1);
\draw[thick] (0.9, 0.9) -- (0.1, 0.1);
\draw[thick] (1.1, 0.9) -- (1.9, 0.1);
\draw[red] (0,2) circle [radius=0.2];
\draw[fill] (-1,1) circle [radius=0.2];
\draw (1,1) circle [radius=0.2];
\draw (-2,0) circle [radius=0.2];
\draw (0,0) circle [radius=0.2];
\draw (2,0) circle [radius=0.2];
\node[above] at (0,2) {$f$};
\node[left] at (-1,1) {$d$};
\node[right] at (1,1) {$e$};
\node[below] at (-2,0) {$a$};
\node[below] at (0,0) {$b$};
\node[below] at (2,0) {3};
\end{scope}
\node at (38.5,-3) {$\stackrel{\tau_f}{\longmapsto}$};
\begin{scope}[shift={(42,-4)}]
\draw[thick] (-0.1, 1.9) -- (-0.9, 1.1);
\draw[thick] (0.1, 1.9) -- (0.9, 1.1);
\draw[thick] (-1.1, 0.9) -- (-1.9, 0.1);
\draw[thick] (-0.9, 0.9) -- (-0.1, 0.1);
\draw[thick] (0.9, 0.9) -- (0.1, 0.1);
\draw[thick] (1.1, 0.9) -- (1.9, 0.1);
\draw (0,2) circle [radius=0.2];
\draw[fill] (-1,1) circle [radius=0.2];
\draw (1,1) circle [radius=0.2];
\draw (-2,0) circle [radius=0.2];
\draw (0,0) circle [radius=0.2];
\draw (2,0) circle [radius=0.2];
\node[above] at (0,2) {$f$};
\node[left] at (-1,1) {$d$};
\node[right] at (1,1) {$e$};
\node[below] at (-2,0) {$a$};
\node[below] at (0,0) {$b$};
\node[below] at (2,0) {$c$};
\end{scope}
\end{tikzpicture}
\end{center}
\end{example}

\subsection{Graded posets and gyration}\label{subsec:gp}
Thus far, the posets for which rowmotion has been shown to exhibit nice behavior are \textbf{graded}, i.e., posets $P$ with a well defined \textbf{rank function} $\rk: P\ra \zz_{\geq 0}$ satisfying \begin{itemize}
\item $\rk(x)=0$ for any minimal element $x$,
\item $\rk(y)=\rk(x)+1$ if $y\gtrdot x$,
\item every maximal element $x$ has $\rk(x)=r$, where $r$ is called the \textbf{rank} of $P$.
\end{itemize}

For $x\in P$, we call $\rk(x)$ the \textbf{rank} of $x$.  Note that the rank function is uniquely determined y the poset.

In a graded poset $P$, elements of the same rank are pairwise incomparable.  Thus we can define toggling by an entire rank at once (either order ideal or antichain toggling).  This has already been well-studied for order ideal toggles~\cite{strikerwilliams,einpropp}\footnote{Actually, Striker and Williams defined this for a related family of ``rowed-and-columned'' posets~\cite{strikerwilliams}.  Since we can draw the Hasse diagram for a graded poset in a way where each row corresponds to a rank, the name ``rowmotion'' came from the fact that it is toggling by rows for special posets.}.

\begin{definition}
For a graded poset $P$, define $$t_{\rk=i}:=\prod\limits_{\rk(x)=i} t_x,\hspace{0.3 in}\tau_{\rk=i}:=\prod\limits_{\rk(x)=i} \tau_x,\hspace{0.3 in}t_{\rk=i}^*:=\prod\limits_{\rk(x)=i} t_x^*,\hspace{0.3 in}\tau_{\rk=i}^*:=\prod\limits_{\rk(x)=i} \tau_x^*.$$
\end{definition}

All of the rank toggles
$t_{\rk=i},\tau_{\rk=i},t_{\rk=i}^*,\tau_{\rk=i}^*$
defined above are involutions because they are products of commuting involutions.
The following is clear from Propositions~\ref{prop:row-toggles} and~\ref{prop:row-toggles-anti}.  The $\row_\calj$ part is~\cite[Corollary 4.9]{strikerwilliams}, also found in~\cite{einpropp}.

\begin{cor}\label{cor:row-rank}
For a graded poset $P$ of rank $r$, 
$\row_\calj=t_{\rk=0}t_{\rk=1}t_{\rk=2}\cdots t_{\rk=r}$
and
$\row_\cala=\tau_{\rk=r}\cdots\tau_{\rk=2}\tau_{\rk=1}\tau_{\rk=0}$.
\end{cor}

\begin{ex}\label{ex:row-by-rank}
In Figure~\ref{fig:row-by-rank}, we demonstrate both $\row_\calj$ (top) and $\row_\cala$ (bottom) in terms of the rank toggles.  For applying $t_{\rk=i}$, we can insert or remove each element of rank $i$ subject to Proposition~\ref{prop:wild nile ride}.  To apply $\tau_{\rk=i}$, we remove each element of rank $i$ that is in the antichain; otherwise we add the element if and only if it is incomparable with every element in the antichain.  The poset elements toggled in the following step are shown in {\color{red} red}.
\end{ex}

\begin{figure}
\begin{center}
\begin{tikzpicture}[scale=0.486]
\begin{scope}[shift={(0,-4)}]
\draw[thick] (-0.1, 1.9) -- (-0.9, 1.1);
\draw[thick] (0.1, 1.9) -- (0.9, 1.1);
\draw[thick] (-1.1, 0.9) -- (-1.9, 0.1);
\draw[thick] (-0.9, 0.9) -- (-0.1, 0.1);
\draw[thick] (0.9, 0.9) -- (0.1, 0.1);
\draw[thick] (-1.1, -0.9) -- (-1.9, -0.1);
\draw[thick] (-0.9, -0.9) -- (-0.1, -0.1);
\draw[red] (0,2) circle [radius=0.2];
\draw (-1,1) circle [radius=0.2];
\draw[fill] (1,1) circle [radius=0.2];
\draw[fill] (-2,0) circle [radius=0.2];
\draw[fill] (0,0) circle [radius=0.2];
\draw[fill] (-1,-1) circle [radius=0.2];
\node at (-3.75,-1) {\small{rank 0}};
\node at (-3.75,0) {\small{rank 1}};
\node at (-3.75,1) {\small{rank 2}};
\node at (-3.75,2) {\small{rank 3}};
\end{scope}
\node at (3,-3.5) {\Large{$\stackrel{\normalsize{t_{\rk=3}}}{\longmapsto}$}};
\begin{scope}[shift={(7,-4)}]
\draw[thick] (-0.1, 1.9) -- (-0.9, 1.1);
\draw[thick] (0.1, 1.9) -- (0.9, 1.1);
\draw[thick] (-1.1, 0.9) -- (-1.9, 0.1);
\draw[thick] (-0.9, 0.9) -- (-0.1, 0.1);
\draw[thick] (0.9, 0.9) -- (0.1, 0.1);
\draw[thick] (-1.1, -0.9) -- (-1.9, -0.1);
\draw[thick] (-0.9, -0.9) -- (-0.1, -0.1);
\draw (0,2) circle [radius=0.2];
\draw[red] (-1,1) circle [radius=0.2];
\draw[red,fill] (1,1) circle [radius=0.2];
\draw[fill] (-2,0) circle [radius=0.2];
\draw[fill] (0,0) circle [radius=0.2];
\draw[fill] (-1,-1) circle [radius=0.2];
\end{scope}
\node at (10,-3.5) {\Large{$\stackrel{\normalsize{t_{\rk=2}}}{\longmapsto}$}};
\begin{scope}[shift={(14,-4)}]
\draw[thick] (-0.1, 1.9) -- (-0.9, 1.1);
\draw[thick] (0.1, 1.9) -- (0.9, 1.1);
\draw[thick] (-1.1, 0.9) -- (-1.9, 0.1);
\draw[thick] (-0.9, 0.9) -- (-0.1, 0.1);
\draw[thick] (0.9, 0.9) -- (0.1, 0.1);
\draw[thick] (-1.1, -0.9) -- (-1.9, -0.1);
\draw[thick] (-0.9, -0.9) -- (-0.1, -0.1);
\draw (0,2) circle [radius=0.2];
\draw[fill] (-1,1) circle [radius=0.2];
\draw (1,1) circle [radius=0.2];
\draw[red,fill] (-2,0) circle [radius=0.2];
\draw[red,fill] (0,0) circle [radius=0.2];
\draw[fill] (-1,-1) circle [radius=0.2];
\end{scope}
\node at (17,-3.5) {\Large{$\stackrel{\normalsize{t_{\rk=1}}}{\longmapsto}$}};
\begin{scope}[shift={(21,-4)}]
\draw[thick] (-0.1, 1.9) -- (-0.9, 1.1);
\draw[thick] (0.1, 1.9) -- (0.9, 1.1);
\draw[thick] (-1.1, 0.9) -- (-1.9, 0.1);
\draw[thick] (-0.9, 0.9) -- (-0.1, 0.1);
\draw[thick] (0.9, 0.9) -- (0.1, 0.1);
\draw[thick] (-1.1, -0.9) -- (-1.9, -0.1);
\draw[thick] (-0.9, -0.9) -- (-0.1, -0.1);
\draw (0,2) circle [radius=0.2];
\draw[fill] (-1,1) circle [radius=0.2];
\draw (1,1) circle [radius=0.2];
\draw[fill] (-2,0) circle [radius=0.2];
\draw[fill] (0,0) circle [radius=0.2];
\draw[red,fill] (-1,-1) circle [radius=0.2];
\end{scope}
\node at (24,-3.5) {\Large{$\stackrel{\normalsize{t_{\rk=0}}}{\longmapsto}$}};
\begin{scope}[shift={(28,-4)}]
\draw[thick] (-0.1, 1.9) -- (-0.9, 1.1);
\draw[thick] (0.1, 1.9) -- (0.9, 1.1);
\draw[thick] (-1.1, 0.9) -- (-1.9, 0.1);
\draw[thick] (-0.9, 0.9) -- (-0.1, 0.1);
\draw[thick] (0.9, 0.9) -- (0.1, 0.1);
\draw[thick] (-1.1, -0.9) -- (-1.9, -0.1);
\draw[thick] (-0.9, -0.9) -- (-0.1, -0.1);
\draw (0,2) circle [radius=0.2];
\draw[fill] (-1,1) circle [radius=0.2];
\draw (1,1) circle [radius=0.2];
\draw[fill] (-2,0) circle [radius=0.2];
\draw[fill] (0,0) circle [radius=0.2];
\draw[fill] (-1,-1) circle [radius=0.2];
\end{scope}
\end{tikzpicture}
\end{center}

\begin{center}
\begin{tikzpicture}[scale=0.486]
\begin{scope}[shift={(0,-4)}]
\draw[thick] (-0.1, 1.9) -- (-0.9, 1.1);
\draw[thick] (0.1, 1.9) -- (0.9, 1.1);
\draw[thick] (-1.1, 0.9) -- (-1.9, 0.1);
\draw[thick] (-0.9, 0.9) -- (-0.1, 0.1);
\draw[thick] (0.9, 0.9) -- (0.1, 0.1);
\draw[thick] (-1.1, -0.9) -- (-1.9, -0.1);
\draw[thick] (-0.9, -0.9) -- (-0.1, -0.1);
\draw (0,2) circle [radius=0.2];
\draw (-1,1) circle [radius=0.2];
\draw[fill] (1,1) circle [radius=0.2];
\draw[fill] (-2,0) circle [radius=0.2];
\draw (0,0) circle [radius=0.2];
\draw[red] (-1,-1) circle [radius=0.2];
\node at (-3.75,-1) {\small{rank 0}};
\node at (-3.75,0) {\small{rank 1}};
\node at (-3.75,1) {\small{rank 2}};
\node at (-3.75,2) {\small{rank 3}};
\end{scope}
\node at (3,-3.5) {\Large{$\stackrel{\normalsize{\tau_{\rk=0}}}{\longmapsto}$}};
\begin{scope}[shift={(7,-4)}]
\draw[thick] (-0.1, 1.9) -- (-0.9, 1.1);
\draw[thick] (0.1, 1.9) -- (0.9, 1.1);
\draw[thick] (-1.1, 0.9) -- (-1.9, 0.1);
\draw[thick] (-0.9, 0.9) -- (-0.1, 0.1);
\draw[thick] (0.9, 0.9) -- (0.1, 0.1);
\draw[thick] (-1.1, -0.9) -- (-1.9, -0.1);
\draw[thick] (-0.9, -0.9) -- (-0.1, -0.1);
\draw (0,2) circle [radius=0.2];
\draw (-1,1) circle [radius=0.2];
\draw[fill] (1,1) circle [radius=0.2];
\draw[red,fill] (-2,0) circle [radius=0.2];
\draw[red] (0,0) circle [radius=0.2];
\draw (-1,-1) circle [radius=0.2];
\end{scope}
\node at (10,-3.5) {\Large{$\stackrel{\normalsize{\tau_{\rk=1}}}{\longmapsto}$}};
\begin{scope}[shift={(14,-4)}]
\draw[thick] (-0.1, 1.9) -- (-0.9, 1.1);
\draw[thick] (0.1, 1.9) -- (0.9, 1.1);
\draw[thick] (-1.1, 0.9) -- (-1.9, 0.1);
\draw[thick] (-0.9, 0.9) -- (-0.1, 0.1);
\draw[thick] (0.9, 0.9) -- (0.1, 0.1);
\draw[thick] (-1.1, -0.9) -- (-1.9, -0.1);
\draw[thick] (-0.9, -0.9) -- (-0.1, -0.1);
\draw (0,2) circle [radius=0.2];
\draw[red] (-1,1) circle [radius=0.2];
\draw[red,fill] (1,1) circle [radius=0.2];
\draw (-2,0) circle [radius=0.2];
\draw (0,0) circle [radius=0.2];
\draw (-1,-1) circle [radius=0.2];
\end{scope}
\node at (17,-3.5) {\Large{$\stackrel{\normalsize{\tau_{\rk=2}}}{\longmapsto}$}};
\begin{scope}[shift={(21,-4)}]
\draw[thick] (-0.1, 1.9) -- (-0.9, 1.1);
\draw[thick] (0.1, 1.9) -- (0.9, 1.1);
\draw[thick] (-1.1, 0.9) -- (-1.9, 0.1);
\draw[thick] (-0.9, 0.9) -- (-0.1, 0.1);
\draw[thick] (0.9, 0.9) -- (0.1, 0.1);
\draw[thick] (-1.1, -0.9) -- (-1.9, -0.1);
\draw[thick] (-0.9, -0.9) -- (-0.1, -0.1);
\draw[red] (0,2) circle [radius=0.2];
\draw[fill] (-1,1) circle [radius=0.2];
\draw (1,1) circle [radius=0.2];
\draw (-2,0) circle [radius=0.2];
\draw (0,0) circle [radius=0.2];
\draw (-1,-1) circle [radius=0.2];
\end{scope}
\node at (24,-3.5) {\Large{$\stackrel{\normalsize{\tau_{\rk=3}}}{\longmapsto}$}};
\begin{scope}[shift={(28,-4)}]
\draw[thick] (-0.1, 1.9) -- (-0.9, 1.1);
\draw[thick] (0.1, 1.9) -- (0.9, 1.1);
\draw[thick] (-1.1, 0.9) -- (-1.9, 0.1);
\draw[thick] (-0.9, 0.9) -- (-0.1, 0.1);
\draw[thick] (0.9, 0.9) -- (0.1, 0.1);
\draw[thick] (-1.1, -0.9) -- (-1.9, -0.1);
\draw[thick] (-0.9, -0.9) -- (-0.1, -0.1);
\draw (0,2) circle [radius=0.2];
\draw[fill] (-1,1) circle [radius=0.2];
\draw (1,1) circle [radius=0.2];
\draw (-2,0) circle [radius=0.2];
\draw (0,0) circle [radius=0.2];
\draw (-1,-1) circle [radius=0.2];
\end{scope}
\end{tikzpicture}
\end{center}
\caption{In this illustration of Corollary~\ref{cor:row-rank}, we demonstrate $\row_\calj$ on top and $\row_\cala$ on the bottom. See Example~\ref{ex:row-by-rank}.}
\label{fig:row-by-rank}
\end{figure}

The following is a basic corollary to Theorems~\ref{thm:t-star} and~\ref{thm:tau-star}.

\begin{cor}
For a graded poset $P$, the following diagrams commute.

\begin{center}\phantom{1}\hfill
\begin{tikzpicture}
\node at (0,1.8) {$\cala(P)$};
\node at (0,0) {$\calj(P)$};
\node at (3.25,1.8) {$\cala(P)$};
\node at (3.25,0) {$\calj(P)$};
\draw[semithick, ->] (0,1.3) -- (0,0.5);
\node[left] at (0,0.9) {$\bfI$};
\draw[semithick, ->] (0.7,0) -- (2.5,0);
\node[below] at (1.5,0) {$t_{\rk=i}$};
\draw[semithick, ->] (0.7,1.8) -- (2.5,1.8);
\node[above] at (1.5,1.8) {$t_{\rk=i}^*$};
\draw[semithick, ->] (3.25,1.3) -- (3.25,0.5);
\node[right] at (3.25,0.9) {$\bfI$};
\end{tikzpicture}
\hfill
\begin{tikzpicture}
\node at (0,1.8) {$\cala(P)$};
\node at (0,0) {$\calj(P)$};
\node at (3.25,1.8) {$\cala(P)$};
\node at (3.25,0) {$\calj(P)$};
\draw[semithick, ->] (0,1.3) -- (0,0.5);
\node[left] at (0,0.9) {$\bfI$};
\draw[semithick, ->] (0.7,0) -- (2.5,0);
\node[below] at (1.5,0) {$\tau_{\rk=i}^*$};
\draw[semithick, ->] (0.7,1.8) -- (2.5,1.8);
\node[above] at (1.5,1.8) {$\tau_{\rk=i}$};
\draw[semithick, ->] (3.25,1.3) -- (3.25,0.5);
\node[right] at (3.25,0.9) {$\bfI$};
\end{tikzpicture}\hfill\phantom{1}
\end{center}
\end{cor}

In a graded poset, we can state any $t_e^*$ and $\tau_e^*$ in terms of $\tau_e$, $t_e$, and rank toggles.

\begin{prop}
If $\rk(e)=i$, then $t_e^*=\tau_{\rk=i-1}\tau_e\tau_{\rk=i-1}$ and $t_{\rk=i}^*=
\tau_{\rk=i-1}\tau_{\rk=i}\tau_{\rk=i-1}$
(where the empty product $\tau_{\rk=-1}$ is the identity).
\end{prop}

\begin{proof}
Let $e_1,\dots,e_k,x_1,\dots,x_m$ be the elements of rank $i-1$, where $e_1,\dots,e_k$ are covered by $e$ and $x_1,\dots,x_m$ are not.  Then $x_1,\dots,x_m$ are each incomparable with each other, with $e$, and with $e_1,\dots,e_k$.  Thus in the expression $$\tau_{\rk=i-1}\tau_e\tau_{\rk=i-1}=
\tau_{e_1}\cdots \tau_{e_k} \tau_{x_1}\cdots \tau_{x_m} \tau_e \tau_{e_1}\cdots \tau_{e_k} \tau_{x_1}\cdots \tau_{x_m}$$ each $\tau_{x_j}$ can be moved and canceled with the other one.  Therefore,
$$\tau_{\rk=i-1}\tau_e\tau_{\rk=i-1}=
\tau_{e_1}\cdots \tau_{e_k}\tau_e \tau_{e_1}\cdots \tau_{e_k}=t_e^*.$$

Now let $y_1,\dots,y_h$ be the elements of rank $i$.  Then \begin{align*}
t_{\rk=i}^*&= t_{y_1}^* t_{y_2}^* \cdots t_{y_h}^*
\\&=
\tau_{\rk=i-1}\tau_{y_1}\tau_{\rk=i-1}
\tau_{\rk=i-1}\tau_{y_2}\tau_{\rk=i-1}
\cdots
\tau_{\rk=i-1}\tau_{y_h}\tau_{\rk=i-1}
\\&=
\tau_{\rk=i-1}\tau_{y_1}\tau_{y_2}\cdots\tau_{y_h}\tau_{\rk=i-1}\\&=
\tau_{\rk=i-1}\tau_{\rk=i}\tau_{\rk=i-1}.
\end{align*}
\end{proof}

\begin{prop}\label{prop:rank tau star}
If $\rk(e)=i$, then $\tau_e^*=t_{\rk=0}t_{\rk=1}\cdots t_{\rk=i-1} t_e t_{\rk=i-1}\cdots t_{\rk=1}t_{\rk=0}$ and
$\tau_{\rk=i}^*=t_{\rk=0}t_{\rk=1}\cdots t_{\rk=i-1}t_{\rk=i}t_{\rk=i-1}\cdots t_{\rk=1}t_{\rk=0}$.
\end{prop}

\begin{proof}
Let $(x_1,\dots,x_a)$ be a linear extension of
$\{x\in P\;|\; x<e\}$.
If $y\parallel e$, then $y$ is not less than any of $x_1,\dots,x_a$.  Thus, we have a linear extension of the form $(x_1,\dots,x_a,y_1,\dots,y_b)$ for $\{p\in P\;|\; \rk(p)\leq i-1\}$, where $y_1,\dots,y_b$ are all incomparable with $e$.  Since we can rearrange the toggles in $t_{\rk=0}t_{\rk=1}\cdots t_{\rk=i-1}$ according to any linear extension, $$t_{\rk=0}t_{\rk=1}\cdots t_{\rk=i-1}= t_{x_1}t_{x_2}\cdots t_{x_a} t_{y_1}t_{y_2}\cdots t_{y_b}.$$  Therefore,
\begin{align*}
t_{\rk=0}t_{\rk=1}\cdots t_{\rk=i-1}t_e t_{\rk=i-1}\cdots t_{\rk=1}t_{\rk=0}
&=
t_{x_1}t_{x_2}\cdots t_{x_a} t_{y_1}t_{y_2}\cdots t_{y_b} t_e t_{y_b}\cdots t_{y_2}t_{y_1} t_{x_a} \cdots t_{x_2}t_{x_1}
\\&=
t_{x_1}t_{x_2}\cdots t_{x_a} t_e t_{y_1}t_{y_2}\cdots t_{y_b} t_{y_b}\cdots t_{y_2}t_{y_1} t_{x_a} \cdots t_{x_2}t_{x_1}
\\&=
t_{x_1}t_{x_2}\cdots t_{x_a} t_e  t_{x_a} \cdots t_{x_2}t_{x_1}
\\&= \eta_e t_e \eta_e^{-1}
\\&= \tau_e^*.
\end{align*}
Then the $\tau_{\rk=i}^*$ expression follows easily from the above or from Lemma~\ref{lem:tau-star}.
\end{proof}

Given any graded poset $P$, Striker defines in~\cite[\S6]{strikerRS} an element of $\tog_\calj(P)$ called \textit{gyration}, which is conjugate to $\row_\calj$.  The name ``gyration'' is due to its connection with Wieland's map of the same name on alternating sign matrices~\cite{wieland}.

\begin{defn}[\cite{strikerRS}]
Let $P$ be a graded poset.  Then \textbf{order ideal gyration}
$\gyr_\calj:\calj(P) \ra \calj(P)$ is the map that applies the order ideal
toggles for elements in even ranks first, then the odd ranks.
\end{defn}

The order ideal rank toggles $t_{\rk=i}, t_{\rk=j}$ commute
when $i$ and $j$ have the same parity (or more generally when $|i-j|\not=1$).  This is because there are no cover relations between an element of rank $i$ and one of rank $j$ in this
scenario.  Therefore, the definition of $\gyr_\calj$ is well-defined.
It does not matter the order in which elements of even rank are toggled, and similarly for odd rank.

We credit David Einstein and James Propp for the suggestion to define an analogue of gyration with antichain toggles instead, and for great assistance in its definition.
Antichain rank toggles never commute
with each other, so toggling ``the even ranks'' and ``the odd ranks'' are ambiguous unless we define an order for applying the toggles.
We choose the following for the definition of antichain gyration.

\begin{defn}
Let $P$ be a graded poset.  Then \textbf{antichain gyration}
$\gyr_\cala:\cala(P) \ra \cala(P)$ is the map that first applies the antichain toggles for odd ranks starting from the bottom of the poset
up to the top, and then toggles the even ranks from the top of the poset
down to the bottom.
\end{defn}

For example, if $P$ has rank 6, then
$\gyr_\cala=\tau_{\rk=0}\tau_{\rk=2}\tau_{\rk=4}\tau_{\rk=6}\tau_{\rk=5}\tau_{\rk=3}\tau_{\rk=1}$.  We define $\gyr_\cala$ in this way so that
the relation between $\gyr_\calj$ and $\gyr_\cala$ matches that of 
$\row_\calj$ and $\row_\cala$, as in the following theorem.

\begin{thm}\label{thm:antichain gyration}
Let $P$ be a graded poset. The following diagram commutes.

\begin{center}
\begin{tikzpicture}
\node at (0,1.8) {$\cala(P)$};
\node at (0,0) {$\calj(P)$};
\node at (3.25,1.8) {$\cala(P)$};
\node at (3.25,0) {$\calj(P)$};
\draw[semithick, ->] (0,1.3) -- (0,0.5);
\node[left] at (0,0.9) {$\bfI$};
\draw[semithick, ->] (0.7,0) -- (2.5,0);
\node[below] at (1.5,0) {$\gyr_\calj$};
\draw[semithick, ->] (0.7,1.8) -- (2.5,1.8);
\node[above] at (1.5,1.8) {$\gyr_\cala$};
\draw[semithick, ->] (3.25,1.3) -- (3.25,0.5);
\node[right] at (3.25,0.9) {$\bfI$};
\end{tikzpicture}
\end{center}
\end{thm}

See Figure~\ref{fig:gyr-by-rank} for an example illustrating Theorem~\ref{thm:antichain gyration}.  In order to prove the theorem, we begin with a lemma.

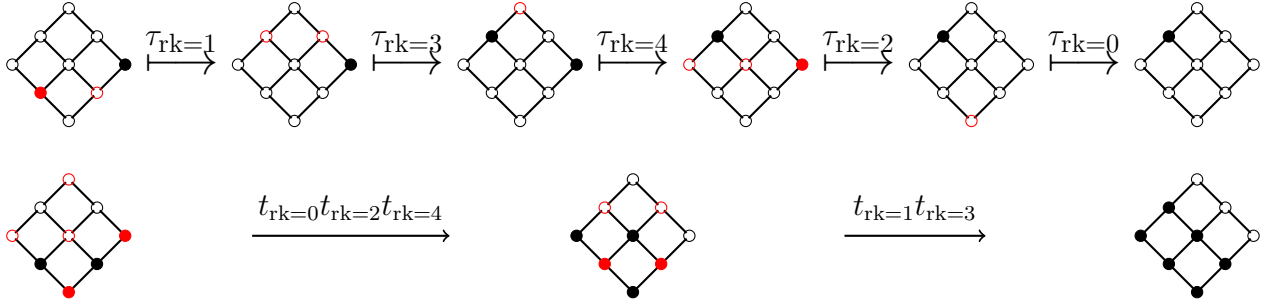
\begin{figure}
\begin{center}
\begin{tikzpicture}[scale=3/8]
\begin{scope}[shift={(0,-4)}]
\draw[thick] (-0.1, 1.9) -- (-0.9, 1.1);
\draw[thick] (0.1, 1.9) -- (0.9, 1.1);
\draw[thick] (-1.1, 0.9) -- (-1.9, 0.1);
\draw[thick] (1.1, 0.9) -- (1.9, 0.1);
\draw[thick] (-0.9, 0.9) -- (-0.1, 0.1);
\draw[thick] (0.9, 0.9) -- (0.1, 0.1);
\draw[thick] (-1.1, -0.9) -- (-1.9, -0.1);
\draw[thick] (1.1, -0.9) -- (1.9, -0.1);
\draw[thick] (-0.9, -0.9) -- (-0.1, -0.1);
\draw[thick] (0.9, -0.9) -- (0.1, -0.1);
\draw[thick] (-0.1, -1.9) -- (-0.9, -1.1);
\draw[thick] (0.1, -1.9) -- (0.9, -1.1);
\draw (0,2) circle [radius=0.2];
\draw (-1,1) circle [radius=0.2];
\draw (1,1) circle [radius=0.2];
\draw (-2,0) circle [radius=0.2];
\draw (0,0) circle [radius=0.2];
\draw[fill] (2,0) circle [radius=0.2];
\draw[red,fill] (-1,-1) circle [radius=0.2];
\draw[red] (1,-1) circle [radius=0.2];
\draw (0,-2) circle [radius=0.2];
\end{scope}
\node at (4,-3.5) {\Large{$\stackrel{\normalsize{\tau_{\rk=1}}}{\longmapsto}$}};
\begin{scope}[shift={(8,-4)}]
\draw[thick] (-0.1, 1.9) -- (-0.9, 1.1);
\draw[thick] (0.1, 1.9) -- (0.9, 1.1);
\draw[thick] (-1.1, 0.9) -- (-1.9, 0.1);
\draw[thick] (1.1, 0.9) -- (1.9, 0.1);
\draw[thick] (-0.9, 0.9) -- (-0.1, 0.1);
\draw[thick] (0.9, 0.9) -- (0.1, 0.1);
\draw[thick] (-1.1, -0.9) -- (-1.9, -0.1);
\draw[thick] (1.1, -0.9) -- (1.9, -0.1);
\draw[thick] (-0.9, -0.9) -- (-0.1, -0.1);
\draw[thick] (0.9, -0.9) -- (0.1, -0.1);
\draw[thick] (-0.1, -1.9) -- (-0.9, -1.1);
\draw[thick] (0.1, -1.9) -- (0.9, -1.1);
\draw (0,2) circle [radius=0.2];
\draw[red] (-1,1) circle [radius=0.2];
\draw[red] (1,1) circle [radius=0.2];
\draw (-2,0) circle [radius=0.2];
\draw (0,0) circle [radius=0.2];
\draw[fill] (2,0) circle [radius=0.2];
\draw (-1,-1) circle [radius=0.2];
\draw (1,-1) circle [radius=0.2];
\draw (0,-2) circle [radius=0.2];
\end{scope}
\node at (12,-3.5) {\Large{$\stackrel{\normalsize{\tau_{\rk=3}}}{\longmapsto}$}};
\begin{scope}[shift={(16,-4)}]
\draw[thick] (-0.1, 1.9) -- (-0.9, 1.1);
\draw[thick] (0.1, 1.9) -- (0.9, 1.1);
\draw[thick] (-1.1, 0.9) -- (-1.9, 0.1);
\draw[thick] (1.1, 0.9) -- (1.9, 0.1);
\draw[thick] (-0.9, 0.9) -- (-0.1, 0.1);
\draw[thick] (0.9, 0.9) -- (0.1, 0.1);
\draw[thick] (-1.1, -0.9) -- (-1.9, -0.1);
\draw[thick] (1.1, -0.9) -- (1.9, -0.1);
\draw[thick] (-0.9, -0.9) -- (-0.1, -0.1);
\draw[thick] (0.9, -0.9) -- (0.1, -0.1);
\draw[thick] (-0.1, -1.9) -- (-0.9, -1.1);
\draw[thick] (0.1, -1.9) -- (0.9, -1.1);
\draw[red] (0,2) circle [radius=0.2];
\draw[fill] (-1,1) circle [radius=0.2];
\draw (1,1) circle [radius=0.2];
\draw (-2,0) circle [radius=0.2];
\draw (0,0) circle [radius=0.2];
\draw[fill] (2,0) circle [radius=0.2];
\draw (-1,-1) circle [radius=0.2];
\draw (1,-1) circle [radius=0.2];
\draw (0,-2) circle [radius=0.2];
\end{scope}
\node at (20,-3.5) {\Large{$\stackrel{\normalsize{\tau_{\rk=4}}}
{\longmapsto}$}};
\begin{scope}[shift={(24,-4)}]
\draw[thick] (-0.1, 1.9) -- (-0.9, 1.1);
\draw[thick] (0.1, 1.9) -- (0.9, 1.1);
\draw[thick] (-1.1, 0.9) -- (-1.9, 0.1);
\draw[thick] (1.1, 0.9) -- (1.9, 0.1);
\draw[thick] (-0.9, 0.9) -- (-0.1, 0.1);
\draw[thick] (0.9, 0.9) -- (0.1, 0.1);
\draw[thick] (-1.1, -0.9) -- (-1.9, -0.1);
\draw[thick] (1.1, -0.9) -- (1.9, -0.1);
\draw[thick] (-0.9, -0.9) -- (-0.1, -0.1);
\draw[thick] (0.9, -0.9) -- (0.1, -0.1);
\draw[thick] (-0.1, -1.9) -- (-0.9, -1.1);
\draw[thick] (0.1, -1.9) -- (0.9, -1.1);
\draw (0,2) circle [radius=0.2];
\draw[fill] (-1,1) circle [radius=0.2];
\draw (1,1) circle [radius=0.2];
\draw[red] (-2,0) circle [radius=0.2];
\draw[red] (0,0) circle [radius=0.2];
\draw[red,fill] (2,0) circle [radius=0.2];
\draw (-1,-1) circle [radius=0.2];
\draw (1,-1) circle [radius=0.2];
\draw (0,-2) circle [radius=0.2];
\end{scope}
\node at (28,-3.5) {\Large{$\stackrel{\normalsize{\tau_{\rk=2}}}
{\longmapsto}$}};
\begin{scope}[shift={(32,-4)}]
\draw[thick] (-0.1, 1.9) -- (-0.9, 1.1);
\draw[thick] (0.1, 1.9) -- (0.9, 1.1);
\draw[thick] (-1.1, 0.9) -- (-1.9, 0.1);
\draw[thick] (1.1, 0.9) -- (1.9, 0.1);
\draw[thick] (-0.9, 0.9) -- (-0.1, 0.1);
\draw[thick] (0.9, 0.9) -- (0.1, 0.1);
\draw[thick] (-1.1, -0.9) -- (-1.9, -0.1);
\draw[thick] (1.1, -0.9) -- (1.9, -0.1);
\draw[thick] (-0.9, -0.9) -- (-0.1, -0.1);
\draw[thick] (0.9, -0.9) -- (0.1, -0.1);
\draw[thick] (-0.1, -1.9) -- (-0.9, -1.1);
\draw[thick] (0.1, -1.9) -- (0.9, -1.1);
\draw (0,2) circle [radius=0.2];
\draw[fill] (-1,1) circle [radius=0.2];
\draw (1,1) circle [radius=0.2];
\draw (-2,0) circle [radius=0.2];
\draw (0,0) circle [radius=0.2];
\draw (2,0) circle [radius=0.2];
\draw (-1,-1) circle [radius=0.2];
\draw (1,-1) circle [radius=0.2];
\draw[red] (0,-2) circle [radius=0.2];
\end{scope}
\node at (36,-3.5) {\Large{$\stackrel{\normalsize{\tau_{\rk=0}}}
{\longmapsto}$}};
\begin{scope}[shift={(40,-4)}]
\draw[thick] (-0.1, 1.9) -- (-0.9, 1.1);
\draw[thick] (0.1, 1.9) -- (0.9, 1.1);
\draw[thick] (-1.1, 0.9) -- (-1.9, 0.1);
\draw[thick] (1.1, 0.9) -- (1.9, 0.1);
\draw[thick] (-0.9, 0.9) -- (-0.1, 0.1);
\draw[thick] (0.9, 0.9) -- (0.1, 0.1);
\draw[thick] (-1.1, -0.9) -- (-1.9, -0.1);
\draw[thick] (1.1, -0.9) -- (1.9, -0.1);
\draw[thick] (-0.9, -0.9) -- (-0.1, -0.1);
\draw[thick] (0.9, -0.9) -- (0.1, -0.1);
\draw[thick] (-0.1, -1.9) -- (-0.9, -1.1);
\draw[thick] (0.1, -1.9) -- (0.9, -1.1);
\draw (0,2) circle [radius=0.2];
\draw[fill] (-1,1) circle [radius=0.2];
\draw (1,1) circle [radius=0.2];
\draw (-2,0) circle [radius=0.2];
\draw (0,0) circle [radius=0.2];
\draw (2,0) circle [radius=0.2];
\draw (-1,-1) circle [radius=0.2];
\draw (1,-1) circle [radius=0.2];
\draw (0,-2) circle [radius=0.2];
\end{scope}
\end{tikzpicture}
\end{center}

\begin{center}
\begin{tikzpicture}[scale=3/8]
\begin{scope}[shift={(0,-4)}]
\draw[thick] (-0.1, 1.9) -- (-0.9, 1.1);
\draw[thick] (0.1, 1.9) -- (0.9, 1.1);
\draw[thick] (-1.1, 0.9) -- (-1.9, 0.1);
\draw[thick] (1.1, 0.9) -- (1.9, 0.1);
\draw[thick] (-0.9, 0.9) -- (-0.1, 0.1);
\draw[thick] (0.9, 0.9) -- (0.1, 0.1);
\draw[thick] (-1.1, -0.9) -- (-1.9, -0.1);
\draw[thick] (1.1, -0.9) -- (1.9, -0.1);
\draw[thick] (-0.9, -0.9) -- (-0.1, -0.1);
\draw[thick] (0.9, -0.9) -- (0.1, -0.1);
\draw[thick] (-0.1, -1.9) -- (-0.9, -1.1);
\draw[thick] (0.1, -1.9) -- (0.9, -1.1);
\draw[red] (0,2) circle [radius=0.2];
\draw (-1,1) circle [radius=0.2];
\draw (1,1) circle [radius=0.2];
\draw[red] (-2,0) circle [radius=0.2];
\draw[red] (0,0) circle [radius=0.2];
\draw[red,fill] (2,0) circle [radius=0.2];
\draw[fill] (-1,-1) circle [radius=0.2];
\draw[fill] (1,-1) circle [radius=0.2];
\draw[red,fill] (0,-2) circle [radius=0.2];
\end{scope}
\node at (10,-3) {$t_{\rk=0}t_{\rk=2}t_{\rk=4}$};
\draw[thick, ->] (6.5,-4) -- (13.5,-4);
\begin{scope}[shift={(20,-4)}]
\draw[thick] (-0.1, 1.9) -- (-0.9, 1.1);
\draw[thick] (0.1, 1.9) -- (0.9, 1.1);
\draw[thick] (-1.1, 0.9) -- (-1.9, 0.1);
\draw[thick] (1.1, 0.9) -- (1.9, 0.1);
\draw[thick] (-0.9, 0.9) -- (-0.1, 0.1);
\draw[thick] (0.9, 0.9) -- (0.1, 0.1);
\draw[thick] (-1.1, -0.9) -- (-1.9, -0.1);
\draw[thick] (1.1, -0.9) -- (1.9, -0.1);
\draw[thick] (-0.9, -0.9) -- (-0.1, -0.1);
\draw[thick] (0.9, -0.9) -- (0.1, -0.1);
\draw[thick] (-0.1, -1.9) -- (-0.9, -1.1);
\draw[thick] (0.1, -1.9) -- (0.9, -1.1);
\draw (0,2) circle [radius=0.2];
\draw[red] (-1,1) circle [radius=0.2];
\draw[red] (1,1) circle [radius=0.2];
\draw[fill] (-2,0) circle [radius=0.2];
\draw[fill] (0,0) circle [radius=0.2];
\draw (2,0) circle [radius=0.2];
\draw[red,fill] (-1,-1) circle [radius=0.2];
\draw[red,fill] (1,-1) circle [radius=0.2];
\draw[fill] (0,-2) circle [radius=0.2];
\end{scope}
\node at (30,-3) {$t_{\rk=1}t_{\rk=3}$};
\draw[thick, ->] (27.5,-4) -- (32.5,-4);
\begin{scope}[shift={(40,-4)}]
\draw[thick] (-0.1, 1.9) -- (-0.9, 1.1);
\draw[thick] (0.1, 1.9) -- (0.9, 1.1);
\draw[thick] (-1.1, 0.9) -- (-1.9, 0.1);
\draw[thick] (1.1, 0.9) -- (1.9, 0.1);
\draw[thick] (-0.9, 0.9) -- (-0.1, 0.1);
\draw[thick] (0.9, 0.9) -- (0.1, 0.1);
\draw[thick] (-1.1, -0.9) -- (-1.9, -0.1);
\draw[thick] (1.1, -0.9) -- (1.9, -0.1);
\draw[thick] (-0.9, -0.9) -- (-0.1, -0.1);
\draw[thick] (0.9, -0.9) -- (0.1, -0.1);
\draw[thick] (-0.1, -1.9) -- (-0.9, -1.1);
\draw[thick] (0.1, -1.9) -- (0.9, -1.1);
\draw (0,2) circle [radius=0.2];
\draw[fill] (-1,1) circle [radius=0.2];
\draw (1,1) circle [radius=0.2];
\draw[fill] (-2,0) circle [radius=0.2];
\draw[fill] (0,0) circle [radius=0.2];
\draw (2,0) circle [radius=0.2];
\draw[fill] (-1,-1) circle [radius=0.2];
\draw[fill] (1,-1) circle [radius=0.2];
\draw[fill] (0,-2) circle [radius=0.2];
\end{scope}
\end{tikzpicture}
\end{center}
\caption{Top: An example of $\gyr_\cala$.  Bottom: An example of $\gyr_\calj$.  Together, they illustrate Theorem~\ref{thm:antichain gyration}.  In each step, we indicate elements whose toggles we apply next in {\color{red}red}.}
\label{fig:gyr-by-rank}
\end{figure}

\begin{lemma}\label{lem:toad village}
Let $a_0,a_1,\dots,a_k$ be elements of a group $G$, such that $a_i^2$ is the identity for every $i\in\{0,1,\dots,k\}$.
For every $j\in\{0,1,\dots,k\}$, set
$$b_j=\underbrace{a_0 a_1\cdots a_{j-1}}_
{\emph{subscripts increase by 1}}
a_j \underbrace{a_{j-1}\cdots a_1 a_0}_
{\emph{subscripts decrease by 1}}.$$
Then for each $i\in\nn$ satisfying $2i\leq k$, we have
$$\underbrace{b_0b_2\cdots b_{2i}}_
{\emph{subscripts increase by 2}}=
\underbrace{a_1a_3\cdots a_{2i-1}}_
{\emph{subscripts increase by 2}}
a_{2i}
\underbrace{a_{2i-1}\cdots a_1a_0}_
{\emph{subscripts decrease by 1}}.
$$
\end{lemma}

\begin{proof}
We proceed inductively.
For the base case $i=0$, $b_0=a_0$.
This is consistent with the lemma
as $a_1 a_3 \cdots a_{2i-1}$ and $a_{2i-1}\cdots a_3 a_1$ are empty products.
The $i=1,2$ cases
$$b_0b_2=\underbrace{a_0
a_0}_{\text{identity}}a_1a_2a_1a_0=
a_1a_2a_1a_0$$
and
\begin{align*}
b_0b_2b_4&=\underbrace{a_0a_0}_{\text{identity}}a_1\underbrace{a_2a_1a_0
a_0a_1a_2}_{\text{identity}}a_3a_4a_3a_2a_1a_0\\&=
a_1a_3a_4a_3a_2a_1a_0
\end{align*}
help illustrate the lemma more clearly.

Now for the induction hypothesis,
we assume the lemma for $i-1$.  That is, we assume
$$\underbrace{b_0b_2\cdots b_{2i-2}}_
{\text{subscripts increase by 2}}=
\underbrace{a_1a_3\cdots a_{2i-3}}_
{\text{subscripts increase by 2}}
a_{2i-2}
\underbrace{a_{2i-3}\cdots a_1a_0}_
{\text{subscripts decrease by 1}}.$$
Now we multiply both sides on the right by $b_{2i}$, which is
$$a_0a_1a_2\cdots a_{2i-1}a_{2i}a_{2i-1}
\cdots a_2a_1a_0.$$
This gives us
\begin{align*}
&\phantom{==}\underbrace{
b_0 b_2 \cdots b_{2i-2}b_{=2i}
}_{\text{subscripts increase by 2}}
\\&=
\underbrace{
a_1 a_3 \cdots a_{2i-3}
}_{\text{subscripts increase by 2}}
\underbrace{
a_{2i-2} a_{2i-3}
\cdots a_2a_1a_0
}_{\text{subscripts decrease by 1}}
\underbrace{
a_0a_1a_2
\cdots a_{2i-3}a_{2i-2}
}_{\text{subscripts increase by 1}}
a_{2i-1}
\underbrace{
a_{2i} a_{2i-1} a_{2i-2}
\cdots a_2a_1a_0
}_{\text{subscripts decrease by 1}}\\
&=
\underbrace{
a_1 a_3 \cdots a_{2i-3}
}_{\text{subscripts increase by 2}}
a_{2i-1}
\underbrace{
a_{2i} a_{2i-1} a_{2i-2}
\cdots a_2a_1a_0
}_{\text{subscripts decrease by 1}}\\
&=
\underbrace{
a_1a_3 \cdots a_{2i-3}a_{2i-1}
}_{\text{subscripts increase by 2}}
\underbrace{
a_{2i} a_{2i-1} a_{2i-2}
\cdots a_2a_1a_0
}_{\text{subscripts decrease by 1}}
\end{align*}
which proves the lemma.
\end{proof}

\begin{proof}[Proof of Theorem~\ref{thm:antichain gyration}]
If $P$ has rank $2k$, then
$$\gyr_\calj = \underbrace{
t_{\rk=1} t_{\rk=3} \cdots t_{\rk=2k-1}
}_{\text{ranks increase by 2}}
\underbrace{
t_{\rk=2k} t_{\rk=2k-2} \cdots t_{\rk=2} t_{\rk=0}
}_{\text{ranks decrease by 2}}$$
and
$$\gyr_\cala = \underbrace{
\tau_{\rk=0} \tau_{\rk=2} \cdots \tau_{\rk=2k-2}\tau_{\rk=2k}
}_{\text{ranks increase by 2}}
\underbrace{
\tau_{\rk=2k-1} \cdots \tau_{\rk=3} \tau_{\rk=1}
}_{\text{ranks decrease by 2}}.
$$
So for posets of even rank $2k$, it suffices to prove that
\begin{equation}\label{eq:goal 2k}
\underbrace{
\tau^*_{\rk=0} \tau^*_{\rk=2} \cdots \tau^*_{\rk=2k}
}_{\text{ranks increase by 2}}
\underbrace{
\tau^*_{\rk=2k-1} \cdots \tau^*_{\rk=3} \tau^*_{\rk=1}
}_{\text{ranks decrease by 2}}%\nonumber
=
\underbrace{
t_{\rk=1} t_{\rk=3} \cdots t_{\rk=2k-1}
}_{\text{ranks increase by 2}}
\underbrace{
t_{\rk=2k} \cdots t_{\rk=2} t_{\rk=0}
}_{\text{ranks decrease by 2}}.
\end{equation}
On the other hand, if $P$ has rank $2k+1$, then
$$\gyr_\calj = \underbrace{
t_{\rk=1} t_{\rk=3} \cdots t_{\rk=2k-1}t_{\rk=2k+1}
}_{\text{ranks increase by 2}}
\underbrace{
t_{\rk=2k} t_{\rk=2k-2} \cdots t_{\rk=2} t_{\rk=0}
}_{\text{ranks decrease by 2}}$$
and
$$\gyr_\cala = \underbrace{
\tau_{\rk=0} \tau_{\rk=2} \cdots \tau_{\rk=2k-2}\tau_{\rk=2k}
}_{\text{ranks increase by 2}}
\underbrace{
\tau_{\rk=2k+1}\tau_{\rk=2k-1} \cdots \tau_{\rk=3} \tau_{\rk=1}
}_{\text{ranks decrease by 2}}.
$$
Thus, for posets of odd rank $2k+1$, it suffices to prove that
\begin{equation}\label{eq:goal 2k+1}
\underbrace{
\tau^*_{\rk=0} \tau^*_{\rk=2} \cdots \tau^*_{\rk=2k}
}_{\text{ranks increase by 2}}
\underbrace{
\tau^*_{\rk=2k+1} \cdots \tau^*_{\rk=3} \tau^*_{\rk=1}
}_{\text{ranks decrease by 2}}%\nonumber
=
\underbrace{
t_{\rk=1} t_{\rk=3} \cdots t_{\rk=2k+1}
}_{\text{ranks increase by 2}}
\underbrace{
t_{\rk=2k} \cdots t_{\rk=2} t_{\rk=0}
}_{\text{ranks decrease by 2}}.
\end{equation}

To prove Eq.~(\ref{eq:goal 2k}) and~(\ref{eq:goal 2k+1}),
we list a few equations.
By setting $a_j=t_{\rk=j}$ and $i=k$ in Lemma~\ref{lem:toad village} (so $b_j=\tau^*_{\rk=j}$), we obtain
\begin{equation}\label{eq:toad village}
\underbrace{
\tau^*_{\rk=0} \tau^*_{\rk=2} \cdots \tau^*_{\rk=2k}
}_{\text{ranks increase by 2}}
=
\underbrace{
t_{\rk=1} t_{\rk=3} \cdots t_{\rk=2k-1}
}_{\text{ranks increase by 2}}
\underbrace{
t_{\rk=2k} t_{\rk=2k-1} t_{\rk=2k-2}
\cdots t_{\rk=2}t_{\rk=1}t_{\rk=0}
}_{\text{ranks decrease by 1}}.
\end{equation}

We can prove the following by setting $a_j=t_{\rk=i+1}$ and $i=k-1$ in Lemma~\ref{lem:toad village} (so $b_j=t_{\rk=0}\tau^*_{\rk=j+1}t_{\rk=0})$), then conjugating both sides by $t_{\rk=0}$ and inverting both sides.
\begin{equation}\label{eq:gee wiz}
\underbrace{
\tau^*_{\rk=2k-1} \tau^*_{\rk=2k-3} \cdots \tau^*_{\rk=3}\tau^*_{\rk=1}
}_{\text{ranks decrease by 2}}
=
\underbrace{
t_{\rk=0} t_{\rk=1} t_{\rk=2} \cdots t_{\rk=2k-1}
}_{\text{ranks increase by 1}}
\underbrace{
t_{\rk=2k-2} t_{\rk=2k-4}
\cdots t_{\rk=2}t_{\rk=0}
}_{\text{ranks decrease by 2}}.
\end{equation}
By replacing $k$ with $k+1$ in Eq.~(\ref{eq:gee wiz}), we obtain
\begin{equation}\label{eq:double header}
\underbrace{
\tau^*_{\rk=2k+1} \tau^*_{\rk=2k-1} \cdots \tau^*_{\rk=3}\tau^*_{\rk=1}
}_{\text{ranks decrease by 2}}
=
\underbrace{
t_{\rk=0} t_{\rk=1} t_{\rk=2} \cdots t_{\rk=2k+1}
}_{\text{ranks increase by 1}}
\underbrace{
t_{\rk=2k} t_{\rk=2k-2}
\cdots t_{\rk=2}t_{\rk=0}
}_{\text{ranks decrease by 2}}.
\end{equation}

To prove Eq.~(\ref{eq:goal 2k}), we multiply the left and right sides of Eq.~(\ref{eq:toad village}) by those of Eq.~(\ref{eq:gee wiz})
to obtain
\begin{align*}
&\phantom{==}\underbrace{
\tau^*_{\rk=0} \tau^*_{\rk=2} \cdots \tau^*_{\rk=2k}
}_{\text{ranks increase by 2}}
\underbrace{
\tau^*_{\rk=2k-1} \cdots \tau^*_{\rk=3}\tau^*_{\rk=1}
}_{\text{ranks decrease by 2}}\\
&=
\underbrace{
t_{\rk=1} t_{\rk=3} \cdots t_{\rk=2k-1}
}_{\text{ranks increase by 2}}
t_{\rk=2k} \underbrace{
t_{\rk=2k-1} t_{\rk=2k-2}
\cdots t_{\rk=2}t_{\rk=1}t_{\rk=0}
}_{\text{ranks decrease by 1}}\\
&\phantom{==}
\underbrace{
t_{\rk=0} t_{\rk=1} t_{\rk=2} \cdots t_{\rk=2k-1}
}_{\text{ranks increase by 1}}
\underbrace{
t_{\rk=2k-2} 
\cdots t_{\rk=2}t_{\rk=0}
}_{\text{ranks decrease by 2}}\\
&= \underbrace{
t_{\rk=1} t_{\rk=3} \cdots t_{\rk=2k-1}
}_{\text{ranks increase by 2}}
t_{\rk=2k} \underbrace{
t_{\rk=2k-2} 
\cdots t_{\rk=2}t_{\rk=0}
}_{\text{ranks decrease by 2}}\\
&=
\underbrace{
t_{\rk=1} t_{\rk=3} \cdots t_{\rk=2k-1}
}_{\text{ranks increase by 2}}
\underbrace{
t_{\rk=2k} t_{\rk=2k-2} 
\cdots t_{\rk=2}t_{\rk=0}
}_{\text{ranks decrease by 2}}.
\end{align*}

Similarly, for the proof of
Eq.~(\ref{eq:goal 2k+1}), we multiply the left and right sides of Eq.~(\ref{eq:toad village}) by those of Eq.~(\ref{eq:double header}).
This gives us
\begin{align*}
&\phantom{==}\underbrace{
\tau^*_{\rk=0} \tau^*_{\rk=2} \cdots \tau^*_{\rk=2k}
}_{\text{ranks increase by 2}}
\underbrace{
\tau^*_{\rk=2k+1} \cdots \tau^*_{\rk=3}\tau^*_{\rk=1}
}_{\text{ranks decrease by 2}}\\
&=
\underbrace{
t_{\rk=1} t_{\rk=3} \cdots t_{\rk=2k-1}
}_{\text{ranks increase by 2}}
\underbrace{t_{\rk=2k} 
t_{\rk=2k-1} t_{\rk=2k-2}
\cdots t_{\rk=2}t_{\rk=1}t_{\rk=0}
}_{\text{ranks decrease by 1}}\\
&\phantom{==}
\underbrace{
t_{\rk=0} t_{\rk=1} t_{\rk=2} \cdots t_{\rk=2k-1}t_{\rk=2k} 
}_{\text{ranks increase by 1}}
t_{\rk=2k+1}
\underbrace{
t_{\rk=2k} 
\cdots t_{\rk=2}t_{\rk=0}
}_{\text{ranks decrease by 2}}\\
&= \underbrace{
t_{\rk=1} t_{\rk=3} \cdots t_{\rk=2k-1}
}_{\text{ranks increase by 2}}
t_{\rk=2k+1} \underbrace{
t_{\rk=2k} 
\cdots t_{\rk=2}t_{\rk=0}
}_{\text{ranks decrease by 2}}\\
&=
\underbrace{
t_{\rk=1} t_{\rk=3} \cdots t_{\rk=2k-1}t_{\rk=2k+1}
}_{\text{ranks increase by 2}}
\underbrace{
t_{\rk=2k}
\cdots t_{\rk=2}t_{\rk=0}
}_{\text{ranks decrease by 2}}
\end{align*}
concluding the proof of Eq.~(\ref{eq:goal 2k+1}) and the proof
of the theorem.
\end{proof}

\section{Piecewise-linear generalization}\label{sec:cpl}

We call the toggles and rowmotion maps on $\calj(P)$ and $\cala(P)$ \textbf{combinatorial} toggling and rowmotion as they are acting on combinatorial sets.
Einstein and Propp~\cite{einpropp}
have generalized these maps on $\calj(P)$ to \textbf{piecewise-linear} toggling and rowmotion, by constructing continuous maps that act on Stanley's ``order polytope,'' an extension of $\calj(P)$ and $\calf(P)$~\cite{Sta86}.
In this section, we expand on this work and generalize the toggles $\tau_e$ on antichains to another polytope of Stanley, called the ``chain polytope'' which extends antichains.
For certain posets $P$ in which cardinality is a homomesic statistic under $\row_\cala$,  this appears to extend to the piecewise-linear setting.

Many of the algebraic properties that hold in the combinatorial setting have also been proven for the piecewise-linear setting, and furthermore generalized to the birational setting~\cite{einpropp,grinberg-roby,robydac}.  We will show that almost all that we proved for the relationship between toggles in $\tog_\cala(P)$ and $\tog_\calj(P)$ also extends to the piecewise-linear setting.
We will not discuss birational toggling here except in the final two paragraphs
of Section~\ref{sec:future tense}, where we mention it as a possible direction for future research.

\subsection{Poset polytopes}\label{subsec:cpl-intro}
\begin{notn}
For a set $X$ and finite poset $P$, let $X^P$ denote the set of $X$-\textbf{labelings} of $P$, i.e., the set of functions $f:P\ra X$.
Given $f\in X^P$ and $e\in P$, we call $f(e)$ the \textbf{label}
of $e$.
\end{notn}

A subset $S\subseteq P$ corresponds naturally to a $\{0,1\}$-labeling $f$ of $P$ by letting $f(x)=1$ if $x\in S$ and $f(x)=0$ if $x\not\in S$, as in the example below.

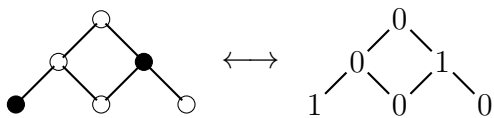
\begin{center}
\begin{tikzpicture}[scale=0.567]
\begin{scope}
\draw[thick] (-0.1, 1.9) -- (-0.9, 1.1);
\draw[thick] (0.1, 1.9) -- (0.9, 1.1);
\draw[thick] (-1.1, 0.9) -- (-1.9, 0.1);
\draw[thick] (-0.9, 0.9) -- (-0.1, 0.1);
\draw[thick] (0.9, 0.9) -- (0.1, 0.1);
\draw[thick] (1.1, 0.9) -- (1.9, 0.1);
\draw (0,2) circle [radius=0.2];
\draw (-1,1) circle [radius=0.2];
\draw[fill] (1,1) circle [radius=0.2];
\draw[fill] (-2,0) circle [radius=0.2];
\draw (0,0) circle [radius=0.2];
\draw (2,0) circle [radius=0.2];
\end{scope}
\node at (3.5,1) {$\longleftrightarrow$};
\begin{scope}[shift={(7,0)}]
\draw[thick] (-0.25, 1.75) -- (-0.75, 1.25);
\draw[thick] (0.25, 1.75) -- (0.75, 1.25);
\draw[thick] (-1.25, 0.75) -- (-1.75, 0.25);
\draw[thick] (-0.75, 0.75) -- (-0.25, 0.25);
\draw[thick] (0.75, 0.75) -- (0.25, 0.25);
\draw[thick] (1.25, 0.75) -- (1.75, 0.25);
\node at (0,2) {0};
\node at (-1,1) {0};
\node at (1,1) {1};
\node at (-2,0) {1};
\node at (0,0) {0};
\node at (2,0) {0};
\end{scope}
\end{tikzpicture}
\end{center}

This labeling is called the \textbf{indicator function} of the subset.  We consider a subset and its indicator function as two separate ways of writing the same object, so we will not distinguish the two.

\begin{prop}\label{prop:as binary labelings}\hspace{-8 in}\begin{tiny}.\end{tiny}
\begin{enumerate}
\item Antichains of $P$ are precisely the $\{0,1\}$-labelings $f$ of $P$ such that for every chain $x_1<x_2<\cdots< x_n$ in $P$, we have $\sum\limits_{i=1}^n f(x_i)\leq 1$.
\item Order ideals of $P$ are precisely the $\{0,1\}$-labelings $f$ of $P$ that are \textbf{order-reversing}, meaning that $f(x)\geq f(y)$ whenever $x\leq y$.
\item Order filters of $P$ are precisely the $\{0,1\}$-labelings $f$ of $P$ that are \textbf{order-preserving}, meaning that $f(x)\leq f(y)$ whenever $x\leq y$.
\end{enumerate}
\end{prop}

\begin{proof}\hspace{-8 in}\begin{tiny}.\end{tiny}

\begin{enumerate}
\item A subset $S\subseteq P$ is an antichain if and only if $S$ contains at most one element in any chain $x_1<x_2<\cdots<x_n$; for binary functions this is exactly the same condition as $\sum\limits_{i=1}^n f(x_i)\leq 1$.
\item The condition that makes $I\subseteq P$ an order ideal is that if $x<y$ and $y\in I$, then $x\in I$.  Consider a pair $x,y\in P$ satisfying $x\leq y$.  If $f(y)=0$, then automatically $f(x)\geq f(y)$.  If $f(y)=1$, then $f(x)\geq f(y)$ if and only if $f(x)=1$, which is exactly the requirement to be an order ideal.
\item Analogous to (2).
\end{enumerate}
\end{proof}

We now generalize these from labelings in $\{0,1\}^P$ to $[0,1]^P$.
%\footnote{In some combinatorics literature, the notation $[m,n]$ refers to the set of \textit{integers} from $m$ to $n$, but that is not what we mean.  In here $[0,1]$ refers to the set of real numbers from 0 to 1.}
In~\cite{Sta86}, Stanley introduced two polytopes associated with a poset: the chain polytope and the order polytope.  Stanley's ``order polytope'' is what we call the ``order-preserving polytope.''

\begin{defn}\hspace{-1 in}\begin{tiny}.\end{tiny}
\begin{itemize}
\item The \textbf{chain polytope} of $P$, denoted $\calc(P)$, is the set of all labelings $f\in [0,1]^P$ such that $\sum\limits_{i=1}^n f(x_i)\leq 1$ for any chain $x_1<x_2<\cdots<x_n$.
\item The \textbf{order-reversing polytope} of $P$, denoted $OR(P)$, is the set of all order-reversing labelings $f\in [0,1]^P$.
\item The \textbf{order-preserving polytope} of $P$, denoted $OP(P)$, is the set of all order-preserving labelings $f\in [0,1]^P$.
\end{itemize}
\end{defn}

By Proposition~\ref{prop:as binary labelings}, $\cala(P)=\calc(P)\cap \{0,1\}^P$, $\calj(P)=OR(P)\cap \{0,1\}^P$, and $\calf(P)=OP(P)\cap \{0,1\}^P$ (the vertices of the respective polytopes~\cite{Sta86}).  Thus, anything we prove to be true on these polytopes is also true for the combinatorial sets $\cala(P)$, $\calj(P)$, and $\calf(P)$.  What is more surprising, however, is that almost all of what we
proved in Section~\ref{sec:combinatorial} when working over $\cala(P)$, $\calj(P)$, and $\calf(P)$ can be extended to $\calc(P)$, $OR(P)$, and $OP(P)$ in a natural way.

As we will not use polytope theory in this paper, knowledge of polytopes is not necessary to understand the rest of this paper.  The reader may choose to think of $\calc(P)$, $OR(P)$, and $OP(P)$ simply as subsets of $[0,1]^P$.

\subsection{The poset $\hat{P}$}\label{subsec:P-hat}

In order to work with $OR(P)$ and $OP(P)$, we create a new poset $\hat{P}=P\cup\left\{
\hat{m}, \hat{M} \right\}$ from any given poset $P$ by adjoining a minimal element $\hat{m}$ and maximal element $\hat{M}$.
For any $x,y\in P$, $x\leq y$ in $\hat{P}$ if and only if $x\leq y$ in $P$.  For any $x\in\hat{P}$, we let $\hat{m}\leq x \leq \hat{M}$.  When we make statements like ``$x\leq y$'' or ``$x\gtrdot y$'' or ``$x\parallel y$,'' we need not clarify if we mean in $P$ or $\hat{P}$, since there is no ambiguity:  If at least one of $x$ and $y$ is $\hat{m}$ or $\hat{M}$, then we must mean $\hat{P}$.  On the other hand,
if both $x,y\in P$, then those types of statements hold in $P$ if and only if they hold in $\hat{P}$.  Note that a maximal or minimal element of $P$ does not remain as such in $\hat{P}$.

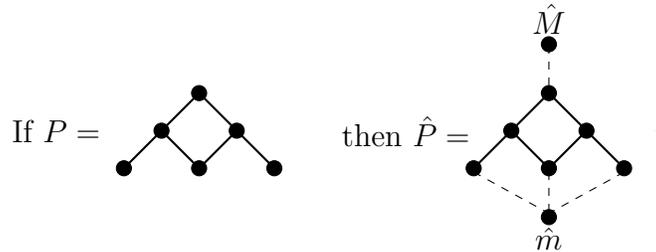
\begin{ex}\label{ex:P-hat}\hspace{-1 in}\begin{tiny}.\end{tiny}

\begin{center}
\begin{tikzpicture}[scale=0.5]
\draw[thick] (-0.1, 1.9) -- (-0.9, 1.1);
\draw[thick] (0.1, 1.9) -- (0.9, 1.1);
\draw[thick] (-1.1, 0.9) -- (-1.9, 0.1);
\draw[thick] (-0.9, 0.9) -- (-0.1, 0.1);
\draw[thick] (0.9, 0.9) -- (0.1, 0.1);
\draw[thick] (1.1, 0.9) -- (1.9, 0.1);
\draw[fill] (0,2) circle [radius=0.2];
\draw[fill] (-1,1) circle [radius=0.2];
\draw[fill] (1,1) circle [radius=0.2];
\draw[fill] (-2,0) circle [radius=0.2];
\draw[fill] (0,0) circle [radius=0.2];
\draw[fill] (2,0) circle [radius=0.2];
\node at (-3.8,1) {If $P=$};
\begin{scope}[shift={(9.3,0)}];
\draw[thick] (-0.1, 1.9) -- (-0.9, 1.1);
\draw[thick] (0.1, 1.9) -- (0.9, 1.1);
\draw[thick] (-1.1, 0.9) -- (-1.9, 0.1);
\draw[thick] (-0.9, 0.9) -- (-0.1, 0.1);
\draw[thick] (0.9, 0.9) -- (0.1, 0.1);
\draw[thick] (1.1, 0.9) -- (1.9, 0.1);
\draw[dashed] (0, 3.26) -- (0,2.1);
\draw[dashed] (0, -0.1) -- (0,-1.2);
\draw[dashed] (-2, -0.1) -- (0,-1.2);
\draw[dashed] (2, -0.1) -- (0,-1.2);
\draw[fill] (0,2) circle [radius=0.2];
\draw[fill] (-1,1) circle [radius=0.2];
\draw[fill] (1,1) circle [radius=0.2];
\draw[fill] (-2,0) circle [radius=0.2];
\draw[fill] (0,0) circle [radius=0.2];
\draw[fill] (2,0) circle [radius=0.2];
\draw[fill] (0,3.3) circle [radius=0.2];
\draw[fill] (0,-1.3) circle [radius=0.2];
\node at (-3.8,1) {then $\hat{P}=$};
\node at (2.9,1) {.};
\node[above] at (0,3.3) {$\hat{M}$};
\node[below] at (0,-1.3) {$\hat{m}$};
\end{scope}
\end{tikzpicture}
\end{center}
We will use dashed lines throughout the paper to denote the edges going to $\hat{m}$ and $\hat{M}$, so that it will be clear if we are drawing $P$ or $\hat{P}$. 
\end{ex}

We extend every $f\in OR(P)$ to a labeling of $\hat{P}$ by setting $f\left(\hat{m}\right) =1$ and $f\left(\hat{M}\right) =0$.\footnote{Elsewhere in the literature, $\hat{m}$ and $\hat{M}$ are denoted $\hat{0}$ and $\hat{1}$ respectively.  With this norm, order-reversing maps would have $f\left(\hat{0}\right)=1$ and $f\left(\hat{1}\right)=0$.  This is potentially confusing so we deviate from this norm.}
We likewise extend every $f\in OP(P)$ to a labeling of $\hat{P}$ by $f\left(\hat{m}\right) =0$ and $f\left(\hat{M}\right) =1$.
Even though a constant labeling is both order-reversing and order-preserving, we only consider it to be in one of $OR(P)$ and $OP(P)$ at any time, and assign the appropriate labels to $\hat{m}$ and $\hat{M}$ accordingly.  We do not extend elements of $\calc(P)$ to $\hat{P}$.

Working over $\hat{P}$ will allow us to state definitions and theorems without splitting them into several cases.  For example, $x\gtrdot \hat{m}$ (resp.\ $x\lessdot \hat{M}$) means that $x$ is a minimal (resp.\ maximal) element of $P$. Also for $x\in P$, the sets $\left\{\left.y\in \hat{P}\;\right|\; y\gtrdot x\right\}$ and $\left\{\left.y\in \hat{P}\;\right|\; y\lessdot x\right\}$ are always nonempty so a labeling $f$ achieves maximum and minimum values on these sets.

\subsection{Rowmotion on poset polytopes}\label{subsec:row-polytope}

In this subsection, we define rowmotion on $\calc(P)$, $OR(P)$, and $OP(P)$ as the composition of three maps in a way analogous to the rowmotion definitions in Section~\ref{sec:combinatorial}.

\begin{defn}
The \textbf{complement} of a labeling is given by $\comp:[0,1]^P \ra [0,1]^P$ where $(\comp(f))(x)=1-f(x)$ for all $x\in P$.
\end{defn}

Note that $\comp$ is an involution that takes elements in $OR(P)$ to ones in $OP(P)$ and vice versa.  When restricted to $\{0,1\}^P$ (which again we think of as subsets of $P$), $\comp$ corresponds to the usual complementation operation, hence the name.

\begin{prop}\label{prop:transfer-map}
There is a bijection $\orb: \calc(P) \ra OR(P)$ given by
$$(\orb(g))(x) = \max\left\{g(y_1)+g(y_2)+\cdots+ g(y_k)\;\left|\; x=y_1 \lessdot y_2 \lessdot \cdots \lessdot y_k \lessdot \hat{M} \right.\right\}$$ with inverse given by
$$\left(\orb^{-1}(f)\right)(x) = \min\left\{f(x)-f(y)\;\left|\; y\in \hat{P}, y \gtrdot x\right.\right\}
=f(x) - \max\left\{f(y)\;\left|\; y\in \hat{P}, y\gtrdot x\right.\right\}.$$  Also there is a
bijection $\opb: \calc(P) \ra OP(P)$ given by
$$(\opb(g))(x) = \max\left\{g(y_1)+g(y_2)+\cdots+ g(y_k)\;|\; \hat{m}
\lessdot y_1 \lessdot y_2 \lessdot \cdots \lessdot y_k =x \right\}$$ with inverse given by
$$\left(\opb^{-1}(f)\right)(x) = \min\left\{f(x)-f(y)\;\left|\; y\in \hat{P}, y \lessdot x\right.\right\}
=f(x) - \max\left\{f(y)\;\left|\; y\in \hat{P}, y \lessdot x\right.\right\}.$$
\end{prop}

We omit the proof as it is straightforward to show that $\orb$ (resp.\ $\opb$) sends elements of $\calc(P)$ to elements of $OR(P)$ (resp.\ $OP(P)$), that $\orb^{-1}$ (resp.\ $\opb^{-1}$) sends elements of $OR(P)$ (resp.\ $OP(P)$) to elements of $\calc(P)$, and that $\orb^{-1}$ and $\opb^{-1}$ are inverses of $\orb$ and $\opb$.
The map $\opb^{-1}$ is what Stanley calls the ``transfer map'' because it can be used to transfer properties from one of $OP(P)$ or $\calc(P)$ to the other~\cite[\S3]{Sta86}.  Also $\orb$ is just $\opb$ but applied to the dual poset that reverses the `$\geq$' and `$\leq$' relations.  Clearly if $Q$ is the dual poset of $P$, then they have the same chains and antichains, so $\calc(P) = \calc(Q)$.

We can replace $y\gtrdot x$ with $y> x$ in the definition of $\orb^{-1}$, since it would produce the same result by the order-reversing property.  Similarly, we can replace $y\lessdot x$ with $y< x$ in the definition of $\opb^{-1}$.
Also any $g\in\calc(P)$ has only nonnegative outputs.  So in the $\orb$ and $\opb$ definitions, we can replace ``$x=y_1 \lessdot y_2 \lessdot \cdots \lessdot y_k \lessdot \hat{M}$'' and ``$\hat{m} \lessdot y_1 \lessdot y_2 \lessdot \cdots \lessdot y_k=x$'' with ``$x=y_1<y_2 < \cdots < y_k$''
and ``$y_1<y_2 < \cdots < y_k=x$'' respectively since the maximum sum must occur on a chain that cannot be extended.

It is easy to see that $\orb$ and $\opb$ can be described recursively as well.

\begin{equation}\label{eq:recur-orb}
(\orb(g))(x) =
\left\{\begin{array}{ll}
0 &\text{if }x=\hat{M}\\
g(x)+\max\limits_{y\gtrdot x}(\orb(g))(y)&\text{if }x\in P\\
1
&\text{if }x=\hat{m}\end{array}\right.
\end{equation}

\begin{equation}\label{eq:recur-opb}
(\opb(g))(x) =
\left\{\begin{array}{ll}
0 &\text{if }x=\hat{m}\\
g(x)+\max\limits_{y\lessdot x}(\opb(g))(y)&\text{if }x\in P\\
1
&\text{if }x=\hat{M}\end{array}\right.
\end{equation}

We call $\orb(g)$ and $\opb(g)$ the order-reversing and order-preserving labelings generated by the chain polytope element $g$.

\begin{prop}\label{prop:same-on-anti}
If $g\in \cala(P)$, then $\orb(g)=\bfI(g)$ and $\opb(g)=\bfF(g)$.
\end{prop}

\begin{proof}
Let $g\in\cala(P)$ and $f=\bfI(g)$.  Then for $x\in P$, $f(x)=1$ if and only if there exists $y\geq x$ such that $g(y)=1$.  Otherwise $f(x)=0$.  Since $g$ is an antichain, it is a $\{0,1\}$-labeling.  Therefore, any chain $x=y_1\lessdot y_2\lessdot \cdots \lessdot y_k \lessdot \hat{M}$ satisfies
$$\sum\limits_{i=1}^k g(y_i)=1$$ precisely when some $g(y_i)=1$; otherwise the sum is 0.  Such a chain exists precisely when some $y\geq x$ satisfies $g(y)=1$.  Thus, $f=\orb(g)$.

Proving that $\opb(g)=\bfF(g)$ is analogous.
\end{proof}

Since $\orb$ and $\opb$ are extensions of $\bfI$ and $\bfF$ to $\calc(P)$, $OR(P)$, and $OP(P)$, we can extend the definition of rowmotion to these polytopes by composing these similarly to the definitions of $\row_\cala$, $\row_\calj$, and $\row_\calf$.  In fact, $\row_\cala$, $\row_\calj$, and $\row_\calf$ are the restrictions of the following maps to $\cala(P)$, $\calj(P)$, and $\calf(P)$, respectively.

\begin{defn}
Let $\row_\calc$, $\row_{OR}$, $\row_{OP}$ be defined by composing maps as follows.
\begin{center}
\begin{tabular}{ccccccccc}
$\row_\calc$ &:& $\calc(P)$ & $\stackrel{\orb}{\longrightarrow}$ & $OR(P)$ & $\stackrel{\comp}  {\longrightarrow}$ & $OP(P)$ & $\stackrel{\opb^{-1}}{\longrightarrow}$ & $\calc(P)$\\
$\row_{OR}$ &:& $OR(P)$ & $\stackrel{\comp}{\longrightarrow}$ & $OP(P)$ & $\stackrel{\opb^{-1}}{\longrightarrow}$ & $\calc(P)$ & $\stackrel{\orb}{\longrightarrow}$ & $OR(P)$\\
$\row_{OP}$ &:& $OP(P)$ & $\stackrel{\comp}{\longrightarrow}$ & $OR(P)$ & $\stackrel{\orb^{-1}}{\longrightarrow}$ & $\calc(P)$ & $\stackrel{\opb}{\longrightarrow}$ & $OP(P)$
\end{tabular}
\end{center}
\end{defn}

\begin{ex}\label{ex:a3-row-cpl}
We demonstrate $\row_\calc$ and $\row_{OR}$.

\begin{center}
\begin{tikzpicture}[scale=0.567]
\node at (-3.5,1) {$\row_\calc:$};
\node at (-3.5,-3) {$\row_{OR}:$};
\begin{scope}
\draw[thick] (-0.3, 1.7) -- (-0.7, 1.3);
\draw[thick] (0.3, 1.7) -- (0.7, 1.3);
\draw[thick] (-1.3, 0.7) -- (-1.7, 0.3);
\draw[thick] (-0.7, 0.7) -- (-0.3, 0.3);
\draw[thick] (0.7, 0.7) -- (0.3, 0.3);
\draw[thick] (1.3, 0.7) -- (1.7, 0.3);
\node at (0,2) {0.2};
\node at (-1,1) {0.7};
\node at (1,1) {0};
\node at (-2,0) {0.1};
\node at (0,0) {0};
\node at (2,0) {0.3};
\end{scope}
\node at (3.5,1) {$\stackrel{\orb}{\longmapsto}$};
\begin{scope}[shift={(7,0)}]
\draw[thick] (-0.3, 1.7) -- (-0.7, 1.3);
\draw[thick] (0.3, 1.7) -- (0.7, 1.3);
\draw[thick] (-1.3, 0.7) -- (-1.7, 0.3);
\draw[thick] (-0.7, 0.7) -- (-0.3, 0.3);
\draw[thick] (0.7, 0.7) -- (0.3, 0.3);
\draw[thick] (1.3, 0.7) -- (1.7, 0.3);
\node at (0,2) {0.2};
\node at (-1,1) {0.9};
\node at (1,1) {0.2};
\node at (-2,0) {1};
\node at (0,0) {0.9};
\node at (2,0) {0.5};
\end{scope}
\node at (10.5,1) {$\stackrel{\comp}{\longmapsto}$};
\begin{scope}[shift={(14,0)}]
\draw[thick] (-0.3, 1.7) -- (-0.7, 1.3);
\draw[thick] (0.3, 1.7) -- (0.7, 1.3);
\draw[thick] (-1.3, 0.7) -- (-1.7, 0.3);
\draw[thick] (-0.7, 0.7) -- (-0.3, 0.3);
\draw[thick] (0.7, 0.7) -- (0.3, 0.3);
\draw[thick] (1.3, 0.7) -- (1.7, 0.3);
\node at (0,2) {0.8};
\node at (-1,1) {0.1};
\node at (1,1) {0.8};
\node at (-2,0) {0};
\node at (0,0) {0.1};
\node at (2,0) {0.5};
\end{scope}
\node at (17.5,1) {$\stackrel{\opb^{-1}}{\longmapsto}$};
\begin{scope}[shift={(21,0)}]
\draw[thick] (-0.3, 1.7) -- (-0.7, 1.3);
\draw[thick] (0.3, 1.7) -- (0.7, 1.3);
\draw[thick] (-1.3, 0.7) -- (-1.7, 0.3);
\draw[thick] (-0.7, 0.7) -- (-0.3, 0.3);
\draw[thick] (0.7, 0.7) -- (0.3, 0.3);
\draw[thick] (1.3, 0.7) -- (1.7, 0.3);
\node at (0,2) {0};
\node at (-1,1) {0};
\node at (1,1) {0.3};
\node at (-2,0) {0};
\node at (0,0) {0.1};
\node at (2,0) {0.5};
\end{scope}

\begin{scope}[shift={(0,-4)}]
\draw[thick] (-0.3, 1.7) -- (-0.7, 1.3);
\draw[thick] (0.3, 1.7) -- (0.7, 1.3);
\draw[thick] (-1.3, 0.7) -- (-1.7, 0.3);
\draw[thick] (-0.7, 0.7) -- (-0.3, 0.3);
\draw[thick] (0.7, 0.7) -- (0.3, 0.3);
\draw[thick] (1.3, 0.7) -- (1.7, 0.3);
\node at (0,2) {0.2};
\node at (-1,1) {0.9};
\node at (1,1) {0.2};
\node at (-2,0) {1};
\node at (0,0) {0.9};
\node at (2,0) {0.5};
\end{scope}
\node at (3.5,-3) {$\stackrel{\comp}{\longmapsto}$};
\begin{scope}[shift={(7,-4)}]
\draw[thick] (-0.3, 1.7) -- (-0.7, 1.3);
\draw[thick] (0.3, 1.7) -- (0.7, 1.3);
\draw[thick] (-1.3, 0.7) -- (-1.7, 0.3);
\draw[thick] (-0.7, 0.7) -- (-0.3, 0.3);
\draw[thick] (0.7, 0.7) -- (0.3, 0.3);
\draw[thick] (1.3, 0.7) -- (1.7, 0.3);
\node at (0,2) {0.8};
\node at (-1,1) {0.1};
\node at (1,1) {0.8};
\node at (-2,0) {0};
\node at (0,0) {0.1};
\node at (2,0) {0.5};
\end{scope}
\node at (10.5,-3) {$\stackrel{\opb^{-1}}{\longmapsto}$};
\begin{scope}[shift={(14,-4)}]
\draw[thick] (-0.3, 1.7) -- (-0.7, 1.3);
\draw[thick] (0.3, 1.7) -- (0.7, 1.3);
\draw[thick] (-1.3, 0.7) -- (-1.7, 0.3);
\draw[thick] (-0.7, 0.7) -- (-0.3, 0.3);
\draw[thick] (0.7, 0.7) -- (0.3, 0.3);
\draw[thick] (1.3, 0.7) -- (1.7, 0.3);
\node at (0,2) {0};
\node at (-1,1) {0};
\node at (1,1) {0.3};
\node at (-2,0) {0};
\node at (0,0) {0.1};
\node at (2,0) {0.5};
\end{scope}
\node at (17.5,-3) {$\stackrel{\orb}{\longmapsto}$};
\begin{scope}[shift={(21,-4)}]
\draw[thick] (-0.3, 1.7) -- (-0.7, 1.3);
\draw[thick] (0.3, 1.7) -- (0.7, 1.3);
\draw[thick] (-1.3, 0.7) -- (-1.7, 0.3);
\draw[thick] (-0.7, 0.7) -- (-0.3, 0.3);
\draw[thick] (0.7, 0.7) -- (0.3, 0.3);
\draw[thick] (1.3, 0.7) -- (1.7, 0.3);
\node at (0,2) {0};
\node at (-1,1) {0};
\node at (1,1) {0.3};
\node at (-2,0) {0};
\node at (0,0) {0.4};
\node at (2,0) {0.8};
\end{scope}
\end{tikzpicture}
\end{center}
\end{ex}

Rowmotion on $OR(P)$ and $OP(P)$ has received much attention, particularly in~\cite{einpropp} and~\cite{grinberg-roby}.  For certain ``nice'' posets, $\row_{OR}$ has been shown to exhibit many of the same properties as $\row_\calj$.  For example, on a product of two chains $[a]\times[b]$, the order of rowmotion in
both the combinatorial ($\row_\calj,\row_\cala$)
and piecewise-linear ($\row_{OR},\row_\calc$) realms is $a+b$, and the homomesy of cardinality for $\row_\calj$ extends to $\row_{OR}$.  On the other hand, we will see in Subsection~\ref{subsec:zigzag} that a homomesy for $\row_\calj$ on zigzag posets~\cite[\S5]{indepsetspaper} does not extend in general to $\row_{OR}$.

We will focus primarily on $\row_\calc$ and $\row_{OR}$ since $\row_{OP}:OP(P)\ra OP(P)$ is equivalent to $\row_{OR}$ for the dual poset.
Some literature focuses more on $OP(P)$ as order-preserving maps may seem more natural to work with, as in Stanley's order polytope definition.  However, as $OR(P)$ generalizes order ideals, we will be consistent and focus on $OR(P)$.

As is clear by definition, there is a relation between $\row_\calc$ and $\row_{OR}$ depicted by the following commutative diagram.  This relation can also be seen in Example~\ref{ex:a3-row-cpl}, in which  the order-reversing labeling we started with is the one generated by the element of $\calc(P)$ we started with.

\begin{center}
\begin{tikzpicture}
\node at (0,1.8) {$\calc(P)$};
\node at (0,0) {$OR(P)$};
\node at (3.25,1.8) {$\calc(P)$};
\node at (3.25,0) {$OR(P)$};
\draw[semithick, ->] (0,1.3) -- (0,0.5);
\node[left] at (0,0.9) {$\orb$};
\draw[semithick, ->] (0.7,0) -- (2.5,0);
\node[below] at (1.5,0) {$\row_{OR}$};
\draw[semithick, ->] (0.7,1.8) -- (2.5,1.8);
\node[above] at (1.5,1.8) {$\row_\calc$};
\draw[semithick, ->] (3.25,1.3) -- (3.25,0.5);
\node[right] at (3.25,0.9) {$\orb$};
\end{tikzpicture}
\end{center}

\subsection{Toggles on poset polytopes}\label{subsec:tog-poly}

Toggles on $OR(P)$ and $OP(P)$, referred to as \emph{piecewise-linear toggles}, have been explored by Einstein and Propp~\cite{einpropp} and by Grinberg and Roby~\cite{grinberg-roby}, who have taken the concept further and analyzed birational toggling also.  In this section, we define toggles on the chain polytope $\calc(P)$.  We prove that almost all of the algebraic properties from Section~\ref{sec:combinatorial} relating toggles on $\cala(P)$ to those on $\calj(P)$ also hold for the piecewise-linear setting $\calc(P)$ and $OR(P)$.

\begin{prop}\label{prop:heavy machinery} Let $f\in OR(P)$, $e\in P$,
$L=\max\limits_{y\gtrdot e} f(y)$, and $R=\min\limits_{y\lessdot e} f(y)$.
Let $h: \hat{P}\ra [0,1]$ be defined by
$$h(x) =
\left\{\begin{array}{ll}
f(x) &\text{if }x\not=e\\
L+R - f(e)
&\text{if }x=e\end{array}\right..
$$
\begin{enumerate}
\item If $f\in OR(P)$, then $h\in OR(P)$.
\item If $f\in \calj(P)$, then $h=t_e(f)$.
\end{enumerate}
\end{prop}

Recall that we can extend to the poset $\hat{P}$ when necessary.  So if $e$ is a maximal element of $P$, then $L=f\left(\hat{M}\right)=0$ and if $e$ is a minimal element of $P$, then
$R=f\left(\hat{m}\right)=1$.

\begin{proof}\hspace{-8 in}\begin{tiny}.\end{tiny}

\begin{enumerate}
\item Note that $h(x)$ and $f(x)$ can only differ if $x=e$, so $h\in OR(P)$ if and only if $h(e)\in [L,R]$.
Since $f(e)\in[L,R]$, $h(e)=L+R-f(e)\in [L,R]$. Thus, $h\in OR(P)$.

\item Since $f\in \calj(P)$, it follows that $L,R,f(e)\in\{0,1\}$ and $L\leq f(e)\leq R$.

\noindent\textbf{Case 1: } $L=R$. Then $f(e)=L=R$ so $h(e)=f(e)+f(e)-f(e)=f(e)$.  If $L=1$, then some $y\gtrdot e$ is in the order ideal $f$, so applying $t_e$ does not change $f$.  If $R=0$, then some $y\lessdot e$ is not in the order ideal $f$, so likewise applying $t_e$ does not change $f$.  So $t_e(f)=h$.

\noindent\textbf{Case 2: } $L\not=R$. Then $L=0$ and $R=1$ so all elements covered by $e$ are in $f$ and no element that covers $e$ is in $f$.  So $t_e$ changes the label of $e$ between 0 and 1.  Since $h(e)=1-f(e)=(t_e(f))(e)$, it follows that $t_e(f)=h$.

\end{enumerate}
\end{proof}

\begin{defn}
A \textbf{maximal} chain of $P$ is a chain that cannot be extended into a longer chain, i.e., a chain that starts at a minimal element, uses only cover relations, and ends at a maximal element.  
For each $e\in P$, let $\mc_e(P)$ denote the set of all maximal chains $(y_1,\dots,y_k)$ in $P$ that contain $e$ as some $y_i$.  That is,
$$\mc_e(P) = 
\left\{
(y_1,\dots,y_k)\; \left|\; \hat{m} \lessdot y_1 \lessdot y_2 \lessdot \cdots \lessdot y_k \lessdot \hat{M}, e=y_i \text{ for some }i\right.
\right\}.
$$
\end{defn}

\begin{prop}\label{prop:castle machinery} Let $g\in \calc(P)$, $e\in P$,
and let $h:P\ra [0,1]$ be defined by
$$h(x) =
\left\{\begin{array}{ll}
g(x) &\text{if }x\not=e\\
1 - \max\left\{\left.
\sum\limits_{i=1}^k g(y_i) \right| (y_1,\dots,y_k)\in\mc_e(P)
\right\}
&\text{if }x=e\end{array}\right..
$$
\begin{enumerate}
\item If $g\in \calc(P)$, then $h\in \calc(P)$.
\item If $g\in \cala(P)$, then $h=\tau_e(g)$.
\end{enumerate}
\end{prop}

\begin{proof}\hspace{-8 in}\begin{tiny}.\end{tiny}
\begin{enumerate}
\item Let $g\in\calc(P)$.  Since $h(x)=g(x)$ unless $x=e$, we only need to confirm that $h(e)\geq 0$ and that
$\sum\limits_{i=1}^k h(y_i)\leq 1$  for all
$(y_1,\dots,y_k)\in\mc_e(P)$.  Since $\sum\limits_{i=1}^k g(y_i)\leq 1$ for all chains containing $e$, $h(e)\geq 0$.  Also for any $(y_1,\dots,y_k)\in\mc_e(P)$,
\begin{align*}
\sum\limits_{i=1}^k h(y_i) &= h(e)-g(e) + \sum\limits_{i=1}^k g(y_i)
\\&=
1-\max\left\{\left.
\sum\limits_{i=1}^\ell g(z_i) \right| (z_1,\dots,z_\ell)\in\mc_e(P)
\right\} - g(e)+\sum\limits_{i=1}^k g(y_i)
\\&\leq
1-g(e)
\\&\leq 1.
\end{align*}
So $h\in\calc(P)$.

\item Let $g\in\cala(P)$.  If no $y$ that is comparable with $e$ (including $e$ itself) is in $g$, then
$\max\left\{\left.
\sum\limits_{i=1}^k g(y_i) \right| (y_1,\dots,y_k)\in\mc_e(P)
\right\}=0$.  In this case, $h(e)=1=(\tau_e(g))(e)$.
Otherwise, some $y$ comparable with $e$ (possibly $e$ itself) is in the antichain $g$.  Then $e$ is not in $\tau_e(g)$, either by removing $e$ from $g$ or by the inability to insert $e$ into $g$.  In this case, $\max\left\{\left.
\sum\limits_{i=1}^k g(y_i) \right| (y_1,\dots,y_k)\in\mc_e(P)
\right\}=1$, so $h(e)=0=(\tau_e(g))(e)$.
\end{enumerate}
\end{proof}

As we have just shown, we can extend our earlier definitions of $t_e: \calj(P) \ra \calj(P)$ and $\tau_e: \cala(P) \ra \cala(P)$ to $t_e: OR(P) \ra OR(P)$ and $\tau_e: \calc(P) \ra \calc(P)$ below in Definitions~\ref{def:cpl-t} and~\ref{def:cpl-tau}.
While $t_e: OR(P) \ra OR(P)$ and $\tau_e: \calc(P) \ra \calc(P)$ are now continuous and piecewise-linear functions, they correspond exactly to the earlier definitions when restricted to $\calj(P)$ and $\cala(P)$. So it is not ambiguous to use the same notation for the combinatorial and piecewise-linear toggles.

\begin{defn}[\cite{einpropp}]\label{def:cpl-t}
Given $f\in OR(P)$ and $e\in P$, let
$t_e(f):\hat{P}\ra [0,1]$ be defined by
$$(t_e(f))(x) =
\left\{\begin{array}{ll}
f(x) &\text{if }x\not=e\\
\max\limits_{y\gtrdot e} f(y) + 
\min\limits_{y\lessdot e} f(y) - f(e)
&\text{if }x=e\end{array}\right..
$$
This defines a map $t_e: OR(P)\ra OR(P)$ because of Proposition~\ref{prop:heavy machinery}. The group $\tog_{OR}(P)$ generated by $\{t_e\;|\; e\in P\}$ is called the \textbf{toggle group} of $OR(P)$.
(Each $t_e$ is an involution and thus invertible as we will prove in Proposition~\ref{prop:cpl-inv-commute}(1), so we do obtain a group.)
\end{defn}

\begin{ex}\label{ex:cpl-3x3}
For the poset $P$ with elements named as on the left, we consider $f\in OR(P)$.  The dashed lines indicate $f\left(\hat{m}\right)$ and $f\left(\hat{M}\right)$ and their position within $\hat{P}$.  Then
$(t_E(f))(E)=\max(0.1,0.1) + \min(0.5,0.7) - 0.4 =0.2$
and
$(t_A(f))(A)=\max(0.5,0.7) + 1 - 0.9 =0.8$.
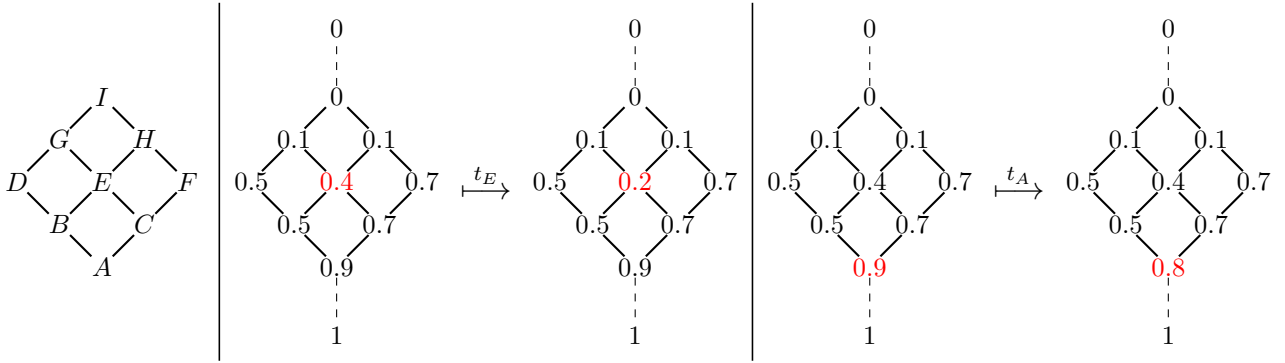
\begin{center}
\begin{tikzpicture}[scale=0.567]
\begin{scope}
\draw[thick] (-0.25, 1.75) -- (-0.75, 1.25);
\draw[thick] (0.25, 1.75) -- (0.75, 1.25);
\draw[thick] (-1.25, 0.75) -- (-1.75, 0.25);
\draw[thick] (-0.75, 0.75) -- (-0.25, 0.25);
\draw[thick] (0.75, 0.75) -- (0.25, 0.25);
\draw[thick] (1.25, 0.75) -- (1.75, 0.25);
\draw[thick] (-0.25, -1.75) -- (-0.75, -1.25);
\draw[thick] (0.25, -1.75) -- (0.75, -1.25);
\draw[thick] (-1.25, -0.75) -- (-1.75, -0.25);
\draw[thick] (-0.75, -0.75) -- (-0.25, -0.25);
\draw[thick] (0.75, -0.75) -- (0.25, -0.25);
\draw[thick] (1.25, -0.75) -- (1.75, -0.25);
\node at (0,2) {\footnotesize{$I$}};
\node at (-1,1) {\footnotesize{$G$}};
\node at (1,1) {\footnotesize{$H$}};
\node at (-2,0) {\footnotesize{$D$}};
\node at (0,0) {\footnotesize{$E$}};
\node at (2,0) {\footnotesize{$F$}};
\node at (-1,-1) {\footnotesize{$B$}};
\node at (1,-1) {\footnotesize{$C$}};
\node at (0,-2) {\footnotesize{$A$}};
\draw[semithick] (2.75,-4.2) -- (2.75,4.2);
\draw[semithick] (15.25,-4.2) -- (15.25,4.2);
\end{scope}
\begin{scope}[shift={(5.5,0)}]
\draw[thick] (-0.25, 1.75) -- (-0.75, 1.25);
\draw[thick] (0.25, 1.75) -- (0.75, 1.25);
\draw[thick] (-1.25, 0.75) -- (-1.75, 0.25);
\draw[thick] (-0.75, 0.75) -- (-0.25, 0.25);
\draw[thick] (0.75, 0.75) -- (0.25, 0.25);
\draw[thick] (1.25, 0.75) -- (1.75, 0.25);
\draw[thick] (-0.25, -1.75) -- (-0.75, -1.25);
\draw[thick] (0.25, -1.75) -- (0.75, -1.25);
\draw[thick] (-1.25, -0.75) -- (-1.75, -0.25);
\draw[thick] (-0.75, -0.75) -- (-0.25, -0.25);
\draw[thick] (0.75, -0.75) -- (0.25, -0.25);
\draw[thick] (1.25, -0.75) -- (1.75, -0.25);
\draw[dashed] (0,2.25) -- (0,3.35);
\draw[dashed] (0,-2.25) -- (0,-3.35);
\node at (0,2) {\footnotesize{0}};
\node at (-1,1) {\footnotesize{0.1}};
\node at (1,1) {\footnotesize{0.1}};
\node at (-2,0) {\footnotesize{0.5}};
\node at (0,0) {\footnotesize{{\color{red} 0.4}}};
\node at (2,0) {\footnotesize{0.7}};
\node at (-1,-1) {\footnotesize{0.5}};
\node at (1,-1) {\footnotesize{0.7}};
\node at (0,-2) {\footnotesize{0.9}};
\node at (0,3.6) {\footnotesize{0}};
\node at (0,-3.6) {\footnotesize{1}};
\node at (3.5,0) {$\stackrel{t_E}{\longmapsto}$};
\end{scope}
\begin{scope}[shift={(12.5,0)}]
\draw[thick] (-0.25, 1.75) -- (-0.75, 1.25);
\draw[thick] (0.25, 1.75) -- (0.75, 1.25);
\draw[thick] (-1.25, 0.75) -- (-1.75, 0.25);
\draw[thick] (-0.75, 0.75) -- (-0.25, 0.25);
\draw[thick] (0.75, 0.75) -- (0.25, 0.25);
\draw[thick] (1.25, 0.75) -- (1.75, 0.25);
\draw[thick] (-0.25, -1.75) -- (-0.75, -1.25);
\draw[thick] (0.25, -1.75) -- (0.75, -1.25);
\draw[thick] (-1.25, -0.75) -- (-1.75, -0.25);
\draw[thick] (-0.75, -0.75) -- (-0.25, -0.25);
\draw[thick] (0.75, -0.75) -- (0.25, -0.25);
\draw[thick] (1.25, -0.75) -- (1.75, -0.25);
\draw[dashed] (0,2.25) -- (0,3.35);
\draw[dashed] (0,-2.25) -- (0,-3.35);
\node at (0,2) {\footnotesize{0}};
\node at (-1,1) {\footnotesize{0.1}};
\node at (1,1) {\footnotesize{0.1}};
\node at (-2,0) {\footnotesize{0.5}};
\node at (0,0) {\footnotesize{{\color{red} 0.2}}};
\node at (2,0) {\footnotesize{0.7}};
\node at (-1,-1) {\footnotesize{0.5}};
\node at (1,-1) {\footnotesize{0.7}};
\node at (0,-2) {\footnotesize{0.9}};
\node at (0,3.6) {\footnotesize{0}};
\node at (0,-3.6) {\footnotesize{1}};
\end{scope}
\begin{scope}[shift={(18,0)}]
\draw[thick] (-0.25, 1.75) -- (-0.75, 1.25);
\draw[thick] (0.25, 1.75) -- (0.75, 1.25);
\draw[thick] (-1.25, 0.75) -- (-1.75, 0.25);
\draw[thick] (-0.75, 0.75) -- (-0.25, 0.25);
\draw[thick] (0.75, 0.75) -- (0.25, 0.25);
\draw[thick] (1.25, 0.75) -- (1.75, 0.25);
\draw[thick] (-0.25, -1.75) -- (-0.75, -1.25);
\draw[thick] (0.25, -1.75) -- (0.75, -1.25);
\draw[thick] (-1.25, -0.75) -- (-1.75, -0.25);
\draw[thick] (-0.75, -0.75) -- (-0.25, -0.25);
\draw[thick] (0.75, -0.75) -- (0.25, -0.25);
\draw[thick] (1.25, -0.75) -- (1.75, -0.25);
\draw[dashed] (0,2.25) -- (0,3.35);
\draw[dashed] (0,-2.25) -- (0,-3.35);
\node at (0,2) {\footnotesize{0}};
\node at (-1,1) {\footnotesize{0.1}};
\node at (1,1) {\footnotesize{0.1}};
\node at (-2,0) {\footnotesize{0.5}};
\node at (0,0) {\footnotesize{0.4}};
\node at (2,0) {\footnotesize{0.7}};
\node at (-1,-1) {\footnotesize{0.5}};
\node at (1,-1) {\footnotesize{0.7}};
\node at (0,-2) {\footnotesize{{\color{red} 0.9}}};
\node at (0,3.6) {\footnotesize{0}};
\node at (0,-3.6) {\footnotesize{1}};
\node at (3.5,0) {$\stackrel{t_A}{\longmapsto}$};
\end{scope}
\begin{scope}[shift={(25,0)}]
\draw[thick] (-0.25, 1.75) -- (-0.75, 1.25);
\draw[thick] (0.25, 1.75) -- (0.75, 1.25);
\draw[thick] (-1.25, 0.75) -- (-1.75, 0.25);
\draw[thick] (-0.75, 0.75) -- (-0.25, 0.25);
\draw[thick] (0.75, 0.75) -- (0.25, 0.25);
\draw[thick] (1.25, 0.75) -- (1.75, 0.25);
\draw[thick] (-0.25, -1.75) -- (-0.75, -1.25);
\draw[thick] (0.25, -1.75) -- (0.75, -1.25);
\draw[thick] (-1.25, -0.75) -- (-1.75, -0.25);
\draw[thick] (-0.75, -0.75) -- (-0.25, -0.25);
\draw[thick] (0.75, -0.75) -- (0.25, -0.25);
\draw[thick] (1.25, -0.75) -- (1.75, -0.25);
\draw[dashed] (0,2.25) -- (0,3.35);
\draw[dashed] (0,-2.25) -- (0,-3.35);
\node at (0,2) {\footnotesize{0}};
\node at (-1,1) {\footnotesize{0.1}};
\node at (1,1) {\footnotesize{0.1}};
\node at (-2,0) {\footnotesize{0.5}};
\node at (0,0) {\footnotesize{0.4}};
\node at (2,0) {\footnotesize{0.7}};
\node at (-1,-1) {\footnotesize{0.5}};
\node at (1,-1) {\footnotesize{0.7}};
\node at (0,-2) {\footnotesize{{\color{red} 0.8}}};
\node at (0,3.6) {\footnotesize{0}};
\node at (0,-3.6) {\footnotesize{1}};
\end{scope}
\end{tikzpicture}
\end{center}
\end{ex}

We do not define toggles for $\hat{m}$ and $\hat{M}$; the values $f\left(\hat{m}\right)$ and $f\left(\hat{M}\right)$ are fixed across all of $OR(P)$.  We could generalize $OR(P)$ and $OP(P)$ and consider the order polytopes of $[a,b]$-labelings where the values $f\left(\hat{m}\right)$ and $f\left(\hat{M}\right)$ are set to any $a<b$.  However these polytopes are just linear rescalings of $OR(P)$ and $OP(P)$.  Furthermore, Einstein and Propp have extended these toggles from acting on order polytopes to acting on $\rr^{\hat{P}}$~\cite{einpropp}.

\begin{defn}\label{def:cpl-tau}
Given $g\in\calc(P)$ and $e\in P$, let $\tau_e(g):P\ra[0,1]$ be defined by
$$(\tau_e(g))(x) =
\left\{\begin{array}{ll}
g(x) &\text{if }x\not=e\\
1 - \max\left\{\left.
\sum\limits_{i=1}^k g(y_i) \right| (y_1,\dots,y_k)\in\mc_e(P)
\right\}
&\text{if }x=e\end{array}\right..
$$
This defines a map $\tau_e:\calc(P)\ra\calc(P)$ because of Proposition~\ref{prop:castle machinery}. The group $\tog_\calc(P)$ generated by $\{\tau_e\;|\; e\in P\}$ is called the \textbf{toggle group} of $\calc(P)$.
(Each $\tau_e$ is an involution and thus invertible as we will prove in Proposition~\ref{prop:cpl-inv-commute}(1), so we do obtain a group.)
\end{defn}

Every chain in $\mc_e(P)$ can be split into segments below $e$, $e$ itself, and above $e$, and we can take the maximum sum of $g$ on each part.  So an equivalent formula for $(\tau_e(g))(e)$ is
\begin{equation}\label{eq:alt_tau_e}
(\tau_e(g))(e) = 1-\max\limits_{y\lessdot e} (\opb(g))(y) - g(e) - \max\limits_{y\gtrdot e} (\orb(g))(y).
\end{equation}
In Eq.~(\ref{eq:alt_tau_e}), note that we regard $\opb(g)$ as an order-preserving labeling,
so $(\opb(g))(\hat{m})=0$.  Similarly, we regard $\orb(g)$ as
an order-reversing labeling, so $(\orb(g))\left(\hat{M}\right)=0$.

Also, note that since any $g\in\calc(P)$ has nonnegative labels, it would be equivalent in the definition of $\tau_e$ to use the set of all chains of $P$ through $e$, instead of the set $\mc_e(P)$ of maximal chains through $e$.  We will use the definition with maximal chains for various reasons.  For one, it gives us far fewer chains to worry about in computations, as in the following example.
%Note from the order conditions on $\opb(g)$ and $\orb(g)$ that $\max\limits_{y<e} (\opb(g))(y)$ occurs at an element $y\lessdot e$ and that $\max\limits_{y>e} (\orb(g))(y)$ occurs at an element $y\gtrdot e$.  That is why we can use cover relations in Eq.~\ref{eq:alt_tau_e}.

\begin{ex}
For the poset $P$ with elements named as on the left, we consider $g\in \calc(P)$.  Then summing the outputs of $g$ along the maximal chains
through $F$, we get the following.

\begin{align*}
g(A) + g(C) + g(F) + g(G) + g(I) &=
0.2+0+0.1+0.1+0.1 &= 0.5\\
g(A) + g(C) + g(F) + g(H) + g(J) &=
0.2+0+0.1+0.2+0.1 &= 0.6\\
g(A) + g(C) + g(F) + g(H) + g(K) &= 
0.2+0+0.1+0.2+0 &= 0.5\\
g(B) + g(F) + g(G) + g(I) &= 
0.3+0.1+0.1+0.1 &= 0.6\\
g(B) + g(F) + g(H) + g(J) &= 
0.3+0.1+0.2+0.1 &= 0.7\\
g(B) + g(F) + g(H) + g(K) &= 
0.3+0.1+0.2+0 &= 0.6
\end{align*}

So $\tau_F$ changes the output value of $F$ to
$1-\max(0.5,0.6,0.5,0.6,0.7,0.6)=1-0.7=0.3$.
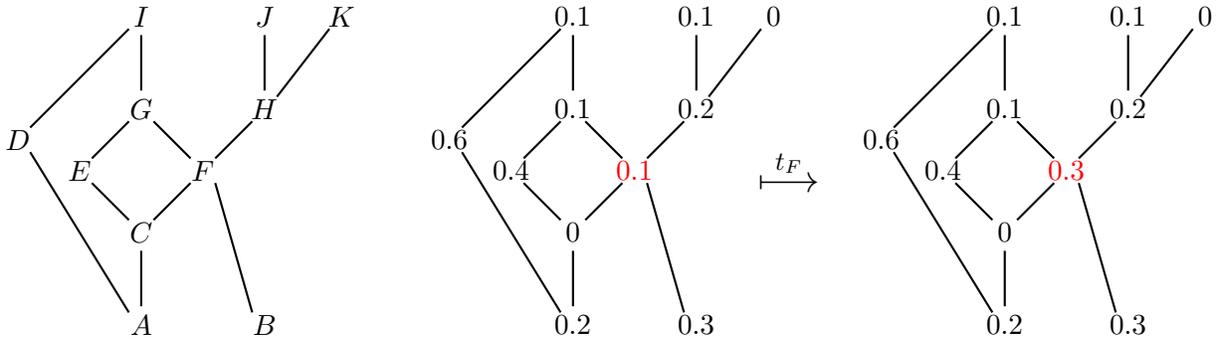
\begin{center}
\begin{tikzpicture}[scale=0.82]
\begin{scope}
\draw[thick] (0,0.3) -- (0,1.2);
\draw[thick] (-0.2,0.2) -- (-1.8,2.8);
\draw[thick] (-0.2,1.7) -- (-0.8,2.3);
\draw[thick] (0.2,1.7) -- (0.8,2.3);
\draw[thick] (-0.2,3.3) -- (-0.8,2.7);
\draw[thick] (0.2,3.3) -- (0.8,2.7);
\draw[thick] (1.8,0.2) -- (1.2,2.3);
\draw[thick] (0,3.8) -- (0,4.7);
\draw[thick] (-1.8,3.2) -- (-0.2,4.8);
\draw[thick] (1.8,3.3) -- (1.2,2.7);
\draw[thick] (2,3.8) -- (2,4.7);
\draw[thick] (2.2,3.7) -- (3.05,4.8);
\node at (0,0) {\small{$A$}};
\node at (2,0) {\small{$B$}};
\node at (0,1.5) {\small{$C$}};
\node at (-2,3) {\small{$D$}};
\node at (-1,2.5) {\small{$E$}};
\node at (1,2.5) {\small{$F$}};
\node at (0,3.5) {\small{$G$}};
\node at (2,3.5) {\small{$H$}};
\node at (0,5) {\small{$I$}};
\node at (2,5) {\small{$J$}};
\node at (3.25,5) {\small{$K$}};
\end{scope}
\begin{scope}[shift={(7,0)}]
\draw[thick] (0,0.3) -- (0,1.2);
\draw[thick] (-0.2,0.2) -- (-1.8,2.8);
\draw[thick] (-0.2,1.7) -- (-0.8,2.3);
\draw[thick] (0.2,1.7) -- (0.8,2.3);
\draw[thick] (-0.2,3.3) -- (-0.8,2.7);
\draw[thick] (0.2,3.3) -- (0.8,2.7);
\draw[thick] (1.8,0.2) -- (1.2,2.3);
\draw[thick] (0,3.8) -- (0,4.7);
\draw[thick] (-1.8,3.2) -- (-0.2,4.8);
\draw[thick] (1.8,3.3) -- (1.2,2.7);
\draw[thick] (2,3.8) -- (2,4.7);
\draw[thick] (2.2,3.7) -- (3.05,4.8);
\node at (0,0) {\small{0.2}};
\node at (2,0) {\small{0.3}};
\node at (0,1.5) {\small{0}};
\node at (-2,3) {\small{0.6}};
\node at (-1,2.5) {\small{0.4}};
\node at (1,2.5) {\small{{\color{red}0.1}}};
\node at (0,3.5) {\small{0.1}};
\node at (2,3.5) {\small{0.2}};
\node at (0,5) {\small{0.1}};
\node at (2,5) {\small{0.1}};
\node at (3.25,5) {\small{0}};
\node at (3.5,2.5) {\large{$\stackrel{t_F}{\longmapsto}$}};
\end{scope}
\begin{scope}[shift={(14,0)}]
\draw[thick] (0,0.3) -- (0,1.2);
\draw[thick] (-0.2,0.2) -- (-1.8,2.8);
\draw[thick] (-0.2,1.7) -- (-0.8,2.3);
\draw[thick] (0.2,1.7) -- (0.8,2.3);
\draw[thick] (-0.2,3.3) -- (-0.8,2.7);
\draw[thick] (0.2,3.3) -- (0.8,2.7);
\draw[thick] (1.8,0.2) -- (1.2,2.3);
\draw[thick] (0,3.8) -- (0,4.7);
\draw[thick] (-1.8,3.2) -- (-0.2,4.8);
\draw[thick] (1.8,3.3) -- (1.2,2.7);
\draw[thick] (2,3.8) -- (2,4.7);
\draw[thick] (2.2,3.7) -- (3.05,4.8);
\node at (0,0) {\small{0.2}};
\node at (2,0) {\small{0.3}};
\node at (0,1.5) {\small{0}};
\node at (-2,3) {\small{0.6}};
\node at (-1,2.5) {\small{0.4}};
\node at (1,2.5) {\small{{\color{red}0.3}}};
\node at (0,3.5) {\small{0.1}};
\node at (2,3.5) {\small{0.2}};
\node at (0,5) {\small{0.1}};
\node at (2,5) {\small{0.1}};
\node at (3.25,5) {\small{0}};
\end{scope}
\end{tikzpicture}
\end{center}
\end{ex}

We now show that most of the algebraic properties of $t_e$ and $\tau_e$ we proved for the combinatorial setting extend to the piecewise-linear setting.

\begin{prop}\label{prop:cpl-inv-commute}\hspace{-8 in}.
\begin{enumerate}
\item For $x\in P$, $t_x$ and $\tau_x$ are involutions.
\item Two toggles $t_x,t_y$ commute if and only if neither $x$ nor $y$ covers the other.
\item Two toggles $\tau_x,\tau_y$ commute if and only if $x=y$ or $x\parallel y$.
\end{enumerate}
\end{prop}

\begin{proof}\hspace{-8 in}.
\begin{enumerate}
\item We start with $t_x$.  Let $f\in OR(P)$.  Then
$t_x$ does not change the label for any vertex other than $x$, and
\begin{align*}
\left(t_x^2(f)\right)(x)&=
\max\limits_{y\gtrdot x} (t_x(f))(y) + 
\min\limits_{y\lessdot x} (t_x(f))(y) - (t_x(f))(x)
\\&=
\max\limits_{y\gtrdot x} f(y) + 
\min\limits_{y\lessdot x} f(y) - 
\left(\max\limits_{y\gtrdot x} f(y) + 
\min\limits_{y\lessdot x} f(y) - f(x)
\right)
\\&= f(x)
\end{align*}
so $t_x^2(f)=f$. Now we consider $\tau_x$.  Let $g\in\calc(P)$ and $h=\tau_x(g)$.  Again, $\tau_x$ does not change the label for any vertex other than $x$ so it suffices to show that $(\tau_x(h))(x)=g(x)$. By Eq.~(\ref{eq:alt_tau_e}),
\begin{align*}
(\tau_x(h))(x)
&= 1-\max\limits_{y\lessdot x} (\opb(h))(y) - h(x) - \max\limits_{y\gtrdot x} (\orb(h))(y)
\\&=
1-\max\limits_{y\lessdot x} (\opb(g))(y)-\\& 
\left(1-\max\limits_{y\lessdot x} (\opb(g))(y) - g(x) - \max\limits_{y\gtrdot x} (\orb(g))(y)\right)
- \max\limits_{y\gtrdot x} (\orb(g))(y)
\\&=g(x).
\end{align*}
Note we are able to replace $h$ with $g$ in the second equality above for expressions that do not include inputting $x$ into $g$ or $h$ (the only input where $g$ and $h$ can differ).

\item If $x\lessdot y$ or $y\lessdot x$, then $t_x t_y\not=t_y t_x$ when restricted to $\calj(P)$ by Proposition~\ref{prop:J-toggle-inv-commute}, so they are also unequal over the larger set $OR(P)$.
If $x=y$, then $t_xt_y=t_xt_x=t_yt_x$.
Now suppose neither $x$ nor $y$ covers the other and $x\not=y$.
Only the label of $x$ can be changed by $t_x$ and only the label of $y$ can be changed by $t_y$.
For $f\in OR(P)$, the definition of $(t_x(f))(x)$ only involves elements that cover $x$, are covered by $x$, and $x$ itself.
This is similar for $y$ in $(t_y(f))(y)$. Thus, the label of $x$ has no effect on what $t_y$ does and the label of $y$ has no effect on what $t_x$ does.  So $t_x t_y=t_y t_x$.

\item If $x=y$, then $\tau_x\tau_y=\tau_x\tau_x=\tau_y\tau_x$.
If $x$ and $y$ are comparable and unequal, then $\tau_x\tau_y\not=\tau_y\tau_x$ when restricted to $\cala(P)$ by Proposition~\ref{prop:tau-commute}, so they are also unequal over the larger set $\calc(P)$.
Now suppose $x\parallel y$.
Only the label of $x$ can be changed by $\tau_x$ and only the label of $y$ can be changed by $\tau_y$.
No chain contains both $x$ and $y$, so the label of $x$ has no effect on what $\tau_y$ does and the label of $y$ has no effect on what $\tau_x$ does.  Thus, $\tau_x\tau_y=\tau_y\tau_x$.
\end{enumerate}
\end{proof}

\begin{defn}
For $e\in P$ and $S\subseteq P$, we define $t_e^*\in \tog_\calc(P)$, $\eta_S\in \tog_{OR}(P)$,\\$\eta_e\in \tog_{OR}(P)$, and $\tau_e^*\in \tog_{OR}(P)$ exactly
as we defined them in $\tog_\cala(P)$ and $\tog_\calj(P)$:
\begin{itemize}
\item Let $t_e^* := \tau_{e_1} \tau_{e_2}\cdots \tau_{e_k} \tau_e \tau_{e_1} \tau_{e_2}\cdots \tau_{e_k}$ where $e_1,\dots,e_k$ are the elements of $P$ covered by $e$.
\item Let $\eta_S:=t_{x_1}t_{x_2}\cdots t_{x_k}$
where $(x_1,x_2,\dots,x_k)$ is a linear extension of the subposet $\{x\in P \;|\; x<y,y\in S\}$ of $P$.
\item Let $\eta_e:=\eta_{\{e\}}$.
\item Let $\tau_e^* := \eta_e t_e \eta_e^{-1}$.
\end{itemize}
\end{defn}

The following is an analogue of Theorems~\ref{thm:t-star} and~\ref{thm:tau-star}.

\begin{thm}\label{thm:iso-cpl}
For any $e\in P$, the following diagrams commute.
So there is an isomorphism from $\tog_\calc(P)$ to $\tog_{OR}(P)$ given by $\tau_e\mapsto \tau_e^*$, and inverse $t_e\mapsto t_e^*$.
\begin{center}
\phantom{1}
\hfill
\begin{tikzpicture}
\node at (0,1.8) {$\calc(P)$};
\node at (0,0) {$OR(P)$};
\node at (3.25,1.8) {$\calc(P)$};
\node at (3.25,0) {$OR(P)$};
\draw[semithick, ->] (0,1.3) -- (0,0.5);
\node[left] at (0,0.9) {$\orb$};
\draw[semithick, ->] (0.7,0) -- (2.5,0);
\node[below] at (1.5,0) {$t_e$};
\draw[semithick, ->] (0.7,1.8) -- (2.5,1.8);
\node[above] at (1.5,1.8) {$t_e^*$};
\draw[semithick, ->] (3.25,1.3) -- (3.25,0.5);
\node[right] at (3.25,0.9) {$\orb$};
\end{tikzpicture}
\hfill
\begin{tikzpicture}
\node at (0,1.8) {$\calc(P)$};
\node at (0,0) {$OR(P)$};
\node at (3.25,1.8) {$\calc(P)$};
\node at (3.25,0) {$OR(P)$};
\draw[semithick, ->] (0,1.3) -- (0,0.5);
\node[left] at (0,0.9) {$\orb$};
\draw[semithick, ->] (0.7,0) -- (2.5,0);
\node[below] at (1.5,0) {$\tau_e^*$};
\draw[semithick, ->] (0.7,1.8) -- (2.5,1.8);
\node[above] at (1.5,1.8) {$\tau_e$};
\draw[semithick, ->] (3.25,1.3) -- (3.25,0.5);
\node[right] at (3.25,0.9) {$\orb$};
\end{tikzpicture}
\hfill
\phantom{1}
\end{center}
\end{thm}

See Figure~\ref{fig:iso-cpl} for an example demonstrating the first commutative diagram in this theorem.

\begin{figure}
\begin{tikzpicture}[scale=0.455]
\draw[thick] (-0.25, 1.75) -- (-0.75, 1.25);
\draw[thick] (0.25, 1.75) -- (0.75, 1.25);
\draw[thick] (-1.25, 0.75) -- (-1.75, 0.25);
\draw[thick] (-0.75, 0.75) -- (-0.25, 0.25);
\draw[thick] (0.75, 0.75) -- (0.25, 0.25);
\draw[thick] (1.25, 0.75) -- (1.75, 0.25);
\draw[thick] (-0.25, -1.75) -- (-0.75, -1.25);
\draw[thick] (0.25, -1.75) -- (0.75, -1.25);
\draw[thick] (-1.25, -0.75) -- (-1.75, -0.25);
\draw[thick] (-0.75, -0.75) -- (-0.25, -0.25);
\draw[thick] (0.75, -0.75) -- (0.25, -0.25);
\draw[thick] (1.25, -0.75) -- (1.75, -0.25);
\node at (0,2) {\footnotesize{0}};
\node at (-1,1) {\footnotesize{0.1}};
\node at (1,1) {\footnotesize{0.1}};
\node at (-2,0) {\footnotesize{{\color{red}0.4}}};
\node at (0,0) {\footnotesize{0.3}};
\node at (2,0) {\footnotesize{0.6}};
\node at (-1,-1) {\footnotesize{0}};
\node at (1,-1) {\footnotesize{0}};
\node at (0,-2) {\footnotesize{0.2}};
\draw[thick, ->] (2.75,0) -- (3.5,0);
\node[below] at (3.125,0) {$\tau_{D}$};
\begin{scope}[shift={(6.25,0)}]
\draw[thick] (-0.25, 1.75) -- (-0.75, 1.25);
\draw[thick] (0.25, 1.75) -- (0.75, 1.25);
\draw[thick] (-1.25, 0.75) -- (-1.75, 0.25);
\draw[thick] (-0.75, 0.75) -- (-0.25, 0.25);
\draw[thick] (0.75, 0.75) -- (0.25, 0.25);
\draw[thick] (1.25, 0.75) -- (1.75, 0.25);
\draw[thick] (-0.25, -1.75) -- (-0.75, -1.25);
\draw[thick] (0.25, -1.75) -- (0.75, -1.25);
\draw[thick] (-1.25, -0.75) -- (-1.75, -0.25);
\draw[thick] (-0.75, -0.75) -- (-0.25, -0.25);
\draw[thick] (0.75, -0.75) -- (0.25, -0.25);
\draw[thick] (1.25, -0.75) -- (1.75, -0.25);
\node at (0,2) {\footnotesize{0}};
\node at (-1,1) {\footnotesize{0.1}};
\node at (1,1) {\footnotesize{0.1}};
\node at (-2,0) {\footnotesize{0.3}};
\node at (0,0) {\footnotesize{{\color{red}0.3}}};
\node at (2,0) {\footnotesize{0.6}};
\node at (-1,-1) {\footnotesize{0}};
\node at (1,-1) {\footnotesize{0}};
\node at (0,-2) {\footnotesize{0.2}};
\draw[thick, ->] (2.75,0) -- (3.5,0);
\node[below] at (3.125,0) {$\tau_{E}$};
\end{scope}
\begin{scope}[shift={(12.5,0)}]
\draw[thick] (-0.25, 1.75) -- (-0.75, 1.25);
\draw[thick] (0.25, 1.75) -- (0.75, 1.25);
\draw[thick] (-1.25, 0.75) -- (-1.75, 0.25);
\draw[thick] (-0.75, 0.75) -- (-0.25, 0.25);
\draw[thick] (0.75, 0.75) -- (0.25, 0.25);
\draw[thick] (1.25, 0.75) -- (1.75, 0.25);
\draw[thick] (-0.25, -1.75) -- (-0.75, -1.25);
\draw[thick] (0.25, -1.75) -- (0.75, -1.25);
\draw[thick] (-1.25, -0.75) -- (-1.75, -0.25);
\draw[thick] (-0.75, -0.75) -- (-0.25, -0.25);
\draw[thick] (0.75, -0.75) -- (0.25, -0.25);
\draw[thick] (1.25, -0.75) -- (1.75, -0.25);
\node at (0,2) {\footnotesize{0}};
\node at (-1,1) {\footnotesize{{\color{red}0.1}}};
\node at (1,1) {\footnotesize{0.1}};
\node at (-2,0) {\footnotesize{0.3}};
\node at (0,0) {\footnotesize{0.4}};
\node at (2,0) {\footnotesize{0.6}};
\node at (-1,-1) {\footnotesize{0}};
\node at (1,-1) {\footnotesize{0}};
\node at (0,-2) {\footnotesize{0.2}};
\draw[thick, ->] (2.75,0) -- (3.5,0);
\node[below] at (3.125,0) {$\tau_{G}$};
\end{scope}
\begin{scope}[shift={(18.75,0)}]
\draw[thick] (-0.25, 1.75) -- (-0.75, 1.25);
\draw[thick] (0.25, 1.75) -- (0.75, 1.25);
\draw[thick] (-1.25, 0.75) -- (-1.75, 0.25);
\draw[thick] (-0.75, 0.75) -- (-0.25, 0.25);
\draw[thick] (0.75, 0.75) -- (0.25, 0.25);
\draw[thick] (1.25, 0.75) -- (1.75, 0.25);
\draw[thick] (-0.25, -1.75) -- (-0.75, -1.25);
\draw[thick] (0.25, -1.75) -- (0.75, -1.25);
\draw[thick] (-1.25, -0.75) -- (-1.75, -0.25);
\draw[thick] (-0.75, -0.75) -- (-0.25, -0.25);
\draw[thick] (0.75, -0.75) -- (0.25, -0.25);
\draw[thick] (1.25, -0.75) -- (1.75, -0.25);
\node at (0,2) {\footnotesize{0}};
\node at (-1,1) {\footnotesize{0.3}};
\node at (1,1) {\footnotesize{0.1}};
\node at (-2,0) {\footnotesize{{\color{red}0.3}}};
\node at (0,0) {\footnotesize{0.4}};
\node at (2,0) {\footnotesize{0.6}};
\node at (-1,-1) {\footnotesize{0}};
\node at (1,-1) {\footnotesize{0}};
\node at (0,-2) {\footnotesize{0.2}};
\draw[thick, ->] (2.75,0) -- (3.5,0);
\node[below] at (3.125,0) {$\tau_{D}$};
\end{scope}
\begin{scope}[shift={(25,0)}]
\draw[thick] (-0.25, 1.75) -- (-0.75, 1.25);
\draw[thick] (0.25, 1.75) -- (0.75, 1.25);
\draw[thick] (-1.25, 0.75) -- (-1.75, 0.25);
\draw[thick] (-0.75, 0.75) -- (-0.25, 0.25);
\draw[thick] (0.75, 0.75) -- (0.25, 0.25);
\draw[thick] (1.25, 0.75) -- (1.75, 0.25);
\draw[thick] (-0.25, -1.75) -- (-0.75, -1.25);
\draw[thick] (0.25, -1.75) -- (0.75, -1.25);
\draw[thick] (-1.25, -0.75) -- (-1.75, -0.25);
\draw[thick] (-0.75, -0.75) -- (-0.25, -0.25);
\draw[thick] (0.75, -0.75) -- (0.25, -0.25);
\draw[thick] (1.25, -0.75) -- (1.75, -0.25);
\node at (0,2) {\footnotesize{0}};
\node at (-1,1) {\footnotesize{0.3}};
\node at (1,1) {\footnotesize{0.1}};
\node at (-2,0) {\footnotesize{0.2}};
\node at (0,0) {\footnotesize{{\color{red}0.4}}};
\node at (2,0) {\footnotesize{0.6}};
\node at (-1,-1) {\footnotesize{0}};
\node at (1,-1) {\footnotesize{0}};
\node at (0,-2) {\footnotesize{0.2}};
\draw[thick, ->] (2.75,0) -- (3.5,0);
\node[below] at (3.125,0) {$\tau_{E}$};
\end{scope}
\begin{scope}[shift={(31.25,0)}]
\draw[thick] (-0.25, 1.75) -- (-0.75, 1.25);
\draw[thick] (0.25, 1.75) -- (0.75, 1.25);
\draw[thick] (-1.25, 0.75) -- (-1.75, 0.25);
\draw[thick] (-0.75, 0.75) -- (-0.25, 0.25);
\draw[thick] (0.75, 0.75) -- (0.25, 0.25);
\draw[thick] (1.25, 0.75) -- (1.75, 0.25);
\draw[thick] (-0.25, -1.75) -- (-0.75, -1.25);
\draw[thick] (0.25, -1.75) -- (0.75, -1.25);
\draw[thick] (-1.25, -0.75) -- (-1.75, -0.25);
\draw[thick] (-0.75, -0.75) -- (-0.25, -0.25);
\draw[thick] (0.75, -0.75) -- (0.25, -0.25);
\draw[thick] (1.25, -0.75) -- (1.75, -0.25);
\node at (0,2) {\footnotesize{0}};
\node at (-1,1) {\footnotesize{0.3}};
\node at (1,1) {\footnotesize{0.1}};
\node at (-2,0) {\footnotesize{0.2}};
\node at (0,0) {\footnotesize{0.1}};
\node at (2,0) {\footnotesize{0.6}};
\node at (-1,-1) {\footnotesize{0}};
\node at (1,-1) {\footnotesize{0}};
\node at (0,-2) {\footnotesize{0.2}};
\end{scope}
\draw[thick, ->] (0,-2.75) -- (0,-4.25);
\draw[thick, ->] (31.25,-2.75) -- (31.25,-4.25);
\node[left] at (0,-3.5) {$\orb$};
\node[right] at (31.25,-3.5) {$\orb$};
\begin{scope}[shift={(0,-7)}]
\draw[thick] (-0.25, 1.75) -- (-0.75, 1.25);
\draw[thick] (0.25, 1.75) -- (0.75, 1.25);
\draw[thick] (-1.25, 0.75) -- (-1.75, 0.25);
\draw[thick] (-0.75, 0.75) -- (-0.25, 0.25);
\draw[thick] (0.75, 0.75) -- (0.25, 0.25);
\draw[thick] (1.25, 0.75) -- (1.75, 0.25);
\draw[thick] (-0.25, -1.75) -- (-0.75, -1.25);
\draw[thick] (0.25, -1.75) -- (0.75, -1.25);
\draw[thick] (-1.25, -0.75) -- (-1.75, -0.25);
\draw[thick] (-0.75, -0.75) -- (-0.25, -0.25);
\draw[thick] (0.75, -0.75) -- (0.25, -0.25);
\draw[thick] (1.25, -0.75) -- (1.75, -0.25);
\node at (0,2) {\footnotesize{0}};
\node at (-1,1) {\footnotesize{{\color{red}0.1}}};
\node at (1,1) {\footnotesize{0.1}};
\node at (-2,0) {\footnotesize{0.5}};
\node at (0,0) {\footnotesize{0.4}};
\node at (2,0) {\footnotesize{0.7}};
\node at (-1,-1) {\footnotesize{0.5}};
\node at (1,-1) {\footnotesize{0.7}};
\node at (0,-2) {\footnotesize{0.9}};
\draw[thick, ->] (2.75,0) -- (28.5,0);
\node[below] at (15.625,0) {$t_G$};
\end{scope}
\begin{scope}[shift={(31.25,-7)}]
\draw[thick] (-0.25, 1.75) -- (-0.75, 1.25);
\draw[thick] (0.25, 1.75) -- (0.75, 1.25);
\draw[thick] (-1.25, 0.75) -- (-1.75, 0.25);
\draw[thick] (-0.75, 0.75) -- (-0.25, 0.25);
\draw[thick] (0.75, 0.75) -- (0.25, 0.25);
\draw[thick] (1.25, 0.75) -- (1.75, 0.25);
\draw[thick] (-0.25, -1.75) -- (-0.75, -1.25);
\draw[thick] (0.25, -1.75) -- (0.75, -1.25);
\draw[thick] (-1.25, -0.75) -- (-1.75, -0.25);
\draw[thick] (-0.75, -0.75) -- (-0.25, -0.25);
\draw[thick] (0.75, -0.75) -- (0.25, -0.25);
\draw[thick] (1.25, -0.75) -- (1.75, -0.25);
\node at (0,2) {\footnotesize{0}};
\node at (-1,1) {\footnotesize{{\color{red}0.3}}};
\node at (1,1) {\footnotesize{0.1}};
\node at (-2,0) {\footnotesize{0.5}};
\node at (0,0) {\footnotesize{0.4}};
\node at (2,0) {\footnotesize{0.7}};
\node at (-1,-1) {\footnotesize{0.5}};
\node at (1,-1) {\footnotesize{0.7}};
\node at (0,-2) {\footnotesize{0.9}};
%\draw[thick, ->] (1.5,0.4) -- (36,0.4);
%\node[below] at (18.75,0.4) {$t_e$};
\end{scope}
\end{tikzpicture}
\caption{In this example of Theorem~\ref{thm:iso-cpl}, the poset elements are named as in Example~\ref{ex:cpl-3x3}, so $t_G^*=\tau_E\tau_D\tau_G\tau_E\tau_D$.}
\label{fig:iso-cpl}
\end{figure}

\begin{proof}
%[Proof of Theorem~\ref{thm:iso-cpl}]
We begin with the left commutative diagram.  Let $g\in\calc(P)$.  We must show that $\orb(t_e^* g)=t_e(\orb(g))$.

Throughout the proof, we several times make use of the fact $\orb$ only ``looks up'' while $\opb$ only ``looks down.''
By this we mean, for any $x\in P$, the value of
$(\orb(g))(x)$ depends only on $g(y)$ for $y\geq x$, whereas
$(\opb(g))(x)$ depends only on $g(y)$ for $y\leq x$.

Suppose $e$ is a minimal element of $P$.  Then
$t_e^*=\tau_e$.  By the definition of $\orb$ and minimality of $e$, $(\orb(\tau_e g))(x)=(\orb(g))(x)$ for all $x\not=e$.  Thus,
$(\orb(\tau_e g))(x)=(t_e(\orb(g)))(x)$ for $x\not=e$, so we only have to check at $x=e$.
Since $e$ is minimal, $(\tau_e g)(e)=1-(\orb(g))(e)$
as can be seen from the definitions.
By Eq.~(\ref{eq:recur-orb}),
\begin{align*}
(\orb(\tau_e g))(e)&= (\tau_e g)(e) +
\max\limits_{y\gtrdot e} (\orb(\tau_e g))(y)
\\&=
1-(\orb(g))(e) +
\max\limits_{y\gtrdot e} (\orb(g))(y)\\
&=
(\orb(g))\left(\hat{m}\right)-(\orb(g))(e) +
\max\limits_{y\gtrdot e} (\orb(g))(y)\\
&=
\max\limits_{y\gtrdot e} (\orb(g))(y)+
\min\limits_{y\lessdot e} (\orb(g))(y) -(\orb(g))(e)
\\&=
\big(t_e(\orb(g))\big)(e).
\end{align*}

Now assume $e$ is not minimal in $P$.  Let $e_1,\dots,e_k$ be the elements that $e$ covers. Let
\begin{align*}
g' &= \tau_e\tau_{e_1}\cdots \tau_{e_k}g,\\
g'' &= \tau_{e_1}\cdots \tau_{e_k}g'=t_e^* g,\\
f &= \orb(g),\\
f' &= \orb(g'),\\
f'' &= \orb(g'').
\end{align*}
The goal is to show that $f''=t_e f$.  Note that $g,g',g''$ can only possibly differ in the labels of $e,e_1,e_2,\dots,e_k$.  From the definition of $\orb$ and the fact that $t_e$ can only change the label of $e$, it follows that $t_e f$ and $f''$ can only possibly differ in the labels of elements $\leq e$.

We begin by proving $f''(e)=(t_e f)(e)$.  From Eq.~(\ref{eq:alt_tau_e}) and from the fact that $e_1,e_2,\dots,e_k$ are
pairwise incomparable so each chain can contain at most one of them,
\begin{align}\label{eqarr:tomb time}
(\tau_{e_1}\cdots \tau_{e_k}g)(e_j)
&=
1-\max\limits_{y\lessdot e_j} (\opb(g))(y) \underbrace{- g(e_j) - \max\limits_{y\gtrdot e_j} (\orb(g))(y)}
_{-f(e_j) \text{ by Eq.~(\ref{eq:recur-orb})}}
\\
&=
1-\max\limits_{y\lessdot e_j} (\opb(g))(y)-f(e_j)\nonumber
\end{align}
for $1\leq j\leq k$.  Then to get $g'(e)$, we apply
Eq.~(\ref{eq:alt_tau_e}) to
$\tau_{e_1}\cdots \tau_{e_k}g$ instead of $g$, yielding
\begin{align*}
g''(e)=g'(e)
&=
1-\max\limits_{e_i\lessdot e} (\opb(\tau_{e_1}\cdots \tau_{e_k}g))(e_i) - (\tau_{e_1}\cdots \tau_{e_k}g)(e) - \max\limits_{y\gtrdot e}
\underbrace{(\orb(\tau_{e_1}\cdots \tau_{e_k}g))(y)}_
{(\orb(g))(y) \text{ as $\orb$ only looks up}}
\\
&=
1-\max\limits_{e_i\lessdot e} (\opb(\tau_{e_1}\cdots \tau_{e_k}g))(e_i)
\underbrace{- g(e) - \max\limits_{y\gtrdot e} (\orb(g))(y)}
_{-f(e) \text{ by Eq.~(\ref{eq:recur-orb})}}
\\
&=
1-\max\limits_{e_i\lessdot e} (\opb(\tau_{e_1}\cdots \tau_{e_k}g))(e_i)-f(e)
\\&=
1-\max\limits_{e_i\lessdot e}
\underbrace{\left(
\max\limits_{y\lessdot e_i} (\opb(\tau_{e_1}\cdots \tau_{e_k}g))(y)
+ (\tau_{e_1}\cdots \tau_{e_k}g)(e_i)\right)}_
{\text{from Eq. (\ref{eq:recur-opb})}}
-f(e)
\\&=
1-\max\limits_{e_i\lessdot e}
\underbrace{\left(
\max\limits_{y\lessdot e_i} (\opb(g))(y)
+ (\tau_{e_1}\cdots \tau_{e_k}g)(e_i)
\right)}_
{\text{since $\opb$ only looks down}}
-f(e)
\\&=
1-\max\limits_{e_i\lessdot e}
\Big(
\max\limits_{y\lessdot e_i} (\opb(g))(y)
+ \underbrace{1-\max\limits_{y\lessdot e_i} (\opb(g))(y)-f(e_i)}_
{\text{from Eq.~(\ref{eqarr:tomb time})}}
\Big)
-f(e)
\\&=
\min\limits_{e_i\lessdot e}f(e_i) - f(e).
\end{align*}
Then using Eq.~(\ref{eq:recur-orb})
\begin{align*}
f''(e)&=
g''(e)+\max\limits_{y\gtrdot e}f''(y)
\\&=
g''(e)+\max\limits_{y\gtrdot e}f(y)
\\&=
\min\limits_{e_i\lessdot e}f(e_i) - f(e)
+\max\limits_{y\gtrdot e}f(y)
\\&=
(t_e f)(e).
\end{align*}
Now we will prove that $f''(x)=f(x)=(t_e f)(x)$ for every $x<e$ using 
downward induction on $x$. So we begin with the base case $x\lessdot e$.
From Eq.~(\ref{eq:alt_tau_e}),
\begin{align*}
g''(e_j)
&=
1-\max\limits_{y\lessdot e_j} (\opb(g'))(y) - g'(e_j) - \max\limits_{y\gtrdot e_j} (\orb(g'))(y)
\\
&=
1-\max\limits_{y\lessdot e_j} (\opb(g))(y) - \underbrace{\Big(
1-\max\limits_{y\lessdot e_j} (\opb(g))(y)-f(e_j)
\Big)}_
{\text{from Eq.~(\ref{eqarr:tomb time})}}
- \max\limits_{y\gtrdot e_j} f'(y)
\\&=
f(e_j)-\max\limits_{y\gtrdot e_j} f'(y)
\\&=
f(e_j)-\max\limits_{y\gtrdot e_j} f''(y)
\end{align*}
for $1\leq j\leq k$.
Note that the last equality is because $f'(y)$ and $f''(y)$ only depend on $g'(x)$ and $g''(x)$ for $x\geq y$.  Since $g'(x)=g''(x)$ for $x\geq y>e_j$, $f'(y)=f''(y)$ for such $y$.
Continuing, Eq.~(\ref{eq:recur-orb}) yields
\begin{align*}
f''(e_j)&=
g''(e_j)+\max\limits_{y\gtrdot e_j} f''(y)\\
&=\underbrace{f(e_j)-\max\limits_{y\gtrdot e_j} f''(y)}_
{\text{from above}}
+\max\limits_{y\gtrdot e_j} f''(y)
\\
&=f(e_j)\\
&=(t_e f)(e_j).
\end{align*}

Now let $x<e$ and $x\not\in\{e_1,e_2,\dots,e_k\}$.  Assume
(as induction hypothesis) 
that $f''(y)=f(y)=(t_e  f)(y)$ for every $y$ covering $x$ (which cannot include $y=e$ since $x\not\in\{e_1,\dots,e_k\}$).
Again using Eq.~(\ref{eq:recur-orb}),
\begin{align*}
f''(x)&=g''(x) +\max\limits_{y\gtrdot x}f''(y)\\
&= g(x)+\max\limits_{y\gtrdot x}f(y)\\
&=f(x)\\
&= (t_e f)(x).
\end{align*}
For the second equality above, recall that $g(x)=g''(x)$ because
$x\not=e,e_1,\dots,e_k$.

This concludes the proof of the left commutative diagram.

The right commutative diagram is an analogue of Theorem~\ref{thm:tau-star}.  The proof of that theorem (as well as Lemma~\ref{lem:tau-star}) only depended on algebraic properties of $\tog_\cala(P)$ and $\tog_\calj(P)$, namely that toggles are involutions, when toggles commute, and Theorem~\ref{thm:t-star}.  We have proven analogues for these to $\tog_\calc(P)$ and $\tog_{OR}(P)$ in Proposition~\ref{prop:cpl-inv-commute} and this theorem's first commutative diagram.  Thus the proof of the second commutative diagram is the same as that of Theorem~\ref{thm:tau-star}.
\end{proof}

We will not prove the following piecewise-linear analogue of Proposition~\ref{prop:row-toggles} here.  The result is essentially~\cite[Thm.~4.2]{einpropp}.
In that paper, piecewise-linear rowmotion is defined in terms of toggles and proven to be equivalent to the composition of three maps (our definition of $\row_{OR}$).
There they are defining rowmotion on $OP(P)$ not $OR(P)$, so there is a change of notation between this paper and~\cite{einpropp}, given by $\Theta=\comp$, $\rho_P=\row_{OP}^{-1}=\comp\circ\row_{OR}\circ\comp$, $\nabla=\opb^{-1}$, $\Delta=\orb^{-1}$.

\begin{thm}[{\cite[Thm.~4.2]{einpropp}}]\label{thm:row-OR}
Let $(x_1,x_2,\dots,x_n)$ be any linear extension of a finite poset $P$.  Then $\row_{OR}=t_{x_1} t_{x_2} \cdots t_{x_n}$.
\end{thm}

We use this to prove a similar expression about $\row_\calc$ that is analogous to Proposition~\ref{prop:row-toggles-anti}.

\begin{thm}\label{thm:row-C}
Let $(x_1,x_2,\dots,x_n)$ be any linear extension of a finite poset $P$.  Then $\row_\calc=\tau_{x_n}\cdots \tau_{x_2} \tau_{x_1}$.
\end{thm}

\begin{proof}
The isomorphism from $\tog_\calc(P)$ to $\tog_{OR}(P)$ given by $\tau_e \mapsto \tau_e^*$ sends $\row_\calc$ to $\row_{OR}$.  This is from Theorem~\ref{thm:iso-cpl} and the commutative diagram at the end of Subsection~\ref{subsec:row-polytope}.  Therefore, it suffices to show that
$\tau_{x_n}^*\cdots \tau_{x_2}^* \tau_{x_1}^*=\row_{OR} = t_{x_1} t_{x_2} \cdots t_{x_n}$.
We will use induction to prove that $\tau_{x_k}^*\cdots \tau_{x_2}^* \tau_{x_1}^*
=t_{x_1} t_{x_2} \cdots t_{x_k}$ for $1\leq k\leq n$.

For the base case, $\tau_{x_1}^*=t_{x_1}$ since $x_1$ is a minimal element of $P$.
For the induction hypothesis, let $1\leq k\leq n-1$ and assume that
$\tau_{x_k}^*\cdots \tau_{x_2}^* \tau_{x_1}^*
=t_{x_1} t_{x_2} \cdots t_{x_k}$.
Then 
\begin{equation}\label{eq:future frenzy}
\tau_{x_{k+1}}^*\tau_{x_k}^*\cdots \tau_{x_2}^* \tau_{x_1}^*
= \eta_{x_{k+1}} t_{x_{k+1}} \eta_{x_{k+1}}^{-1} t_{x_1} t_{x_2} \cdots t_{x_k}.
\end{equation}
Let $(y_1,\dots,y_{k'})$ be a linear extension of the subposet $\{y\in P \;|\; y<x_{k+1}\}$ of $P$.
Then since $(x_1,\dots,x_n)$ is a linear extension of $P$, all of $y_1,\dots,y_{k'}$ must be in $\{x_1,\dots,x_k\}$.
Furthermore, any element less than one of $y_1,\dots,y_{k'}$ must be less than $x_{k+1}$ so none of the elements of $\{x_1,\dots,x_k\}$ outside of $\{y_1,\dots,y_{k'}\}$ are less than any of $y_1,\dots,y_{k'}$.
Therefore, we can name these elements in such a way that $(y_1,\dots,y_{k'}, y_{k'+1}, \dots, y_k)$ is a linear extension of $\{x_1,\dots,x_k\}$.
We remind the reader of Remark~\ref{rem:etienne-84}: any two linear extensions of a poset differ by a sequence of swaps between adjacent incomparable elements~\cite{etienne-84}.
Toggles of incomparable elements commute so $t_{x_1}t_{x_2}\cdots t_{x_k} = t_{y_1}\cdots t_{y_{k'}}t_{y_{k'+1}}\cdots t_{y_k}$.
From Eq.~(\ref{eq:future frenzy}) and $\eta_{x_{k+1}}=t_{y_1}\cdots t_{y_{k'}}$, we obtain
\begin{align*}
\tau_{x_{k+1}}^*\tau_{x_k}^*\cdots \tau_{x_2}^* \tau_{x_1}^* &=
\eta_{x_{k+1}} t_{x_{k+1}} \eta_{x_{k+1}}^{-1} t_{x_1} t_{x_2} \cdots t_{x_k}
\\&=
t_{y_1}\cdots t_{y_{k'}} t_{x_{k+1}} t_{y_{k'}}\cdots t_{y_1} t_{y_1}\cdots t_{y_{k'}}t_{y_{k'+1}}\cdots t_{y_k}
\\&=
t_{y_1}\cdots t_{y_{k'}} t_{x_{k+1}} t_{y_{k'+1}}\cdots t_{y_k}
\\&=
t_{y_1}\cdots t_{y_{k'}} t_{y_{k'+1}}\cdots t_{y_k}  t_{x_{k+1}}
\\&=
t_{x_1} t_{x_2} \cdots t_{x_k} t_{x_{k+1}}.
\end{align*}
In the fourth equality above, we could move $t_{x_{k+1}}$ to the right of $t_{y_{k'+1}}\cdots t_{y_k}$ because $x_{k+1}$ is incomparable with each of $y_{k'+1},\dots, y_k$.  This is because none of these are less than $x_{k+1}$ by design nor greater than $x_{k+1}$ by position within the linear extension $(x_1,\dots,x_n)$ of $P$.

By induction, we have $\tau_{x_n}^*\cdots \tau_{x_2}^* \tau_{x_1}^* = t_{x_1} t_{x_2} \cdots t_{x_n} = \row_{OR}$ so
$\tau_{x_n}\cdots \tau_{x_2} \tau_{x_1} = \row_\calc$.
\end{proof}

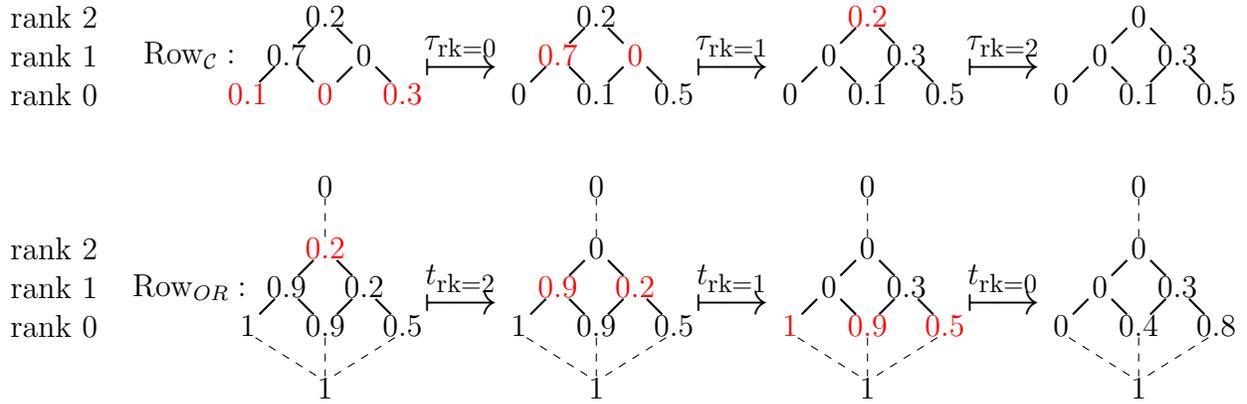
\begin{figure}
\begin{center}
\begin{tikzpicture}[scale=0.515]
\node at (-3.5,1) {$\row_\calc:$};
\node at (-3.5,-5) {$\row_{OR}:$};
\node at (-7,2) {rank 2};
\node at (-7,1) {rank 1};
\node at (-7,0) {rank 0};
\node at (-7,-4) {rank 2};
\node at (-7,-5) {rank 1};
\node at (-7,-6) {rank 0};
\begin{scope}
\draw[thick] (-0.3, 1.7) -- (-0.7, 1.3);
\draw[thick] (0.3, 1.7) -- (0.7, 1.3);
\draw[thick] (-1.3, 0.7) -- (-1.7, 0.3);
\draw[thick] (-0.7, 0.7) -- (-0.3, 0.3);
\draw[thick] (0.7, 0.7) -- (0.3, 0.3);
\draw[thick] (1.3, 0.7) -- (1.7, 0.3);
\node at (0,2) {0.2};
\node at (-1,1) {0.7};
\node at (1,1) {0};
\node at (-2,0) {{\color{red}0.1}};
\node at (0,0) {{\color{red}0}};
\node at (2,0) {{\color{red}0.3}};
\end{scope}
\node at (3.5,1) {\Large{$\stackrel{\normalsize{\tau_{\rk=0}}}{\longmapsto}$}};
\begin{scope}[shift={(7,0)}]
\draw[thick] (-0.3, 1.7) -- (-0.7, 1.3);
\draw[thick] (0.3, 1.7) -- (0.7, 1.3);
\draw[thick] (-1.3, 0.7) -- (-1.7, 0.3);
\draw[thick] (-0.7, 0.7) -- (-0.3, 0.3);
\draw[thick] (0.7, 0.7) -- (0.3, 0.3);
\draw[thick] (1.3, 0.7) -- (1.7, 0.3);
\node at (0,2) {0.2};
\node at (-1,1) {{\color{red}0.7}};
\node at (1,1) {{\color{red}0}};
\node at (-2,0) {0};
\node at (0,0) {0.1};
\node at (2,0) {0.5};
\end{scope}
\node at (10.5,1) {\Large{$\stackrel{\normalsize{\tau_{\rk=1}}}{\longmapsto}$}};
\begin{scope}[shift={(14,0)}]
\draw[thick] (-0.3, 1.7) -- (-0.7, 1.3);
\draw[thick] (0.3, 1.7) -- (0.7, 1.3);
\draw[thick] (-1.3, 0.7) -- (-1.7, 0.3);
\draw[thick] (-0.7, 0.7) -- (-0.3, 0.3);
\draw[thick] (0.7, 0.7) -- (0.3, 0.3);
\draw[thick] (1.3, 0.7) -- (1.7, 0.3);
\node at (0,2) {{\color{red}0.2}};
\node at (-1,1) {0};
\node at (1,1) {0.3};
\node at (-2,0) {0};
\node at (0,0) {0.1};
\node at (2,0) {0.5};
\end{scope}
\node at (17.5,1) {\Large{$\stackrel{\normalsize{\tau_{\rk=2}}}{\longmapsto}$}};
\begin{scope}[shift={(21,0)}]
\draw[thick] (-0.3, 1.7) -- (-0.7, 1.3);
\draw[thick] (0.3, 1.7) -- (0.7, 1.3);
\draw[thick] (-1.3, 0.7) -- (-1.7, 0.3);
\draw[thick] (-0.7, 0.7) -- (-0.3, 0.3);
\draw[thick] (0.7, 0.7) -- (0.3, 0.3);
\draw[thick] (1.3, 0.7) -- (1.7, 0.3);
\node at (0,2) {0};
\node at (-1,1) {0};
\node at (1,1) {0.3};
\node at (-2,0) {0};
\node at (0,0) {0.1};
\node at (2,0) {0.5};
\end{scope}

\begin{scope}[shift={(0,-6)}]
\draw[thick] (-0.3, 1.7) -- (-0.7, 1.3);
\draw[thick] (0.3, 1.7) -- (0.7, 1.3);
\draw[thick] (-1.3, 0.7) -- (-1.7, 0.3);
\draw[thick] (-0.7, 0.7) -- (-0.3, 0.3);
\draw[thick] (0.7, 0.7) -- (0.3, 0.3);
\draw[thick] (1.3, 0.7) -- (1.7, 0.3);
\node at (0,2) {{\color{red}0.2}};
\node at (-1,1) {0.9};
\node at (1,1) {0.2};
\node at (-2,0) {1};
\node at (0,0) {0.9};
\node at (2,0) {0.5};
\draw[dashed] (0,2.3) -- (0,3.3);
\draw[dashed] (0,-0.3) -- (0,-1.3);
\draw[dashed] (-1.8,-0.3) -- (-0.2,-1.3);
\draw[dashed] (1.8,-0.3) -- (0.2,-1.3);
\node at (0,3.6) {0};
\node at (0,-1.6) {1};
\end{scope}
\node at (3.5,-5) {\Large{$\stackrel{\normalsize{t_{\rk=2}}}{\longmapsto}$}};
\begin{scope}[shift={(7,-6)}]
\draw[thick] (-0.3, 1.7) -- (-0.7, 1.3);
\draw[thick] (0.3, 1.7) -- (0.7, 1.3);
\draw[thick] (-1.3, 0.7) -- (-1.7, 0.3);
\draw[thick] (-0.7, 0.7) -- (-0.3, 0.3);
\draw[thick] (0.7, 0.7) -- (0.3, 0.3);
\draw[thick] (1.3, 0.7) -- (1.7, 0.3);
\node at (0,2) {0};
\node at (-1,1) {{\color{red}0.9}};
\node at (1,1) {{\color{red}0.2}};
\node at (-2,0) {1};
\node at (0,0) {0.9};
\node at (2,0) {0.5};
\draw[dashed] (0,2.3) -- (0,3.3);
\draw[dashed] (0,-0.3) -- (0,-1.3);
\draw[dashed] (-1.8,-0.3) -- (-0.2,-1.3);
\draw[dashed] (1.8,-0.3) -- (0.2,-1.3);
\node at (0,3.6) {0};
\node at (0,-1.6) {1};
\end{scope}
\node at (10.5,-5) {\Large{$\stackrel{\normalsize{t_{\rk=1}}}{\longmapsto}$}};
\begin{scope}[shift={(14,-6)}]
\draw[thick] (-0.3, 1.7) -- (-0.7, 1.3);
\draw[thick] (0.3, 1.7) -- (0.7, 1.3);
\draw[thick] (-1.3, 0.7) -- (-1.7, 0.3);
\draw[thick] (-0.7, 0.7) -- (-0.3, 0.3);
\draw[thick] (0.7, 0.7) -- (0.3, 0.3);
\draw[thick] (1.3, 0.7) -- (1.7, 0.3);
\node at (0,2) {0};
\node at (-1,1) {0};
\node at (1,1) {0.3};
\node at (-2,0) {{\color{red}1}};
\node at (0,0) {{\color{red}0.9}};
\node at (2,0) {{\color{red}0.5}};
\draw[dashed] (0,2.3) -- (0,3.3);
\draw[dashed] (0,-0.3) -- (0,-1.3);
\draw[dashed] (-1.8,-0.3) -- (-0.2,-1.3);
\draw[dashed] (1.8,-0.3) -- (0.2,-1.3);
\node at (0,3.6) {0};
\node at (0,-1.6) {1};
\end{scope}
\node at (17.5,-5) {\Large{$\stackrel{\normalsize{t_{\rk=0}}}{\longmapsto}$}};
\begin{scope}[shift={(21,-6)}]
\draw[thick] (-0.3, 1.7) -- (-0.7, 1.3);
\draw[thick] (0.3, 1.7) -- (0.7, 1.3);
\draw[thick] (-1.3, 0.7) -- (-1.7, 0.3);
\draw[thick] (-0.7, 0.7) -- (-0.3, 0.3);
\draw[thick] (0.7, 0.7) -- (0.3, 0.3);
\draw[thick] (1.3, 0.7) -- (1.7, 0.3);
\node at (0,2) {0};
\node at (-1,1) {0};
\node at (1,1) {0.3};
\node at (-2,0) {0};
\node at (0,0) {0.4};
\node at (2,0) {0.8};
\draw[dashed] (0,2.3) -- (0,3.3);
\draw[dashed] (0,-0.3) -- (0,-1.3);
\draw[dashed] (-1.8,-0.3) -- (-0.2,-1.3);
\draw[dashed] (1.8,-0.3) -- (0.2,-1.3);
\node at (0,3.6) {0};
\node at (0,-1.6) {1};
\end{scope}
\end{tikzpicture}
\end{center}
\caption{We demonstrate Theorems~\ref{thm:row-C} and~\ref{thm:row-OR} on this poset.  Since this is a graded poset, we can toggle by ranks at once.  Compare with Example~\ref{ex:a3-row-cpl}.}
\label{fig:row-cpl-rank}
\end{figure}

All of Subsection~\ref{subsec:gp} about graded posets also extends to the piecewise-linear toggling (with $\cala$, $\calj$, and $\bfI$ replaced with $\calc$, $OR$, and $\orb$ respectively) because those results all used algebraic properties that we have proven also hold for the piecewise-linear toggles.  In Figure~\ref{fig:row-cpl-rank}, we demonstrate $\row_\calc$ and $\row_{OR}$ in terms of toggles, but we toggle by ranks since this is a graded poset (like in Corollary~\ref{cor:row-rank}).

\subsection{Toggling the chain polytope of a zigzag poset}~\label{subsec:zigzag}

In~\cite{indepsetspaper}, Roby and the author analyze toggling within the set of independent sets of a path graph.  The set of independent sets of a path graph with $n$ vertices can easily be seen to be the same as the set of antichains of a zigzag poset with $n$ elements.

\begin{defn}[{\cite[p. 367]{ec1ed2}}]
The \textbf{zigzag poset} (or \textbf{fence poset}) with $n$ elements, denoted $\calz_n$, is the poset consisting of
elements $a_1,...,a_n$ and relations $a_{2i-1} < a_{2i}$ and $a_{2i+1} < a_{2i}$.
\end{defn}

Zigzag posets have Hasse diagrams that can be drawn in a zigzag formation.  For example,

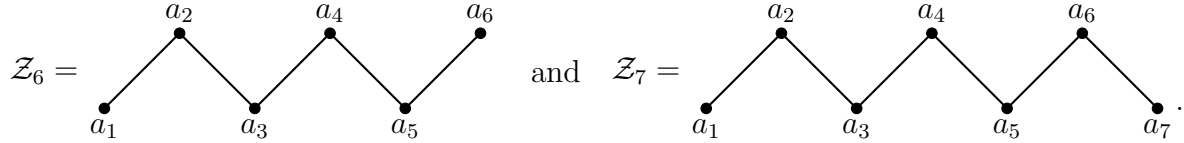
\begin{center}
\begin{tikzpicture}
\draw[thick] (0,0) -- (1,1) -- (2,0) -- (3,1) -- (4,0) -- (5,1);
\draw[fill] (0,0) circle [radius=0.07];
\draw[fill] (1,1) circle [radius=0.07];
\draw[fill] (2,0) circle [radius=0.07];
\draw[fill] (3,1) circle [radius=0.07];
\draw[fill] (4,0) circle [radius=0.07];
\draw[fill] (5,1) circle [radius=0.07];
\node[below] at (0,0) {$a_1$};
\node[above] at (1,1) {$a_2$};
\node[below] at (2,0) {$a_3$};
\node[above] at (3,1) {$a_4$};
\node[below] at (4,0) {$a_5$};
\node[above] at (5,1) {$a_6$};
\node[left] at (-0.2,0.5) {$\calz_6=$};
\node at (6,0.5) {and};
\end{tikzpicture}
\begin{tikzpicture}
\draw[thick] (0,0) -- (1,1) -- (2,0) -- (3,1) -- (4,0) -- (5,1) -- (6,0);
\draw[fill] (0,0) circle [radius=0.07];
\draw[fill] (1,1) circle [radius=0.07];
\draw[fill] (2,0) circle [radius=0.07];
\draw[fill] (3,1) circle [radius=0.07];
\draw[fill] (4,0) circle [radius=0.07];
\draw[fill] (5,1) circle [radius=0.07];
\draw[fill] (6,0) circle [radius=0.07];
\node[below] at (0,0) {$a_1$};
\node[above] at (1,1) {$a_2$};
\node[below] at (2,0) {$a_3$};
\node[above] at (3,1) {$a_4$};
\node[below] at (4,0) {$a_5$};
\node[above] at (5,1) {$a_6$};
\node[below] at (6,0) {$a_7$};
\node[left] at (-0.2,0.5) {$\calz_7=$};
\node[left] at (6.5,0) {.};
\end{tikzpicture}
\end{center}

The main results in~\cite{indepsetspaper} pertain to the homomesy phenomenon.  First isolated by Propp and Roby in~\cite{propproby}, this phenomenon has proven to be quite widespread in combinatorial
dynamical systems consisting of a set and invertible action.

\begin{defn}[\cite{propproby}]
Suppose we have a set $\cals$, an invertible map $w:\cals\ra \cals$ such that every $w$-orbit is finite, and a function (``statistic") $f:\cals\ra\kk$, where $\kk$ is a
field of characteristic 0.
If there exists a constant $c\in\kk$ such that for every $w$-orbit $\eso\subseteq \cals$, $$\frac{1}{\#\eso}\sum\limits_{x\in\eso}f(x)=c,$$
then we say
the statistic $f$
is \textbf{homomesic with average \textsl{c}} (or \textbf{\textsl{c}-mesic}
for short) under the action of $w$ on $\cals$.
%We also say the triple $(S,w,f)$ \textbf{exhibits homomesy with average \textsl{c}}.

%In this case, we say that the function $f$ is \textbf{homomesic with average} $\mathbf{\textsl{c}}$, or $\textbf{\textsl{c}-mesic}$, under the action of $w$ on $\cals$.
\end{defn}

Below we restate~\cite[Cor.~2.31]{indepsetspaper} in terms of
antichains of $\calz_n$.

\begin{thm}\label{thm:indep-sets}
Consider the zigzag poset $\calz_n$.
Let $w$ be a product of each of the antichain toggles
$\tau_{a_1},\dots,\tau_{a_n}$ each used exactly once in some order (called a Coxeter element), and consider the action of $w$ on $\cala(\calz_n)$.
For $1\leq j\leq n$, let $I_j:\cala(\calz_n) \ra \{0,1\}$
be the function defined
as
$$I_j(A) =
\left\{\begin{array}{ll}
0 &\text{if }a_j\not\in A,\\
1 &\text{if }a_j\in A.\\\end{array}\right.
$$
\begin{enumerate}
\item The statistic $I_j-I_{n+1-j}$ is 0-mesic for every
$1\leq j\leq n$.
\item The statistics $2I_1+I_2$ and $I_{n-1}+2I_n$ are both
1-mesic.
\end{enumerate}\vspace{0.1 in}
\end{thm}

\begin{figure}[htb]
\begin{center}
\begin{minipage}{.45\textwidth}
\centering
\begin{tabular}{c||c|c|c|c|c|c}
&$I_1$&$I_2$&$I_3$&$I_4$&$I_5$&$I_6$\\\hline\hline\\[-1.111em]
$             A $&1&0&0&1&0&0\\\hline\\[-1.111em]
$\varphi     (A)$&0&1&0&0&1&0\\\hline\\[-1.111em]
$\varphi^2   (A)$&0&0&1&0&0&1\\\hline\hline\\[-1.111em]
\textbf{Total} & \textbf{1} & \textbf{1} & \textbf{1} & \textbf{1} & \textbf{1} & \textbf{1}
\end{tabular}
\vspace{0.3 in}

\begin{tabular}{c||c|c|c|c|c|c}
&$I_1$&$I_2$&$I_3$&$I_4$&$I_5$&$I_6$\\\hline\hline\\[-1.111em]
$             B $&0&0&0&0&0&0\\\hline\\[-1.111em]
$\varphi     (B)$&1&0&1&0&1&0\\\hline\\[-1.111em]
$\varphi^2   (B)$&0&0&0&0&0&1\\\hline\\[-1.111em]
$\varphi^3   (B)$&1&0&1&0&0&0\\\hline\\[-1.111em]
$\varphi^4   (B)$&0&0&0&1&0&1\\\hline\\[-1.111em]
$\varphi^5   (B)$&1&0&0&0&0&0\\\hline\\[-1.111em]
$\varphi^6   (B)$&0&1&0&1&0&1\\\hline\hline\\[-1.111em]
\textbf{Total} & \textbf{3} & \textbf{1} & \textbf{2} & \textbf{2} & \textbf{1} & \textbf{3}
\end{tabular}
\end{minipage}
\begin{minipage}{.45\textwidth}
\centering
\begin{tabular}{c||c|c|c|c|c|c}
&$I_1$&$I_2$&$I_3$&$I_4$&$I_5$&$I_6$\\\hline\hline\\[-1.111em]
$             C $&1&0&1&0&0&1\\\hline\\[-1.111em]
$\varphi     (C)$&0&0&0&1&0&0\\\hline\\[-1.111em]
$\varphi^2   (C)$&1&0&0&0&1&0\\\hline\\[-1.111em]
$\varphi^3   (C)$&0&1&0&0&0&1\\\hline\\[-1.111em]
$\varphi^4   (C)$&0&0&1&0&0&0\\\hline\\[-1.111em]
$\varphi^5   (C)$&1&0&0&1&0&1\\\hline\\[-1.111em]
$\varphi^6   (C)$&0&1&0&0&0&0\\\hline\\[-1.111em]
$\varphi^7   (C)$&0&0&1&0&1&0\\\hline\\[-1.111em]
$\varphi^8   (C)$&1&0&0&0&0&1\\\hline\\[-1.111em]
$\varphi^9   (C)$&0&1&0&1&0&0\\\hline\\[-1.111em]
$\varphi^{10}(C)$&0&0&0&0&1&0\\\hline\hline\\[-1.111em]
\textbf{Total} & \textbf{4} & \textbf{3} & \textbf{3} & \textbf{3} & \textbf{3} & \textbf{4}
\end{tabular}
\end{minipage}
\end{center}
\caption{For the action
$\varphi=\tau_{a_6}\tau_{a_5}\tau_{a_4}\tau_{a_3}\tau_{a_2}\tau_{a_1}\in
\tog_\cala(\calz_6)$, there are three orbits.
The orbits have size 3, 7, and 11, so $\varphi^3(A)=A$, $\varphi^7(B)=B$,
and $\varphi^{11}(C)=C$.
All statistics in Theorem~\ref{thm:indep-sets} can be verified to be
homomesic for this particular action $\varphi$ on $\calz_6$.  See
Example~\ref{ex:3-7-11}.}
\label{fig:3-7-11}
\end{figure}

\begin{example}~\label{ex:3-7-11}
On $\calz_6$, let
$\varphi=\tau_{a_6}\tau_{a_5}\tau_{a_4}\tau_{a_3}\tau_{a_2}\tau_{a_1}\in
\tog_\cala(\calz_6)$
be the composition that toggles each element from left to right.
In Figure~\ref{fig:3-7-11}, the three $\varphi$-orbits are shown.
According to Theorem~\ref{thm:indep-sets}(2), the statistic
$2I_1+I_2$ is 1-mesic under the action of $\varphi$.  That is, $2I_1+I_2$ has average 1 across every orbit.  We can verify this by computing the averages
$$
\frac{2(1)+1}{3}=1,\hspace{0.55 in}
\frac{2(3)+1}{7}=1,\hspace{0.55 in}
\frac{2(4)+3}{11}=1.
$$
Also, Theorem~\ref{thm:indep-sets}(1) says that $I_2-I_5$ has average 0 across every orbit, which we can also verify
$$
\frac{1-1}{3}=0,\hspace{0.55 in}
\frac{1-1}{7}=0,\hspace{0.55 in}
\frac{3-3}{11}=0.
$$
\end{example}

As one may expect, some results for antichain toggling continue to hold
for chain polytope toggling, and some results do not.
The example in Figure~\ref{fig:counterexample} proves as a counterexample
for Theorem~\ref{thm:indep-sets}(1), e.g.~the average of $h\mapsto h(3)-h(6)$ across this orbit is $\frac{1}{20}\left(\frac{13}{2}-6\right) = \frac{1}{40}\not=0$.
On the other hand, Theorem~\ref{thm:indep-sets}(2) still holds for this orbit, e.g.~the average of $h\mapsto 2h(1)+h(2)$ is $\frac{2(8)+4}{20}=1$.  Indeed, this result can be extended to chain polytope toggles.

\begin{figure}[htb]
\begin{center}
\begin{tabular}{c||c|c|c|c|c|c|c|c}
&$I_1$&$I_2$&$I_3$&$I_4$&$I_5$&$I_6$&$I_7$&$I_8$\\\hline\hline\\[-1.111em]
$             g $
&0&0&0&0&$\frac12$&0&0&1
\\\hline\\[-1.111em]
$\varphi     (g)$
&1&0&1&0&$\frac12$&$\frac12$&0&0
\\\hline\\[-1.111em]
$\varphi^2   (g)$
&0&0&0&$\frac12$&0&$\frac12$&$\frac12$&$\frac12$
\\\hline\\[-1.111em]
$\varphi^3   (g)$
&1&0&$\frac12$&0&$\frac12$&0&0&$\frac12$
\\\hline\\[-1.111em]
$\varphi^4   (g)$
&0&$\frac12$&0&$\frac12$&0&1&0&$\frac12$
\\\hline\\[-1.111em]
$\varphi^5   (g)$
&$\frac12$&0&$\frac12$&0&0&0&$\frac12$&0
\\\hline\\[-1.111em]
$\varphi^6   (g)$
&$\frac12$&$\frac12$&0&1&0&$\frac12$&0&1
\\\hline\\[-1.111em]
$\varphi^7   (g)$
&0&$\frac12$&0&0&$\frac12$&0&0&0
\\\hline\\[-1.111em]
$\varphi^8   (g)$
&$\frac12$&0&1&0&$\frac12$&$\frac12$&$\frac12$&$\frac12$
\\\hline\\[-1.111em]
$\varphi^9   (g)$
&$\frac12$&0&0&$\frac12$&0&0&0&$\frac12$
\\\hline\\[-1.111em]
$\varphi^{10}(g)$
&$\frac12$&$\frac12$&$\frac12$&0&1&0&$\frac12$&0
\\\hline\\[-1.111em]
$\varphi^{11}(g)$
&0&0&$\frac12$&0&0&$\frac12$&0&1
\\\hline\\[-1.111em]
$\varphi^{12}(g)$
&1&0&$\frac12$&$\frac12$&$\frac12$&0&0&0
\\\hline\\[-1.111em]
$\varphi^{13}(g)$
&0&$\frac12$&0&0&$\frac12$&$\frac12$&$\frac12$&$\frac12$
\\\hline\\[-1.111em]
$\varphi^{14}(g)$
&$\frac12$&0&1&0&0&0&0&$\frac12$
\\\hline\\[-1.111em]
$\varphi^{15}(g)$
&$\frac12$&0&0&1&0&1&0&$\frac12$
\\\hline\\[-1.111em]
$\varphi^{16}(g)$
&$\frac12$&$\frac12$&0&0&0&0&$\frac12$&0
\\\hline\\[-1.111em]
$\varphi^{17}(g)$
&0&$\frac12$&$\frac12$&$\frac12$&$\frac12$&$\frac12$&0&1
\\\hline\\[-1.111em]
$\varphi^{18}(g)$
&$\frac12$&0&0&0&0&$\frac12$&0&0
\\\hline\\[-1.111em]
$\varphi^{19}(g)$
&$\frac12$&$\frac12$&$\frac12$&$\frac12$&$\frac12$&0&1&0
\\\hline\hline\\[-1.111em]
\textbf{Total} & \textbf{8} & \textbf{4} & $\frac{\textbf{13}}{\textbf{2}}$ & \textbf{5} & $\frac{\textbf{11}}{\textbf{2}}$ & \textbf{6} & \textbf{4} & \textbf{8}
\end{tabular}
\end{center}
\caption{An example orbit under the action
$\varphi=\tau_{a_8}\tau_{a_7}\tau_{a_6}\tau_{a_5}\tau_{a_4}\tau_{a_3}\tau_{a_2}\tau_{a_1}\in
\tog_\calc(\calz_8)$
on the chain polytope of $\calz_8$.
This orbit has size 20, so $\varphi^{20}(g)=g$.
This shows that Theorem~\ref{thm:indep-sets}(1) fails for the chain polytope toggles.}
\label{fig:counterexample}
\end{figure}

There is an issue with this notion, however.  By toggling within the
finite set $\cala(P)$ of a poset, all orbits are guaranteed to finite.
The chain polytope, on the other hand, is infinite, so there is no guarantee that orbits have finite order.
It is believed that for $n\geq 6$, there are infinite orbits under a
Coxeter element.\footnote{
Using results in~\cite{grinberg-roby}, one can prove that birational rowmotion on $\calz_n$ has finite order for $n\leq 5$.
This implies $\row_{OR}$ has finite order, and therefore
$\row_\calc=\orb^{-1} \circ \row_{OR} \circ \orb$ does too.
However, for $n=7$, birational rowmotion has infinite order~\cite[\S20]{grinberg-roby}, so $\row_{OR}$ may have infinite order.
All Coxeter elements in $\tog_\cala(P)$ are conjugate using
\cite[Lemma~5.1]{strikerwilliams}, so the order of toggling does not affect the order.  David Einstein and James Propp have made progress in proving
infinite order for $\row_{OR}$ for $\calz_n$ with $n\geq 6$, though details
are still being worked out.
}
The original definition of homomesy requires orbits to be finite.
Nonetheless, we can generalize to orbits that need not be finite where
the average value of the statistic $f$, as $w$ is iterated $N$ times, approaches a constant $c$.
This asymptotic generalization, first considered by Propp and Roby, has been used by Vorland for actions on order ideals of infinite posets.

\begin{defn}[{\cite[Definition 5.3.1]{vorland-thesis}}]
Suppose we have a set $\cals$, a map $w:\cals\ra \cals$, and a function (``statistic") $f:\cals\ra\rr$.  If there exists a real number $c$ such that for every $x\in \cals$,
$$
\lim\limits_{N\ra\infty} \frac{1}{N}\sum\limits_{i=0}^{N-1}f\big(w^i(x)\big)=c
$$
then we say that $f$ is \textbf{homomesic with average \textsl{c}} (or \textbf{\textsl{c}-mesic}) under the action of $w$ on $\cals$.
\end{defn}

Below is a generalization of Theorem~\ref{thm:indep-sets}(2) to
$\calc(\calz_n)$.
Despite the numerous homomesy results for finite sets,
this result is notable as one of the few known instances of asymptotic homomesy for orbits that are probably not always finite.

\begin{thm}\label{thm:indep-sets-C}
Let $n\geq 2$, and $\tau_{a_1}, \cdots, \tau_{a_n}$ be chain polytope toggles on $\calc(\calz_n)$.
\begin{enumerate}
\item Let $w\in\tog_\calc(\calz_n)$ be a composition of toggles in which $\tau_{a_1}$ appears exactly once,
$\tau_{a_2}$ at most once, and other toggles can appear any number of times (possibly none).  Then the statistic $g\mapsto 2g(a_1)+g(a_2)$ is
1-mesic under the action of $w$ on $\calc(\calz_n)$.
\item Let $w\in\tog_\calc(\calz_n)$ be a composition of toggles in which $\tau_{a_n}$ appears exactly once,
$\tau_{a_{n-1}}$ at most once, and other toggles can appear any number of times (possibly none).  Then the statistic $g\mapsto g(a_{n-1})+2g(a_n)$ is
1-mesic under the action of $w$ on $\calc(\calz_n)$.
\end{enumerate}
\end{thm}

\begin{proof}
We only prove (1) as the proof of (2) is analogous. Let $w$ be as in the theorem.  Let $g_0\in\calc(\calz_n)$ and define $g_i:=w^i(g_0)$.
Notice that $a_1\lessdot a_2$ is the only maximal chain containing $a_1$.
So, for any $g\in\calc(\calz_n)$,
\begin{equation}\label{eq:boulders}
(\tau_{a_1}g)(a_1)=1-g(a_1)-g(a_2).
\end{equation}
There are two cases.

\textbf{Case 1: }The toggle $\tau_{a_2}$ is performed either after $\tau_{a_1}$ or not at all while applying $w$.
Then, when applying $w$ to $g_i$, the labels of $a_1$ and $a_2$ are unchanged with the toggle $\tau_{a_1}$ is applied.  Thus, by Eq.~(\ref{eq:boulders}),
$(\tau_{a_1}g_i)(a_1)=1-g_i(a_1)-g_i(a_2)$.
Since $a_1$ is only toggled once in $w$, we have
\begin{equation}\label{eq:turtle woods}
g_{i+1}(a_1)=1-g_i(a_1)-g_i(a_2).
\end{equation}
Using this formula is the third equality below,
\begin{align*}
\lim\limits_{N\ra\infty}\frac{1}{N} \sum\limits_{i=0}^{N-1}
\big(2g_i(a_1)+g_i(a_2)\big)
&= \lim\limits_{N\ra\infty}\frac{1}{N} \sum\limits_{i=0}^{N-1}
\big(g_i(a_1)+ g_i(a_1) + g_i(a_2)\big)\\
&= \lim\limits_{N\ra\infty}\frac{1}{N} \left(
g_0(a_1) - g_N(a_1) + \sum\limits_{i=0}^{N-1}
\big(g_{i+1}(a_1)+ g_i(a_1) + g_i(a_2)\big)
\right)\\
&= \lim\limits_{N\ra\infty}\frac{1}{N} \left(
g_0(a_1) - g_N(a_1) + \sum\limits_{i=0}^{N-1} 1
\right)\\
&= \lim\limits_{N\ra\infty}\frac{1}{N} \big(
g_0(a_1) - g_N(a_1) + N
\big)\\
&= 1 + \lim\limits_{N\ra\infty}\frac{1}{N} \big(
g_0(a_1) - g_N(a_1) \big)\\
&= 1+0\\&=1.
\end{align*}
The limit calculation follows from the Squeeze Theorem since $g_N(a_1)$ is between 0 and 1 for all $N$ by the chain polytope's definition.

\textbf{Case 2: }The toggle $\tau_{a_2}$ is performed before $\tau_{a_1}$ while applying $w$.  Then, recall that $\tau_{a_2}$ only appears once in $w$.
So the label of $a_2$ after applying $\tau_{a_2}$ to $g_i$ is $g_{i+1}(a_2)$.
Then using Eq.~(\ref{eq:boulders}),
$(\tau_{a_1}g_i)(a_1)=1-g_i(a_1)-g_{i+1}(a_2)$.
Since $a_1$ is only toggled once in $w$, we have
\begin{equation}\label{eq:snow go}
g_{i+1}(a_1)=1-g_i(a_1)-g_{i+1}(a_2).
\end{equation}
Using this (with $i-1$ in place of $i$) in the third equality below,
\begin{align*}
\lim\limits_{N\ra\infty}\frac{1}{N} \sum\limits_{i=0}^{N-1}
\big(2g_i(a_1)+g_i(a_2)\big)
&= \lim\limits_{N\ra\infty}\frac{1}{N} \sum\limits_{i=0}^{N-1}
\big(g_i(a_1)+ g_i(a_1) + g_i(a_2)\big)\\
&= \lim\limits_{N\ra\infty}\frac{1}{N} \left(
g_{N-1}(a_1) - g_{-1}(a_1) + \sum\limits_{i=0}^{N-1}
\big(g_i(a_1)+ g_{i-1}(a_1) + g_i(a_2)\big)
\right)\\
&= \lim\limits_{N\ra\infty}\frac{1}{N} \left(
g_{N-1}(a_1) - g_{-1}(a_1) + \sum\limits_{i=0}^{N-1} 1
\right)\\
&= \lim\limits_{N\ra\infty}\frac{1}{N} \big(
g_{N-1}(a_1) - g_{-1}(a_1) + N
\big)\\
&= 1 + \lim\limits_{N\ra\infty}\frac{1}{N} \big(
g_{N-1}(a_1) - g_{-1}(a_1) \big)\\
&= 1.
\end{align*}
\end{proof}

\section{Future directions}\label{sec:future tense}
As mentioned before, Einstein and Propp (and others) have generalized the piecewise-linear toggles $t_e$ further to the birational setting, an idea they credit to Kirillov and Berenstein~\cite{berenstein}.
In that generalization, the poset labels are elements of a semifield (e.g.\ $\rr^+$), and in the definition of toggling,
0, addition, subtraction, max, and min are respectively replaced with 1, multiplication, division, addition, and a ``parallel summation.''  Any result that holds true using the semifield axioms (so no use of subtraction nor additive inverses) holds
for the piecewise-linear toggling by working over the tropical semiring discussed in~\cite{tropical-semiring}. This gives a fruitful technique for proving results about piecewise-linear toggling or even just combinatorial toggling.
The author believes that the toggles $\tau_e$ we have defined for $\calc(P)$ can similarly be generalized to the birational setting.  This will likely prove useful in studying $\tau_e$ on $\calc(P)$ or $\cala(P)$ and is the next direction in which the author has begun to collaborate with others for further research.

Also worth noting is that many of the results proved here are at the purely group-theoretic level.  For example, the proofs of
Theorem~\ref{thm:tau-star}, Lemma~\ref{lem:tau-star}, and every result in Subsection~\ref{subsec:gp} only rely on the algebraic properties proven previously. They use the properties that toggles are involutions, conditions for commutativity of toggles, and the homomorphism from $\tog_\calj(P)$ to $\tog_\cala(P)$ (proven later to be an isomorphism) given by $t_e\mapsto t_e^*$ of Theorem~\ref{thm:t-star}.  Due to this, the piecewise-linear analogues extend automatically after proving analogues of these algebraic conditions.

So one could abstract from $\tog_\cala(P)$ to a generic
group $G$ generated by involutions $\{\gamma_e\;|\;e\in P\}$ with relation $\gamma_x, \gamma_y$ that commute if $x$ and $y$ are incomparable (so $\gamma_e$ mimics $\tau_e$).
If one defines
$h_e:=\gamma_{e_1}\cdots \gamma_{e_k} \gamma_{e}\gamma_{e_1}\cdots \gamma_{e_k}$ for $e_1,\dots,e_k$ the elements $e$ covers (so $h_e$ mimics $t_e^*$)
and adds relations in $G$ that $h_x$ and $h_y$ commute when neither $x$ nor $y$ covers the other, then in $G$ one automatically obtains analogues of several results discussed here.
Exploring this generalization may prove useful, and it may even be possible
that this idea could be naturally extended from posets to a larger class of objects (such as directed graphs which generalize Hasse diagrams).

\section*{Acknowledgements}
The author thanks Jessica Striker for motivating the study of generalized toggle groups.
The author is also grateful for David Einstein, James Propp, and Tom Roby for many helpful conversations about dynamical algebraic combinatorics over the years.  Additionally, the author thanks the group he worked with on toggling noncrossing partitions, which includes the aforementioned Einstein and Propp, as well as Miriam Farber, Emily Gunawan, Matthew Macauley, and Simon Rubinstein-Salzedo.  The author began to discover the results of this paper after
exploring if the homomesy for toggling noncrossing partitions also holds for nonnesting partitions.
The author is also quite grateful for an anonymous referee whose multiple careful readings and numerous helpful comments have been of great assistance in improving the paper, and also for suggesting one of the mentioned directions for future research.

\bibliography{bibliography}
\bibliographystyle{halpha}

\end{document}